\documentclass[12pt]{amsart}

\title {Momentum polytopes of projective spherical varieties and related K\"ahler geometry}

\author{St{\'e}phanie Cupit-Foutou}
\address{Fakult{\"a}t f{\"u}r Mathematik, Ruhr-Universit{\"a}t Bochum}
\email{stephanie.cupit@rub.de}
  
\author{Guido Pezzini}
\address{Dipartimento di Matematica ``Guido Castelnuovo'', ``Sapienza'' Universit\`a di Roma}
\email{pezzini@mat.uniroma1.it}

\author{Bart Van Steirteghem}
\address{Department Mathematik, FAU Erlangen-N\"urnberg and Dept.\ of Mathematics, Medgar Evers College \& The Graduate Center, City University of New York}
\email{bartvs@math.fau.de}



\usepackage{fullpage}
\usepackage{amssymb, amsmath, amscd}
\usepackage{mathrsfs}
\usepackage[shortlabels]{enumitem}
\usepackage[colorlinks,breaklinks, allcolors=blue]{hyperref}
\usepackage{longtable}
\usepackage{soul}
\usepackage[toc,page]{appendix}

\newtheorem{theorem}{Theorem}[section]
\newtheorem{lemma}[theorem]{Lemma}
\newtheorem{proposition}[theorem]{Proposition}
\newtheorem{corollary}[theorem]{Corollary}

\theoremstyle{definition}
\newtheorem{definition}[theorem]{Definition}
\newtheorem{remark}[theorem]{Remark}
\newtheorem{example}[theorem]{Example}

\numberwithin{equation}{section}

\newcommand{\CC}{\mathbb C}
\newcommand{\ZZ}{\mathbb Z}
\newcommand{\NN}{\mathbb N}
\newcommand{\QQ}{\mathbb Q}
\newcommand{\QQp}{{\mathbb{Q}_{\geq 0}}}
\newcommand{\RR}{\mathbb R}

\newcommand{\GGm}{{\mathbb G_\text{m}}}
\newcommand{\A}{\mathbf A}

\newcommand{\PP}{\mathbb P}

\newcommand{\lb}{\mathcal{L}}
\newcommand{\F}{\mathcal{F}}

\newcommand{\inv}{^{-1}}

\newcommand{\GL}{\mathrm{GL}}
\newcommand{\SL}{\mathrm{SL}}

\newcommand{\Sp}{\mathrm{Sp}}

\newcommand{\Chi}{\mathcal X}
\newcommand{\Uni}{\mathrm{U}}

\newcommand{\fk}{\mathfrak{k}}

\newcommand{\ft}{\mathfrak{t}}

\newcommand{\sB}{\mathsf{B}}

\newcommand{\sG}{\mathsf{G}}

\DeclareMathOperator{\rk}{rk}

\newcommand{\divB}{\operatorname{div}^B}
\DeclareMathOperator{\Hom}{Hom}

\DeclareMathOperator{\Pic}{Pic}
\DeclareMathOperator{\Spec}{Spec}
\DeclareMathOperator{\Span}{span}
\DeclareMathOperator{\soc}{soc}

\DeclareMathOperator{\Conv}{Conv}

\newcommand{\<}{\langle}
\renewcommand{\>}{\rangle}

\newcommand{\wm}{\Gamma}

\newcommand{\dw}{\Lambda^+}
\newcommand{\wl}{\Lambda}

\newcommand{\sr}{S}
\newcommand{\lat}{\Xi}  %lattice of variety
\newcommand{\col}{\Delta}

\newcommand{\tG}{\widetilde{G}}

\newcommand{\twl}{\widetilde{\wl}}
\newcommand{\tdw}{\widetilde{\Lambda}^+}

\newcommand{\Oh}{\mathcal{O}}
\newcommand{\V}{\mathcal{V}}
\newcommand{\D}{\mathcal{D}}
\newcommand{\C}{\mathcal{C}}
\newcommand{\B}{\mathcal{B}}
\newcommand{\Pc}{\mathcal{P}}

\newcommand{\om}{\omega}
\newcommand{\wom}{\widetilde{\om}}

\newcommand{\msv}{\mathsf{v}}
\newcommand{\msw}{\mathsf{w}}
\newcommand{\msV}{\mathsf{V}}

\newcommand{\inn}{\subseteq}
\newcommand{\loccit}{{\em loc.cit.}}

\newcommand{\pperp}{\perp}
\newcommand{\Spp}{S^{\pperp}}

\begin{document}

\begin{abstract}
We apply the combinatorial theory of spherical varieties to characterize the momentum polytopes of polarized projective spherical varieties. This enables us to derive a classification of these varieties, without specifying the open orbit, as well as a classification of all Fano spherical varieties. In the setting of multiplicity free compact and connected Hamiltonian manifolds, we obtain a necessary and sufficient condition involving momentum polytopes for such manifolds to be K\"ahler and classify the invariant compatible complex structures of a given K\"ahler multiplicity free compact and connected Hamiltonian manifold.
 
\end{abstract}

\maketitle

\section{Introduction}
Let $(M,\omega,\Phi)$ be a compact connected Hamiltonian $K$-manifold acted on by a compact connected Lie group $K$. 
In \cite{kirwan}, F.~Kirwan proved that the image of $M$ through the momentum map $\Phi$ intersects the positive Weyl chamber $\ft^*_+$ in a convex polytope. This polytope is called the \emph{Kirwan polytope} of $(M,\omega,\Phi)$ and we will denote it by $\Pc(M,\omega,\Phi)$. 

The Hamiltonian $K$-manifold $(M,\omega,\Phi)$ is called \emph{multiplicity free} if its momentum map induces a homeomorphism $M/K \to \Pc(M,\omega,\Phi)$.
Toric manifolds and coadjoint orbits are well known examples of multiplicity free compact and connected manifolds.

Under the assumption that $M$ is a toric manifold, T.~Delzant constructively proved in \cite{delzant} that $\Pc(M,\omega,\Phi)$ uniquely determines $(M,\omega,\Phi)$ and described which polytopes can be realized this way. 
These polytopes are now known as \emph{Delzant polytopes}. 
Delzant's construction also shows that any toric manifold has an essentially unique $K$-invariant complex structure compatible with its symplectic form (in other words, it is K\"ahler) and that the $K$-action can be complexified to a holomorphic $K^\mathbb C$-action making the manifold into a smooth projective toric variety.

In~\cite{knop:hamilton}, F.~Knop proved Delzant's conjecture asserting that any multiplicity free compact and connected Hamiltonian $K$-manifold $(M,\omega,\Phi)$ is uniquely determined by the pair $(\lat(M),\Pc(M,\omega,\Phi))$, where $\lat(M)$ is a certain sublattice of the weight lattice of $K$ determined by the generic isotropy group of the $K$-action on $M$. 
F. Knop also determined which pairs $(\lat,\Pc)$, where  $\lat$ is a sublattice of the weight lattice of $K$ and $\Pc$ is a polytope in $\ft^*_{+}$, are realized by a compact connected multiplicity free $K$-manifold.

Like toric manifolds, coadjoint orbits of compact groups are K\"ahler but in general, when $K$ is nonabelian,
the existence of an invariant compatible complex structure no longer holds. Indeed, C. Woodward showed in \cite{woodward-notkaehler} that there exist multiplicity free manifolds that are not K\"ahlerizable. 
It is also known that there are multiplicity free manifolds that allow more than one compatible complex structure (see e.g. \cite[Remark 4.4]{woodward}).

As proved by A. Huckleberry and T. Wurzbacher in~\cite{huckwurz}, K\"ahler multiplicity free $K$-manifolds are spherical $K^\mathbb C$-varieties (i.e. contain a dense orbit of a Borel subgroup of $K^\mathbb{C}$).
Using this fact and applying the Uniqueness Theorem on smooth affine spherical varieties he obtained in~\cite{losev:knopconj},
I. Losev  proved in \loccit\ a refined uniqueness result for invariant K\"ahler structures on a given compact connected multiplicity free Hamiltonian $K$-manifold.

One of the results in this paper is a combinatorial answer to the K\"ahlerizability question, which was first studied in \cite{woodward}. 
Considering Brion's result \cite{brion:image} asserting that a smooth projective $K^\CC$-variety is multiplicity free as a $K$-Hamiltonian manifold if and only if it is spherical as a $K^\CC$-variety, we approach this question by making use of the theory of projective spherical varieties, just like Woodward did in \cite{woodward}. 
In fact, most of the current paper's results are about these varieties. 
In particular, we obtain a combinatorial classification of the (not necessarily smooth) polarized spherical varieties involving rational convex polytopes. This classification parallels the well known classification of polarized toric varieties in terms of integral convex polytopes. Moreover, it enables us to derive a classification of K\"ahler multiplicity free compact and connected Hamiltonian manifolds, which generalizes Delzant's classification of toric manifolds. While Knop's criterion for a pair $(\lat,\Pc)$ to be realized by a multiplicity free $K$-manifold $M$ is purely local, our K\"ahlerizability result involves the global structure of the momentum polytope.

One important ingredient in this work is the \emph{momentum polytope} $Q(X,\lb)$ of a polarized $K^\mathbb C$-variety $(X,\lb)$ introduced by M. Brion in \cite{brion:image}.
This polytope turns out to be a purely algebraic version of the Kirwan polytope. 
In Section~\ref{sec:basics}, we recall Brion's representation-theoretic definition of a momentum polytope and its main properties.
In Section~\ref{sec:brionbasics}, we focus on polarized spherical varieties; we collect some known facts on the combinatorial invariants of these varieties and their intertwining relations with momentum polytopes.

Section~\ref{sec:characterizations_mp} is the heart of this paper.
Here we use the combinatorial theory of spherical varieties, including their recent complete combinatorial classification, to characterize the momentum polytopes of polarized spherical varieties. 
More specifically, we give two combinatorial characterizations of the pairs $(\lat,Q)$, where $\lat$ is a sublattice of the weight lattice of $K^\mathbb C$ and $Q$ is a rational polytope in $\ft^*_{+}$, for which there exists a polarized spherical $K^\mathbb C$-variety $(X,\lb)$ whose weight lattice $\lat(X)$ equals $\lat$ (see Subsection~\ref{subsec:polarized_varieties} for the definition of $\lat(X)$) and momentum polytope is $Q$.
In the first characterization, Corollary~\ref{cor:crit-convexhull}, we view $Q$ as a convex hull, while in the second one, Theorem~\ref{thm:compatible-polytope}, we view it as an intersection of finitely many half-spaces. 
Both characterizations we obtain involve the set $\Sigma(X)$ of spherical roots of $X$, which is an important (finite) invariant of $X$ that is absent in the toric setting (see the precise definition in Subsection~\ref{section:coloredfan}). In the case of so-called $\QQ$-polarized spherical varieties, the second characterization simplifies considerably, and only involves $\lat$ and $Q$ (see Remark~\ref{rem:subsets_of_sigma}\ref{rem:subsets_of_sigma_a}). 

In Section~\ref{sec:classifications}, we make use of our characterizations to deduce some classifications. We classify in terms of the data $(\lat(X),Q(X,\lb),\Sigma(X))$ all polarized spherical $G$-varieties (Theorem~\ref{thm:classification-triples}), and all Fano spherical varieties (Theorem~\ref{thm:Fano}).
In particular, we recover results of Pasquier~\cite{pasquier08} and of Hofscheier and Gagliardi~\cite{hof-gag} about the latter, but in contrast to this earlier work, we do not specify the open orbit (see also Remark~\ref{rmk:Fano}).

In Section~\ref{sec:smoothness}, we adapt a combinatorial smoothness criterion for spherical varieties due to R. Camus~\cite{camus} to our setting: it is a (significantly more involved) generalization of the well-known combinatorial smoothness criterion for toric varieties to the nonabelian case, and it allows us to derive, in Corollary~\ref{cor:classif_smooth_polarized}, a classification of smooth polarized spherical varieties $(X,\lb)$ in terms of the data $(\lat(X), Q(X,\lb),\Sigma(X))$.

Finally, in Section~\ref{sec:Kaehler}, we consider multiplicity free compact and connected Hamiltonian $K$-manifolds.
In Theorem~\ref{thm:criterion-Kaehler}, we obtain 
our combinatorial criterion for the K\"ahlerizability of a multiplicity free compact and connected Hamiltonian $K$-manifold $(M,\omega,\Phi)$ in terms of so-called smooth $\RR$-momentum triples involving $(\lat(M),\Pc(M,\omega,\Phi))$ (see Definitions~\ref{def:R-admissibleset}-
\ref{def:smooth_R_momentum_triple}).
These triples are the real analogues of the algebraic triples mentioned above and they can be seen as a generalization of Delzant polytopes. Furthermore, they enable us to classify the $K$-invariant compatible complex structures of a K\"ahler multiplicity free compact and connected Hamiltonian $K$-manifold (Corollary~\ref{cor:classif_of_complex}). 
We also explain how our criterion generalizes earlier K\"ahlerizability results due to Delzant and Woodward (Section~\ref{subsec:Kaehler_Delzant_Woodward}).
It may be worth noticing that the proof of the K\"ahlerizability criterion we provide does not use any of the results obtained in Section~\ref{sec:characterizations_mp} and even gives an alternative way of showing Theorem~\ref{thm:compatible-polytope} about 
the characterization of momentum polytopes as intersections of finitely many half-spaces.

\subsection*{Notation}
All groups and varieties are defined over the field of complex numbers $\CC$ and varieties are irreducible by assumption. 
If $H$ is a linear algebraic group, we denote by $\Chi(H)$ its set of characters, and by $H^\circ$ the identity connected component of $H$.

Throughout the paper, $G$ is a connected reductive algebraic group. We fix a choice of a Borel subgroup $B\subseteq G$ and a maximal torus $T\inn B$.
We denote by $\sr$ the set of simple roots of $G$, we use $\wl$ for the weight lattice of $G$, that is $\wl=\Chi(T)$.  
If $\alpha$ is a root of $(G,T)$, then we write $\alpha^{\vee} \in \Hom_{\ZZ}(\wl,\ZZ)$ for the corresponding coroot.
We denote by $\dw$ the set of dominant weights in $\wl$.
If $\lambda\in\dw$,  then the irreducible $G$-module of highest weight $\lambda$ is denoted by $V(\lambda)$ and $v_\lambda$ denotes a highest weight vector of $V(\lambda)$; the highest weight of the dual module $V(\lambda)^*$ will be denoted $\lambda^*$.
If we need to specify the acting group, we will add a subscript and write $V_G(\lambda)$. When $V$ is a finite dimensional $G$-module, we will use $\mathbb{P}(V)$ for the associated projective space (of lines in $V$).

If $\Xi$ is a lattice (i.e.\ a finitely generated free abelian group) then $\Xi^* = \Hom_{\ZZ}(\Xi,\ZZ)$ is the dual lattice. Suppose $\mathbb{F}$ is the field $\QQ$ or $\RR$.  We put $\Xi_{\mathbb{F}}:=\Xi\otimes_\ZZ\mathbb{F}$, and if $\Omega$ is a subset of $\Xi_{\mathbb{F}}$, we denote by $\mathbb{F}_{\ge 0}\Omega$ the convex cone generated by $\Omega$ in $\Xi_{\mathbb{F}}$.
We will use $\<\cdot,\cdot\>$ for the pairing between $\Xi_{\mathbb{F}}$ and $\Hom_{\mathbb{F}}(\Xi_{\mathbb{F}},\mathbb{F})$. When $\Omega$ is a subset of $\lat_{\mathbb{F}}$, we will use $\Omega^{\vee}$ for the dual cone: $\Omega^{\vee} = \{\xi \in \Hom_{\mathbb{F}}(\Xi_{\mathbb{F}},\mathbb{F}) \colon \<\xi, v\> \ge 0 \text{ for all $v \in \Omega$}\}$. 
For $\Pi\inn\Hom_{\mathbb{F}}(\Xi_{\mathbb{F}},\mathbb{F})$, $\Pi^\perp$ stands for the set of $\lambda\in\Xi_{\mathbb{F}}$ such that $\<\lambda,\sigma\> = 0$ for all $\sigma \in \Pi$. 
Given a convex rational polyhedral cone $\C\subseteq \Hom_\QQ(\Xi_\QQ,\QQ)$ and an extremal ray $R$ of $\C$, the intersection $R\cap \Hom_\ZZ(\Xi,\ZZ)$ is a monoid generated by a unique element, that we call a {\em ray generator} of $\C$. For any subset $\Omega$ of $\wl_{\RR}$ we put $\Spp(\Omega):= \{\alpha \in S \colon \<\alpha^{\vee},x\> = 0 \text{ for every } x \in \Omega\}$.

\section{Basic material on reductive group actions}\label{sec:basics} 
\subsection{Polarized varieties} \label{subsec:polarized_varieties}
Let $X$ be a $G$-variety, that is, a variety equipped with an action of $G$. Then $G$ acts on the function field  
of $X$ as well as on the ring 
of regular functions on $X$. The \textbf{weight lattice} $\lat(X)$ (resp. \textbf{weight monoid} $\wm(X)$)  of $X$ is the set of $B$-weights of the rational (resp. regular) $B$-eigenfunctions on $X$. Note that if $X$ is affine then $\wm(X)$ generates the group $\lat(X)$. By the \textbf{rank} of $X$ we mean the rank of $\lat(X)$. 

Let $(X,\lb)$ be a \textbf{polarized $G$-variety} that is, a projective $G$-variety $X$ equipped with an ample $G$-linearized line bundle $\lb$. Note that once the $G$-linearization of $\lb$ is fixed, the space of global sections $H^0(X,\lb)$ carries a natural $G$-module structure. 

Following Brion~\cite{brion:image}, we define the \textbf{momentum polytope of $(X,\lb)$} as
\begin{equation} \label{eq:def_momentum_pol_X}
Q(X,\lb):=\Big\{\frac\chi n\in\wl_\QQ: H^0(X,\lb^{n})^{(B)}_\chi\neq \{0\}\Big\},
\end{equation}
where $H^0(X,\lb^{n})^{(B)}_\chi$ is the $B$-eigenspace of weight $\chi$  of $H^0(X,\lb^{n})$. Then $Q(X,\lb)$ is indeed the convex hull of finitely many points in $\wl_{\QQ}$, cf.\ \cite[Proposition 1.2.3]{brion:notes}. 

As mentioned earlier, if $X$ is smooth then $Q(X,\lb)$ can be interpreted as the set of rational points of a Kirwan polytope. 
Specifically, suppose first that $\lb$ is very ample and fix a $K$-invariant hermitian form $\left\langle\cdot,\cdot\right\rangle$ on $V=H^0(X,\lb)^*$ with $K$ being a maximal compact subgroup of $G$.
We then consider the Fubini-Study symplectic form $\omega_\lb$ on $X\hookrightarrow \mathbb P(V)$ and the momentum map $\Phi_\lb$ of $(X,\omega_\lb)$ defined by
$$
\Phi_\lb(x)(\xi)=\frac{i}{2\pi}\left\langle\xi \tilde x,\tilde x\right\rangle 
\quad \mbox{ for }x\in X \mbox{ and } \xi\in\mathfrak k
$$
where $\tilde{x}$ is a unit vector in $V$ representing $x$.
If $\lb$ is only ample, we consider  a positive $m$ such that $\lb^m$ is very ample, the symplectic form 
$\frac{1}{m} \omega_{\lb^{m}}$ and the momentum map $\frac{1}{m}\Phi_{\lb^{m}}$ that we denote by $\Phi_\lb$ (since it does not depend on $m$).
We thus have the following proposition, which goes back to \cite[Appendix]{ness-nullcone} and \cite{guill&stern-conv}, see also \cite[Proposition 2.2]{brion:image}. 

\begin{proposition} \label{prop:brion-Kirwan}
Let $(X,\lb)$ be a smooth polarized $G$-variety.
The polytope  $Q(X,\lb)$ is the set of rational points of the intersection of $\Phi_\lb(X)$ with the positive Weyl chamber determined by the Borel subgroup $B$ of $G$.
\end{proposition} 

Observe that it follows from the definition of $Q(X,\lb)$ that 
\begin{equation} \label{eq:QXLminusp}
Q(X,\lb)-p\quad \text{ is a spanning subset of } \quad \lat(X)_\QQ \quad \text{for every $p \in Q(X,\lb)$}. 
\end{equation}
Next, we denote the total section ring of $(X,\lb)$ by $R(X,\lb)$, that is, 
\[
R(X,\lb) := \bigoplus_{n\geq 0} H^0(X,\lb^{n}).
\]
Recall that $R(X,\lb)$ is a finitely generated graded algebra, with degree $n$ part equal to $H^0(X,\lb^{n})$.
Moreover,  $R(X,\lb)$ is endowed with an action of 
\[\widetilde G := G \times \GGm.\] Indeed, the group $G$ acts naturally on each $H^0(X,\lb^{n})$ while
$\GGm$ acts with weight $n$ on $H^0(X,\lb^{n})$. We fix the Borel subgroup $B \times \GGm$ and the maximal torus $T \times \GGm$ of $\tG$. The associated weight lattice (resp. set of dominant weights) is $\widetilde{\wl}:=\wl \times \ZZ$ (resp. $\widetilde{\wl}^+:= \dw \times \ZZ$).  

The affine $\widetilde G$-variety
\[
\widetilde X:=\Spec R(X,\lb)
\]
is called the \textbf{affine cone} over $X$. It is normal if and only if $X$ is (see, e.g., \cite[Ex. II.5.14(a)]{hartshorne} for the ``if'' part). Its weight lattice $\lat (\widetilde{X})$ will also be denoted $\widetilde{\lat}(X,\lb)$ and will be called the \textbf{extended weight lattice} of $(X,\lb)$. We have an exact sequence
\begin{equation}\label{eq:latticecone}
0 \to \lat(X) \to \widetilde{\lat}(X,\lb) \to \ZZ \to 0
\end{equation}
where the third map is induced by the projection $\twl=\wl_{\widetilde G}\to \wl_\GGm$. Note that the cone in $\wl_{\QQ} \times \QQ$ over $Q(X,\lb) \times \{1\}$  coincides with the cone generated by the weight monoid $\wm(\widetilde{X})$ of the $\widetilde G$-variety $\widetilde X$. Equivalently, we have
\begin{equation}\label{eq:conepolytope}
Q(X,\lb)\times \{1\} = (\QQp\wm(\widetilde X)) \cap (\wl_\QQ \times \{1\})
\end{equation}
as subsets of $\wl_\QQ\times \QQ$ (compare \cite[Section~2.2]{alexeev-brion:SSV}).

\subsection{Polarized spherical varieties}
We call a $G$-module \textbf{multiplicity free} if it is the direct sum of pairwise non-isomorphic simple $G$-modules. An affine $G$-variety is {\em multiplicity free} if its coordinate ring is a multiplicity free $G$-module. A polarized $G$-variety $(X,\lb)$ is 
{\em multiplicity free} if every $G$-module $H^0(X,\lb^{n})$ is multiplicity free; equivalently, $\widetilde X$ is a multiplicity free $\widetilde G$-variety. 
Thanks to a result of~\cite{VK}, an affine $G$-variety is spherical if and only if it is normal and multiplicity free. Likewise, a polarized $G$-variety is spherical if and only if it is normal and multiplicity free. The normality of the affine cone $\widetilde{X}$ over a polarized spherical $G$-variety $(X,\lb)$ implies the following equality for its weight monoid (as a $\widetilde{G}$-variety):
\begin{equation}
\wm(\widetilde{X}) = \QQ_{\ge 0}(Q(X,\lb) \times \{1\}) \cap \widetilde{\lat}(X,\lb). \label{eq:wm_affine_cone}
\end{equation}
Equivalently, for every $n\ge 0$ we have the following isomorphism of $G$-modules: 
\begin{equation}
H^0(X,\lb^n) \cong \bigoplus_{\lambda} V(\lambda),
\end{equation}
where the direct sum is over those $\lambda \in \dw$ such that  $(\lambda,n) \in \widetilde{\lat}(X,\lb)  \cap n(Q \times\{1\})$.

In Proposition~\ref{prop:PSVFundFacts} we gather some well-known basic facts about polarized spherical varieties that we will need.

\begin{proposition} \label{prop:PSVFundFacts}
Let $(X,\lb)$ be a polarized spherical $G$-variety with extended weight lattice $\widetilde{\lat} = \widetilde{\lat}(X,L)$ and momentum polytope $Q = Q(X,\lb)$ and let $Y$ be a closed $G$-subvariety of $X$. Then
\begin{enumerate}[(a)]
\item $(Y,\lb|_{Y})$ is a polarized spherical $G$-variety; \label{prop:PSVFundFacts:item:sub} 
\item $Q(Y,\lb|_{Y})$ is a face of $Q(X,\lb)$ and if $Y'$ is another closed $G$-subvariety of $X$ then $Q(Y,\lb|_{Y}) \neq Q(Y',\lb|_{Y'})$; \label{prop:PSVFundFacts:item:faces} 
\item the restriction map $H^0(X,\lb^n) \to H^0(Y,\lb^n|_Y)$ is surjective for all $n \geq 0$; \label{prop:PSVFundFacts:item:surj} 
\item if $Y$ is a closed $G$-orbit, then $Q(Y, \lb|_{Y})$ is a vertex of $Q(X,\lb)$, and there exists $\lambda_Y \in \dw$ such that $Q(Y, \lb|_{Y}) = \{\lambda_Y\}$  and $(\lambda_{Y},1) \in \widetilde{\lat}(X,\lb)$. \label{prop:PSVFundFacts:item:mompolclosedorbit} 
\end{enumerate}
\end{proposition}
\begin{proof}
A proof of assertions \ref{prop:PSVFundFacts:item:sub}, \ref{prop:PSVFundFacts:item:faces}, and  \ref{prop:PSVFundFacts:item:surj} is given, for example, in \cite[Prop. 2.9]{alexeev-brion:SSV}. For completeness, we also give the short proof of \ref{prop:PSVFundFacts:item:mompolclosedorbit}. Since $X$ is complete, it follows from the Borel-Weil theorem that $H^0(Y,\lb|_Y)$ is a simple $G$-module. Let $\lambda_Y$ be its highest weight. 
Then $(\lambda_{Y},1) \in \widetilde{\lat}(X,\lb)$ by \ref{prop:PSVFundFacts:item:surj}. 
Furthermore, $H^0(Y,\lb^n|_Y) = V(n\lambda_Y)$ for every $n\ge 0$, and so $Q(Y,\lb|_Y) = \{\lambda_Y\}$. 
That $\lambda_Y$ is a vertex of $Q(X,\lb)$ now follows from \ref{prop:PSVFundFacts:item:faces}.   
\end{proof}

\begin{corollary} \label{cor:tildelattice}
Let $(X,\lb)$ be a polarized spherical $G$-variety, let $Y$ be a closed $G$-orbit in $X$, and let $\lambda_Y$ be the dominant weight such that $Q(Y,\lb|_Y) = \{\lambda_Y\}$. Then
\begin{equation}
\widetilde{\lat}(X,\lb) = (\lat(X) \times \{0\}) \oplus \ZZ(\lambda_Y,1).
\end{equation}
\end{corollary}
\begin{proof}
The inclusion ``$\supseteq$'' follows from the exact sequence \eqref{eq:latticecone} and Proposition~\ref{prop:PSVFundFacts}-\ref{prop:PSVFundFacts:item:mompolclosedorbit}. 
To prove the reverse inclusion, take $\tilde\mu=(\mu,n)\in\widetilde{\lat}(X,\lb)$. We then have $\tilde\mu- n(\lambda_Y,1)\in \lat(X) \times \{0\}$.
The corollary follows.
\end{proof}

\section{Momentum polytopes and colored fans of polarized spherical varieties}
\label{sec:brionbasics}
In this section, we recall from \cite{brion:image} and \cite{brion:notes} how the momentum polytope of  polarized spherical $G$-variety $(X,\lb)$ is related to the colored fan of $X$. 

\subsection{The momentum polytope as an intersection of half-spaces}  \label{subsec:mom_pol_inters_halfspaces}
For a moment, let $X$ be any spherical $G$-variety and $\lat(X)$ be its weight lattice.
Set
\[
N(X):=\mathrm{Hom}_\ZZ(\lat(X),\QQ).
\]
Each discrete $\QQ$-valuation $v$ on the  field of rational functions $\mathbb C(X)$ determines an element $\rho_v$ of $N(X)$ given as follows
\[ 
\rho_v:  \lat(X)  \longrightarrow \mathbb Q,\quad \chi\longmapsto v(f_\chi)
\]
with $f_\chi\in\mathbb C(X)^{(B)}$ of weight $\chi$. Note that $f_\chi$ is uniquely determined by $\chi$ up to a scalar in $\mathbb C^{\times}$ since $X$ is spherical.

Let $\divB(X)$ be the set of $B$-stable prime divisors of $X$. Note that $\divB(X)$ is a finite set since it is the set of irreducible components of codimension 1 of some closed set, namely the complement of the open $B$-orbit on $X$. 
Each $D\in \divB(X)$ defines a valuation $v_D$ on $\CC(X)$ given by the order of vanishing along $D$ and, in turn, an element of $N(X)$ defined as 
\[
\rho_X(D):=\rho_{v_D}.
\] 
When there is no risk of confusion, we will also write $\rho(D)$ instead of $\rho_X(D)$. 

Now, suppose once again that $(X,\lb)$ is a polarized spherical $G$-variety.  Consider any nonzero section $s\in H^0(X,\lb)^{(B)}$. Let $\chi(s)$ denote its weight and $v_D(s)$ be its order of vanishing along $D \in \divB(X)$. Recall that any element of $H^0(X,\lb^{n})^{(B)}$ can be written as $s^n f$, for some $f\in\CC(X)^{(B)}$ such that $n v_D(s)+v_D(f)\geq 0$ for all $D\in\divB(X)$. 
Together with the finiteness of $\divB(X)$ this yields the following realization of $Q(X,\lb)$ as a finite intersection of closed half-spaces.

\begin{proposition}[{\cite[Proposition~5.3.1]{brion:notes}}]\label{prop:Pbrion}
Let $(X,\lb)$ be a polarized spherical $G$-variety, and $s \in H^0(X,\lb)^{(B)} \setminus \{0\}$. 
Then:
\[
Q(X,\lb)=\chi(s)+\{\xi\in \Xi(X)_\QQ: \<\rho(D),\xi\> + v_D(s) \geq 0 \, ,\forall D\in\divB(X)\}.
\]
\end{proposition}

\subsection{Colored fans, \texorpdfstring{$G$}{G}-invariant valuations and spherical roots}\label{section:coloredfan}
We recall the definition of some standard combinatorial invariants of a spherical $G$-variety $X$. The set of \textbf{colors} of $X$ is
\[\col(X) := \{D \in \divB(X) \colon \text{$D$ is not $G$-stable}\}.\] 
As proved by Luna and Vust \cite[Prop. 7.4]{lunavust}, the restriction of the map $v\mapsto \rho_v$ to the set of $G$-invariant valuations of $\mathbb C(X)$ is injective.
Let $\V(X)$ denote the image of  this restricted map; by \cite[\S 3]{brion90} it is a convex co-simplicial cone, \textbf{the valuation cone} of $X$. 
The set $\Sigma(X)$ (denoted also  by $\Sigma_G(X)$ when we need to emphasize the acting group) of \textbf{spherical roots} of $X$ is the minimal set of primitive elements  of $\lat(X)$ such that
\[
\V(X)=\left\{ \rho\in N(X): \left\langle \rho,\sigma\right\rangle\leq 0, \forall \sigma \in \Sigma(X)\right\}.
\]
 
The Luna-Vust embedding theory in \cite{lunavust} (see also \cite{knop:LV}) combinatorially describes $X$ as an equivariant embedding of its open $G$-orbit by means of its colored fan $\F(X)$. To define $\F(X)$, first let $Y$ be a $G$-orbit in $X$, and set \begin{align*}
\D_Y &:= \{D \in \col(X) \colon Y \subseteq D\};\\
\C_Y &:= \QQ_{\ge 0}\{\rho(D) \in N(X) \colon D \in \divB(X) \text{ with }Y \subseteq D\}.
\end{align*} 
The couples $(\C_Y,\D_Y)$ are called \textbf{colored cones}, and the \textbf{colored fan} of $X$ is 
\[ \F(X) := \{(\C_Y,\D_Y) \colon Y \text{ is a $G$-orbit in $X$}\}. \]

The colored fan $\F(X)$ satisfies the following properties (see \cite[Section~3]{knop:LV}):
\begin{itemize}
\item[(CC1)] For all $(\C,\D)\in\F(X)$ we have that $\C$ is a convex cone generated by $\{\rho(D)\;|\; D\in \D\}$ and finitely many elements of $\V(X)$;
\item[(CC2)] the relative interior of $\C$ intersects $\V(X)$;
\item[(SCC)] $\C$ is strictly convex, and $\rho(D)\neq 0$ for all $D\in\D$;
\item[(CF1)] every face of $(\C,\D)$ belongs to $\F(X)$, where $(\C',\D')$ is a \textbf{face} of $(\C,\D)$ if $\C'$ is a face of $\C$ and $\D'$ is the set of elements $D\in\D$ such that $\rho(D)\in\C'$.
\item[(CF2)] For every $v\in\V(X)$ there is at most one $(\C,\D)\in\F(X)$ such that $v$ is in the relative interior of $\C$.
\end{itemize}

Conversely, if a set $\F$ of couples $(\C,\D)$ where $\C\subseteq N(X)$ and $\D\subseteq \Delta(X)$ satisfies the above properties (CC1)-(CF2), then there exists a unique spherical variety $Z$, birationally $G$-isomorphic to $X$, such that $\F(Z)=\F$ (see \cite[Theorem~3.3]{knop:LV}).

For future reference, we recall the two following facts.

\begin{proposition}[{\cite[Lemma~3.2]{knop:LV}}]\label{thm:orbits-vs-faces}
Let $X$ be a spherical $G$-variety and $Y$ a $G$-orbit of $X$.
There exists a bijection between the set of $G$-orbits of $X$ whose closure contains $Y$ and the set of faces of $(\C_Y,\D_Y)$.
\end{proposition}

\begin{proposition}\label{prop:affine-criterion}
Let $X$ be an affine spherical $G$-variety with weight monoid $\Gamma$. Let $\C$ be the largest face of the dual cone $(\QQ_{\geq 0}\Gamma)^\vee \inn N(X)$ of which the relative interior intersects $\V(X)$ and put $\D = \{D \in \col(X) \colon \rho(D) \in \C\}$.
Then $(\C,\D)$ is an element of the colored fan $\F(X)$ of $X$, and for every other $(\C',\D') \in \F(X)$, the cone $\C'$ is a proper face of $\C$. 
\end{proposition}

\begin{proof}
See e.g. Proposition 5.14 in~\cite{ACF} and its proof.
\end{proof}

We will say that a color $D$ of the spherical variety $X$ is \textbf{moved} by the simple root $\alpha \in \sr$ if $P_{\alpha}\cdot D \neq D$, where $P_{\alpha}$ is the minimal parabolic subgroup strictly containing $B$ and associated to the simple root $\alpha$. 
We recall from \cite[\S 1.4 and Proposition 3.2]{luna:typeA} that if $D$ is a color moved by $\alpha \in \sr\setminus\Sigma(X)$, then:
\begin{align}
&\rho(D)=    \label{eq:color_functional}
\begin{cases}
\alpha^{\vee}|_{\lat(X)} &\text{if }2 \alpha\notin \Sigma(X);\\
\frac{1}{2}\alpha^{\vee}|_{\lat(X)} &\text{if }2 \alpha\in \Sigma(X);
\end{cases} \quad  \text{ and }\\
& \text{$\beta \in \sr\setminus\{\alpha\}$ also moves $D$ if and only if $\alpha \perp \beta$ and $\alpha+\beta$ or $\frac{1}{2}(\alpha+\beta)$ is in $\Sigma(X)$.} \label{eq:color_moved_by_two}
\end{align}

If we fix $G$, the union of the sets of spherical roots of all spherical $G$-varieties is finite, and denoted  by $\Sigma(G)$. An element $\sigma\in \Sigma(G)$ is always a linear combination of simple roots of $G$ with non-negative rational coefficients; the set of simple roots whose coefficient is non-zero is called the {\em support} of $\sigma$. We list all possible $\sigma$, with their supports, in Table~\ref{tab:sr}. In the table, simple roots are denoted by $\alpha_1,\alpha_2,\ldots$, numbered as in Bourbaki~\cite{bbki}. In two cases the support is of type $\mathsf A_1\times \mathsf A_1$, in which case the two simple roots are $\alpha_1$ and $\alpha_1'$.

\begin{table}\caption{Spherical roots}
\begin{center}\label{tab:sr}
\begin{tabular}{c|c}
type of support & spherical root\\
\hline
$\mathsf A_n$, $n\geq 1$ & $\alpha_1+\ldots+\alpha_n$\\
$\mathsf A_1$ & $2\alpha_1$\\
$\mathsf A_1 \times \mathsf A_1$ & $\alpha_1+\alpha_1'$\\
$\mathsf A_1 \times \mathsf A_1$ & $\frac12(\alpha_1+\alpha_1')$\\
$\mathsf A_3$ & $\alpha_1+2\alpha_2+\alpha_3$ \\
$\mathsf A_3$ & $\frac12\alpha_1+\alpha_2+\frac12\alpha_3$ \\
$\mathsf B_n$, $n\geq 2$ & $\alpha_1+\ldots+\alpha_n$\\
$\mathsf B_n$, $n\geq 2$ & $2(\alpha_1+\ldots+\alpha_n)$\\
$\mathsf B_3$ & $\alpha_1+2\alpha_2+3\alpha_3$ \\
$\mathsf B_3$ & $\frac12(\alpha_1+2\alpha_2+3\alpha_3)$ \\
$\mathsf C_n$, $n\geq 3$ & $\alpha_1+2(\alpha_2+\ldots+\alpha_{n-1})+\alpha_n$\\
$\mathsf D_n$, $n\geq 4$ & $2(\alpha_1+\alpha_2+\ldots+\alpha_{n-2})+\alpha_{n-1}+\alpha_n$\\
$\mathsf D_n$, $n\geq 4$ & $\alpha_1+\alpha_2+\ldots+\alpha_{n-2}+\frac12(\alpha_{n-1}+\alpha_n)$\\
$\mathsf F_4$ & $\alpha_1+2\alpha_2+3\alpha_3+2\alpha_4$ \\
$\mathsf G_2$ & $\alpha_1+\alpha_2$ \\
$\mathsf G_2$ & $2\alpha_1+\alpha_2$ \\
$\mathsf G_2$ & $4\alpha_1+2\alpha_2$ \\
\end{tabular}
\end{center}
\end{table}

\subsection{Momentum polytopes and colored fans}
The next definition recalls a basic construction in convex geometry. 

\begin{definition} \label{def:normal_fan}
Let $\lat$ be a lattice
and let $Q$ be a full dimensional convex polytope in $\lat_\QQ$.
The \textbf{normal fan} of $Q$  is the set
\[\F(Q) := \{\C(F) \colon \text{$F$ is a face of $Q$}\},\]
where for every face $F$ of $Q$, we define $\C(F)$ to be the dual cone in $N = \Hom_{\QQ}(\lat_\QQ,\QQ)$ of $\QQ_{\ge 0}(Q-p)$ where $p$ is any point in the relative interior of $F$, that is
\[\C(F) := \{ \eta \in N \colon \<\eta,v\> \ge 0 \text{ for all $v \in (Q-p)$}\}.\]
\end{definition}

For later use, we introduce the following notion. This terminology is justified by assertion~\ref{thm:Brion-Woodward:orbitface} in Theorem~\ref{thm:Brion-Woodward}.

\begin{definition}\label{def:orbitface}
Retain the notation of Definition~\ref{def:normal_fan}.
Given a convex cone $\V$ in $N$, a face $F$ of $Q$ is an \textbf{orbit face for $\V$} if the relative interior of the cone $\C(F)$ intersects $\V$. 
\end{definition}

Theorem~\ref{thm:Brion-Woodward}, which is essentially a restatement of \cite[Proposition 5.3.2]{brion:notes}, gathers some results of Brion's from~\cite{brion:picard} and Woodward's from \cite{woodward}.
\begin{theorem}[Brion, Woodward]\label{thm:Brion-Woodward}
Let $(X,\lb)$ be a spherical polarized $G$-variety with momentum polytope $Q$ and valuation cone $\V$. Let $s \in H^{0}(X,\lb)^{(B)} \setminus\{0\}$. 
\begin{enumerate}[(a)]
\item The map $Y \mapsto Q(\overline{Y},\lb|_{\overline{Y}})$ is a bijection between the set of $G$-orbits in $X$ and the set of orbit faces of $Q$ for $\V$. \label{thm:Brion-Woodward:orbitface}
\item The colored fan $\F(X)$ of $X$ is the set of pairs $(\C(F),\D(F))$ where $F$ varies over the orbit faces of $Q$ for $\V$ and for each such face $F$
\begin{equation}
\D(F) := \{D \in \col(X) \colon \<\rho(D), p-\chi(s)\> + v_D(s) =0 \text{ for all $p \in F$}\}. 
\end{equation}
\end{enumerate}
\end{theorem}

When $(X,\lb)$, $Q$ and $\V$ are as in Theorem~\ref{thm:Brion-Woodward},  we will also simply say ``orbit face of $Q$'' instead of ``orbit face of $Q$ for $\V$.''

\section{Two characterizations of the momentum polytope of a spherical variety} \label{sec:characterizations_mp}
In this section, we give two purely combinatorial criteria to decide whether a couple $(\lat,Q)$ consisting in a lattice $\lat\inn \wl$ and a rational convex polytope $Q\subseteq \QQ_{\ge 0}\dw$
is geometrically realizable by a polarized spherical $G$-variety; see Theorem~\ref{thm:crit-convexhull} and Theorem~\ref{thm:compatible-polytope} below.

\subsection{A characterization as a convex hull}\label{s:crit1}

The characterization of momentum polytopes of polarized spherical varieties as convex hulls is inspired by a result of Brion's, namely \cite[Proposition 4.2]{brion:image}. 
Before giving the characterization in Theorem~\ref{thm:crit-convexhull}, we  give a slightly stronger version of Brion's result in Proposition~\ref{Prop_from_Brion}. 

Let $\widetilde\lat\subseteq \twl$ be a lattice and $Q\subseteq\QQ_{\ge 0}\dw$ be a convex polytope.
To these data, we associate the  submonoid $\wm(Q)$ of $\tdw$ defined by:
\begin{equation}
\wm(Q)=\QQ_{\ge 0}(Q\times \{1\})\cap \widetilde{\lat}.
\end{equation}

Because it relies on realizing a polarized spherical $G$-variety as the quotient
by $\{1\} \times \GGm \subseteq \widetilde{G}$ of an affine spherical $\widetilde{G}$-variety, the criterion given in this section involves the monoid $\wm(Q)$ and 
the combinatorial classification of affine spherical varieties with prescribed weight monoid in terms of admissible sets of spherical roots given in \cite[Theorem 6.9]{ACF} and in \cite[Proposition 2.24]{PVS}. 
We briefly review this classification, following the terminology used in~\cite{ACF}. 

We recall the definition of another invariant: if $X$ is a spherical $G$-variety with open $B$-orbit $X_0^B$, then 
\[
\Spp(X) : = \{\alpha \in \sr \colon P_{\alpha} \cdot X_0^B = X_0^B\},
\]
where $P_{\alpha}$ is the minimal parabolic subgroup strictly containing $B$ and associated to the simple root $\alpha$. 

\begin{definition} \label{def_SR_comp_with_lattice}
Let $\lat \subseteq \wl$ be a sublattice. A spherical root $\sigma \in \Sigma(G)$ is said to be
\textbf{compatible with $\lat$} if it satisfies the following properties:

\begin{enumerate}
\item \label{CL1}
$\sigma$ is a primitive element of~$\lat$;
\item \label{CL2}
the pair $(\Spp(\lat), \sigma)$ satisfies \textbf{Luna's axiom (S)}:  there exists a spherical $G$-variety $Z$ of rank one such that $\Spp(Z) = \Spp(\lat)$ and $\Sigma(Z)=\{\sigma\}$.
\item \label{CL3}
if $\sigma = \alpha + \beta$ or $\sigma = \frac12(\alpha + \beta)$ for some $\alpha, \beta \in \sr$ with $\alpha \perp \beta$, then $\langle \alpha^\vee, \lambda \rangle = \langle \beta^\vee, \lambda \rangle$ for all $\lambda \in \lat$;

\item \label{CL4}
if $\sigma = 2\alpha$ for some $\alpha \in \sr$, then $\langle \alpha^\vee, \lambda \rangle \in 2\ZZ$ for all $\lambda \in \lat$.
\end{enumerate}
\end{definition}

\begin{remark}
As is well-known, Luna's axiom (S) can be stated in a purely combinatorial fashion, cf.~\cite[\S 1.1.6]{bravi-luna}.
\end{remark}

Until Theorem~\ref{Theorem-AS}, $\wm$ will denote a finitely generated submonoid of $\dw$ which is saturated, i.e.\ for which the following equality holds in $\wl_{\QQ}$: 
\[ \ZZ\wm \cap \QQ_{\ge 0}\wm = \wm, \]
where $\ZZ\wm$ is the subgroup of $\wl$ generated by $\wm$. 
As is well known, the weight monoid $\wm(X)$ of an affine spherical variety $X$ is saturated because $X$ is normal. Recall that the dual cone of $\wm$ in $\Hom_{\QQ}(\QQ\wm,\QQ)$ is denoted by $\wm^\vee$.

\begin{definition} \label{def_SR_comp_with_monoid}
A spherical root $\sigma \in \Sigma(G)$ is said to be \textbf{compatible with $\wm$} if $\sigma$ is compatible with $\ZZ \wm$ and satisfies the following conditions:

\begin{enumerate}
\item \label{CM1}
if $\sigma \notin \sr$ then for every ray generator $\rho \in \wm^\vee$ such that $\<\rho,\sigma\>>0$, there exists $\delta \in \sr \setminus \Spp(\Gamma)$ such that $\delta^\vee|_{\ZZ\wm}$ is a positive multiple of~$\rho$. 

\item \label{CM2}
if $\sigma = \alpha \in \sr$ then there exist $\rho_1, \rho_2 \in \wm^\vee \cap \mathrm{Hom_{\ZZ}(\ZZ\Gamma,\ZZ)}$ with the following properties:
\begin{enumerate}
\item
$\langle \rho_1, \alpha \rangle = \langle \rho_2, \alpha \rangle = 1$;

\item
$\alpha^\vee|_{\ZZ\wm} = \rho_1 + \rho_2$;

\item \label{CM2-3}
if $\rho\in\wm^\vee$ is a ray generator such that  $\rho(\alpha)>0$ then  $\rho=\rho_1$ or $\rho=\rho_2$.
\end{enumerate}
\end{enumerate}
By $\Sigma(\Gamma)$, we denote the set of all spherical roots $\sigma \in \Sigma(G)$ that are compatible with~$\Gamma$.
\end{definition}

\begin{remark} \label{rem:Sdetermined}
Observe that it follows from (\ref{CM2}) of Definition~\ref{def_SR_comp_with_monoid} that at least one of $\rho_1,\rho_2$ is a ray generator of $\wm^{\vee}$, and consequently the set $\{\rho_1,\rho_2\}$ is uniquely determined by $\wm$.
\end{remark}

To every $\alpha \in \Sigma(\Gamma) \cap \sr$, 
we associate a two-element set $\mathcal S(\alpha) = \lbrace D_\alpha^+, D_\alpha^- \rbrace$ equipped with the map $\rho \colon \mathcal S(\alpha) \to \mathrm{Hom(\ZZ\Gamma,\ZZ)}$ given by $\rho(D_\alpha^+) = \rho_1$ and $\rho(D_\alpha^-) = \rho_2$.  

\begin{proposition} \label{prop_comp_with_monoid}
For a spherical root $\sigma \in \Sigma(G)$, the following
conditions are equivalent.

\begin{enumerate}
\item \label{comp_with_Gamma1}
$\sigma \in \Sigma(\Gamma)$.

\item \label{comp_with_Gamma2}
There exists an affine spherical $G$-variety $X$ with $\wm(X) = \Gamma$ and $\Sigma(X) = \lbrace \sigma \rbrace$.
\end{enumerate}
\end{proposition}

\begin{definition}\label{def_AS}
A subset $\Sigma \subseteq \Sigma(\Gamma)$ is said to be \textbf{admissible} for $\wm$ if it satisfies the following condition:
\begin{enumerate}
\item \label{AP}
for every $\alpha \in \Sigma \cap \sr$, $D \in \mathcal S(\alpha)$, and $\sigma \in \Sigma \setminus \lbrace \alpha \rbrace$, the inequality $\langle \rho(D), \sigma \rangle \le 1$ holds, 
and the equality is attained if and only if $\sigma = \beta \in \sr$ and there is $D' \in \mathcal S(\beta)$ with $\rho(D') = \rho(D)$.
\end{enumerate}
\end{definition}

\begin{theorem} \label{Theorem-AS}
For a subset $\Sigma \subseteq \Sigma(\Gamma)$, the following conditions are equivalent.
\begin{enumerate}
\item $\Sigma$ is admissible for $\wm$.
\item There exists an affine spherical $G$-variety $X$ with $\Gamma(X) = \Gamma$ and $\Sigma(X) = \Sigma$.
\end{enumerate}
\end{theorem}

We also need to recall two basic facts about spherical roots. The first one is that if $Y$ is an affine spherical $G$-variety, then
\begin{equation} \label{eq:coneSigma}
\QQ_{\ge 0}\Sigma(Y) = \QQ_{\ge 0}\{\lambda + \mu - \nu \colon \lambda,\mu,\nu \in \wm(Y) \text{ and } \CC[Y]_\nu \subseteq \CC[Y]_{\lambda}\cdot\CC[Y]_{\mu} \}
\end{equation}
where $\CC[Y]_{\lambda}$ is the submodule of $\CC[Y]$ of highest weight $\lambda$ and $\CC[Y]_{\lambda}\cdot\CC[Y]_{\mu}$ is the subspace of $\CC[Y]$ spanned by $\{f\cdot g \colon f \in \CC[Y]_{\lambda}, g \in\CC[Y]_{\mu}\}$. For a proof of \eqref{eq:coneSigma}, see e.g. \cite[Lemma 5.1]{knop:LV} or \cite[Prop. 4.2]{brion:notes}.

The second one is the following lemma, which is well-known. We provide a proof for completeness.

\begin{lemma}\label{lemma:compare-spherical-roots}
Let $(X,\lb)$ be a polarized spherical $G$-variety with momentum polytope $Q$ and let $\widetilde X$ be its affine cone.
Then $\Spp(X)=\Spp(\widetilde{X})=\Spp(Q)$ and $\Sigma_G(X)=\Sigma_{\tG}(\widetilde X)$, where $\lat(X)$ is identified with its image in $\lat(\widetilde{X})$ under the map in \eqref{eq:latticecone}.
\end{lemma}

\begin{proof}
 We view all varieties involved here as $\tG$-varieties, by letting
 $\GGm$ act trivially on $X$. The open $\widetilde G$-orbit $X_0$ of
 $X$ is the quotient by the action of $\GGm$ of the open $\widetilde
 G$-orbit $\widetilde X_0$ of $\widetilde X$. Denote this quotient by
 $\varphi\colon  \widetilde X_0\to X_0$.

Since $\GGm$ is contained in the center of $\tG$ and therefore in any of its Borel subgroups, the assertion $\Spp(X)=\Spp(\widetilde{X})$ follows. Since $\widetilde X$ is affine, it is well-known that $\Spp(\lat(\widetilde{X}))=S^\perp(\widetilde X)$ (see e.g.~\cite[Proposition 5.3]{ACF}), and thanks to  $\Spp(\lat(\widetilde{X})) = \Spp(Q)$, the equality $\Spp(\widetilde{X})=\Spp(Q)$ holds as well. We now prove that $\Sigma_{\tG}(X_0) =
 \Sigma_{\tG}(\widetilde{X}_0)$.  First we claim that this equality
 holds up to replacing some elements with positive rational multiples.
 Indeed, consider the natural linear map $\varphi^*\colon \lat(X_0)
 \to \lat(\widetilde X_0)$ induced by pulling back rational functions
 (it is the map in (\ref{eq:latticecone})), and the dual map
 $\varphi_*\colon N(\widetilde X_0)\to N(X_0)$. Since $\GGm$ is in
 the center of $\widetilde G$, it follows form the first part of the proof of \cite[Theorem 6.1]{knop:LV} that the kernel of $\varphi_*$ is
 contained in the linear part of the valuation cone of $\widetilde X_0$.
Since by \cite[Theorem 4.4]{knop:LV} we have $\varphi_*(\V(\widetilde X_0)) = \V(X_0)$, this implies the claim.

 Now, notice that $\varphi$ is either injective, or has fiber isomorphic
 to $\GGm$. Therefore, by \cite[Lemma 2.4]{gandini}, the quotient
 $\lat(\widetilde X_0)/\lat(X_0)$ has no torsion, hence the elements of
 $\Sigma_{\tG}(X_0)$ are primitive in $\lat(\widetilde X_0)$. This shows
 the desired equality.
 \end{proof}

The following has been inspired by, and has essentially the same proof as Proposition 4.2 in \cite{brion:image}. 
\begin{proposition}\label{Prop_from_Brion}
Let $V$ be a finite dimensional $G$-module and let $X\inn \mathbb P(V)$ be a closed and spherical $G$-subvariety with set of spherical roots $\Sigma$, and with extended weight lattice $\widetilde\lat$ and momentum polytope $Q$ for the polarization given by the restriction $\lb$ of $\Oh_{\PP(V)}(1)$ to $X$. Let $\lambda_1,\lambda_2,\ldots,\lambda_s$ be those $\lambda \in \dw$ for which $\{\lambda\} = Q(Y,\lb|_Y)$ for some closed $G$-orbit $Y$ in $X$. Then $s\ge 1$ and
\begin{enumerate}
\item $(\lambda_1,1),(\lambda_2,1),\ldots,(\lambda_s,1)\in\Gamma(Q)$; \label{Prop_from_Brion_item0}
\item $\lambda_1,\lambda_2,\ldots,\lambda_s$ are vertices of $Q$ as well as highest weights of the dual module $V^*$;\label{Prop_from_Brion_item1}
   \item  $Q\inn\Conv(\lambda_1,\lambda_2,\ldots,\lambda_s)-\QQ_{\ge 0}\Sigma$; and \label{Prop_from_Brion_item2}
   \item $\Sigma$ is an admissible set for the monoid $\wm(Q)$. \label{Prop_from_Brion_item3}
\end{enumerate}
\end{proposition}
\begin{proof}
That $s\ge 1$, assertion (\ref{Prop_from_Brion_item0}) and the first part of assertion (\ref{Prop_from_Brion_item1}) follow from Proposition~\ref{prop:PSVFundFacts}-\ref{prop:PSVFundFacts:item:mompolclosedorbit} and the definition of $\wm(Q)$. Let $\widetilde X \inn V$ be the affine cone given by the inclusion $X\inn \PP(V)$. Recall from \eqref{eq:wm_affine_cone} that the weight monoid of the $\tG$-variety $\widetilde X$ is $\wm(Q)$. Since the surjective map $\CC[V] \to \CC[\widetilde X]$ induced by the inclusion $\widetilde{X} \inn V$ respects the grading, it follows that every $\lambda_i$ is the highest weight of an irreducible $G$-submodule of $V^*$, which completes the proof of assertion (\ref{Prop_from_Brion_item1}).

In this proof,  when $\tilde \gamma \in \wm(\widetilde{X})$, 
we will write $V(\tilde \gamma)$ for the $\tG$-submodule of $\CC[\widetilde{X}]$ of highest weight $\tilde\gamma$. 
Set $F=\{\lambda_1,\lambda_2,\ldots,\lambda_s\}$ and for every $\lambda \in F$ set $\tilde{\lambda} = (\lambda,1) \in \tdw$. Let $R\inn \CC[\widetilde X]$ be the subalgebra generated by the $\widetilde G$-modules $V(\tilde\lambda)$ 
with $\lambda \in F$.
Note that $R$ is graded because it is generated by homogeneous elements (of degree 1),  and that it is a finitely generated $G$-algebra.
Moreover, we claim that the morphism $f:\widetilde X\rightarrow \mathrm{Spec}(R)$
induced by the inclusion $R\subseteq \CC[\widetilde X]$ is finite.
To prove this claim, one notices that $f^{-1}(0)$ does not contain any cone $\widetilde Y$ over a closed $G$-orbit $Y$ of $X$ since $f$ is the identity on $\widetilde Y$. This implies that
$f^{-1}(0)=\{0\}$ which allows to conclude the proof of the claim, since $f$ is homogeneous of degree $1$.
As a consequence, the morphism $f//U: \mathrm{Spec}\,  \CC[\widetilde X]^U\rightarrow \mathrm{Spec}(R^U)$ induced by $f$ is also finite. 
Take $\nu\in Q=Q(X,\lb)$. Then, by the definition of $Q(X,\lb)$,
 there exists $n>0$ such that the homogeneous component 
$\CC[\widetilde X]_n$ contains the simple $\widetilde G$-module $V((n\nu,n))$.
This together with the finiteness of $f//U$ implies that
$V(p(n\nu,n))\subseteq R_{np}$ for some $p>0$. Set $\tilde \nu = (\nu,1)$. 
By the definition of $R$, there thus exist $\lambda_{j_1}, \lambda_{j_2}, \ldots, \lambda_{j_{np}} \in F$ such that
$V(pn\tilde\nu)$ is contained in the product $V(\tilde\lambda_{j_1})\cdot V(\tilde\lambda_{j_2}) \cdot \ldots \cdot V(\tilde\lambda_{j_{np}})\inn \CC[\widetilde X]$.
By \eqref{eq:coneSigma} and Lemma~\ref{lemma:compare-spherical-roots}, it follows that $\tilde\lambda_{j_1}+...+\tilde\lambda_{j_{np}}-np\tilde\nu\in \QQ_{\ge 0}\Sigma(\widetilde X) = \QQ_{\ge 0}\Sigma$. This proves assertion (\ref{Prop_from_Brion_item2}). 
Assertion (\ref{Prop_from_Brion_item3}) follows from Theorem~\ref{Theorem-AS} and the fact that the weight monoid of $\widetilde X$ is precisely $\wm(Q)$.
\end{proof}

\begin{remark}
In Brion's Proposition 4.2 of \cite{brion:image}, the inclusion $Q\inn\Conv(\lambda_1,...,\lambda_s)-\QQp\sr$ is obtained with $\sr$ being the set of simple roots of $G$.
\end{remark}

The main objective of this section consists in proving the converse of Proposition~\ref{Prop_from_Brion} 
in the more general setting of \textbf{$G$-varieties over $\mathbb P(V)$}, for some finite dimensional $G$-module $V$,
namely pairs $(X,f)$ with $X$ being a projective $G$-variety and $f:X\rightarrow \mathbb P(V)$ being a finite $G$-morphism.
Let $\lb=f^*(\mathcal O(1))$. 
By \textbf{the momentum polytope} (resp. \textbf{the extended weight lattice}) \textbf{of $X$ over $\PP(V)$}, we mean the momentum polytope (resp. the extended weight lattice) of $(X,\lb)$ .

Based on Proposition~\ref{Prop_from_Brion}, we start by introducing the following notion.

\begin{definition}\label{def:quadruple}
A quadruple $(\widetilde{\lat}, V, Q, \Sigma)$ where $\widetilde{\lat}$ is a sublattice of $\twl$, $V$ is a finite dimensional $G$-module, $Q$ is a convex polytope in $\QQp\dw$ and $\Sigma \inn \Sigma(G)$ is called \textbf{a momentum quadruple of $G$} if 
\begin{equation} \label{eq:lattice_of_GammaQ}
\ZZ\Gamma(Q)=\widetilde\lat
\end{equation}
and there exist elements of degree $1$ of $\widetilde\lat$, say $(\lambda_1,1),(\lambda_2,1),\ldots,(\lambda_s,1)$ such that the following properties are satisfied:
\begin{enumerate}
   \item $\lambda_1,\lambda_2,\ldots,\lambda_s$ are vertices of $Q$ as well as highest weights of the dual $V^*$;\label{def:quadruple:vert}
   \item  $Q\inn\Conv(\lambda_1,\lambda_2,\ldots,\lambda_s)-\QQp\Sigma$; and \label{def:quadruple:conv}
   \item $\Sigma \inn \Sigma(G)$ is admissible for the weight monoid $\Gamma(Q)$.  \label{def:quadruple:adm}
\end{enumerate}
\end{definition}

\begin{remark} 
\begin{enumerate}[(a)]
\item If $(\widetilde{\lat}, V, Q, \Sigma)$ is a momentum quadruple then so is $(\widetilde{\lat}, V, Q, \Sigma')$ whenever $\Sigma'$ is a set which contains $\Sigma$ and which is admissible with respect to $\Gamma(Q)$.
\item Condition~\eqref{eq:lattice_of_GammaQ} implies that the monoid $\wm(Q)$ is saturated in $\widetilde{\lat}$. 
It is also equivalent to the condition that $Q-p$ be a spanning subset of $(\widetilde{\lat} \cap \wl)_{\QQ}$ for some (equivalently, any) $p\in Q$ (compare with \eqref{eq:QXLminusp}). 
\end{enumerate}
\end{remark}

\begin{lemma}\label{lemma:degree-1-element}
Given a momentum quadruple $(\widetilde\lat, V,Q,\Sigma)$ of $G$, let 
$E$ denote a minimal set of generators of the monoid $\Gamma (Q)$
and let $X(\Sigma)\subseteq W=\oplus_{\tilde\nu\in E} V(\tilde\nu)^*$ be an affine spherical $\widetilde G$-variety with weight monoid $\Gamma(Q)$ and set of spherical roots $\Sigma$.

Suppose $(\lambda_1,1),(\lambda_2,1),\ldots,(\lambda_s,1) \in \widetilde{\lat}$ satisfy conditions (\ref{def:quadruple:vert}) and (\ref{def:quadruple:conv}) in Definition~\ref{def:quadruple}. 
If $\tilde\nu\in E$ and the highest weight vector $v_{\tilde\nu^*}$ is an element of $X(\Sigma)$, then $\tilde{\nu}$ is equal to one of the weights $(\lambda_i,1)$.
\end{lemma}

\begin{proof}
Set $\wm(Q)^{\vee}:=(\QQp\wm(Q))^{\vee}$, and notice that in our setting it is full-dimensional in $\Hom_{\QQ}(\QQ\wm,\QQ)$, because $\QQp\wm(Q)$ doesn't contain any line. Let $\C$ be the largest face of $\wm(Q)^{\vee}$ whose relative interior intersects the valuation cone $\V$ of $X(\Sigma)$. After Proposition~\ref{prop:affine-criterion} and \cite[Lemma 3.2]{knop:LV}, the cone $\C$ has a face $\C'$ corresponding to the $\tG$-orbit closure $Y$ of $v_{\tilde\nu^*}$ in $X(\Sigma)$. 
Note that $\lat(\tG\cdot v_{\tilde\nu^*}) = \ZZ\tilde{\nu}$, hence by the proof of \cite[Theorem 6.3]{knop:LV} it follows that $\ZZ\tilde{\nu} = \ZZ\wm(Q) \cap (\C')^{\perp}.$ This implies that $\C'$ equals the facet $\wm(Q)^{\vee} \cap \tilde{\nu}^\perp$ of $\wm(Q)^{\vee}$. Consequently, the ray $\QQ_{\ge 0} \tilde{\nu}$ is an edge of $\QQ_{\ge 0}\wm(Q)$.  In particular, if we write $\tilde\nu=(n\nu,n)$, then $\nu$ is a vertex of $Q$. 

We claim that $\C'$ meets the relative interior of $\V$. To prove the claim, we observe that it is obvious if $\C'=\C$, because in this case $X(\Sigma)=Y$ and $\Sigma=\emptyset$. For dimension reasons, the only other possibility is that $\C=\Gamma(Q)^\vee$. Then $\C$ and $\V$ are both full-dimensional convex cones, each intersecting the relative interior of the other. Since $\C'$ is the facet of $\C$ having inward-pointing normal $\tilde\nu$, there are elements of $\V$ that are strictly positive on $\tilde\nu$. On the other hand, it is well known that $\V$ contains the image in $N(X(\Sigma))$ of the antidominant chamber, hence there are elements of $\V$ that are strictly negative on $\tilde\nu$. Since the relative interior of $\C'$ intersects $\V$, an elementary argument shows that $\C'$ also intersects the relative interior of $\V$.

Suppose now, to get a contradiction, that $\tilde\nu\neq\tilde\lambda_i$ for all $i \in \{1,2,\ldots,s\}$, where $\tilde\lambda_i=(\lambda_i,1)$.
Let $\tilde{\rho}\in \C'$.
Then $\<\tilde{\rho},\tilde\nu\>=0$ and $\<\tilde{\rho},\tilde{\lambda_i}\>\geq 0$, for all $i$.
These inequalities together with Definition~\ref{def:quadruple}-(\ref{def:quadruple:conv})
imply that there exists $\sigma\in\Sigma$ such that $\<\tilde{\rho},\sigma\>\geq 0$.
Therefore the above claim does not hold --- a contradiction. 
\end{proof}

\begin{theorem}\label{thm:crit-convexhull}
Let $V$ be a finite dimensional $G$-module, $\widetilde\lat$ be a sublattice of $\twl$, $Q\subseteq\QQ_{\geq 0}\dw$ be a convex polytope and $\Sigma$ be a set of spherical roots of $G$.

The following assertions are equivalent.
\begin{enumerate}
\item The tuple $(\widetilde\lat,V,Q,\Sigma)$ is a momentum quadruple of $G$.
\item There exists a spherical projective $G$-variety over $\mathbb{P}(V)$
with extended weight lattice $\widetilde\lat$,
momentum polytope $Q$ and set of spherical roots $\Sigma$.   
\end{enumerate}
\end{theorem}

\begin{proof}
$(2)\Rightarrow (1)$
To prove this implication, we follow the proof of Proposition~\ref{Prop_from_Brion}. Take a finite morphism $f: X\rightarrow \PP(V)$.
Let $\lb=f^*(\mathcal O(1))$ and $\widetilde X = \Spec R(X,\lb)$ be the affine cone of $(X,\lb)$, which is a $\widetilde G$-variety.
Consider  the simple $G$-submodules $V(\lambda_i^*)\inn V$ such that there exists a closed $G$-orbit $Y_i$ of $X$ with $f(Y_i) \inn \PP(V(\lambda_i^*)) \inn \PP(V)$.
We define $R$ to be the subalgebra of $\CC[\widetilde X$] generated
by the corresponding simple $\widetilde G$-modules $V(\tilde\lambda_i)$
with $\tilde\lambda_i=(\lambda_i,1)$.
We thus argue as for the proof of Proposition~\ref{Prop_from_Brion} while considering this algebra $R$.
Finally, note that $\lambda_i$ are indeed vertices of $Q(X,\lb)$; see Proposition~\ref{prop:PSVFundFacts}-\ref{prop:PSVFundFacts:item:mompolclosedorbit}.

$(1)\Rightarrow (2)$
Thanks to Definition~\ref{def:quadruple}-$(3)$, there exists a spherical $\tilde G$-variety $X(\Sigma)$ with weight monoid $\Gamma(Q)$ and set of spherical roots equal to $\Sigma$.
Recall that $X(\Sigma)$ can be regarded as a $\tilde G$-subvariety of $W=\oplus _E V(\tilde\nu)^*$ where $E$ denotes a minimal set of generators of the monoid $\Gamma (Q)$.

Let $R$ be the subalgebra of the coordinate ring of $X(\Sigma)$ generated by the modules $V(\tilde\lambda_i)$ with $\tilde\lambda_i=(\lambda_i,1)$ for $i\in\{1,\ldots,s\}$.
We then have a natural morphism  $\varphi: X(\Sigma)\rightarrow \mathrm{Spec}(R)$; we shall prove that conditions (\ref{def:quadruple:vert}) and (\ref{def:quadruple:conv})  imply that $\varphi$  is finite.
Once this is proven, we get a finite $\tG$-equivariant morphism $X(\Sigma)\rightarrow V$, where the factor $\GGm$ of $\tG$ acts by scalar multiplication on $V$. In turn we obtain that the quotient $\operatorname{Proj}(\CC[X(\Sigma)])$ 
is a projective spherical $G$-variety over $\mathbb P(V)$ which satisfies our requirements.

Let us thus prove that the morphism $\varphi$ is finite. 
Since $\varphi$ is homogeneous of degree $1$, it suffices to show that $\varphi^{-1}(0)$  equals the singleton $\{0\}$.
Note that the $\tilde G$-orbit closures of $v_{\tilde\lambda_i^*}$ within $W$ are obviously not contained in the fiber $\varphi^{-1}(0)$ because $\varphi$ restricted on such orbit closures is the identity map.
Lemma~\ref{lemma:degree-1-element} shows that the $\tilde G$-orbit closures of $v_{\tilde\nu^*}$, for $\tilde\nu\neq\tilde\lambda_i$ and $\tilde\nu\in E$, are not contained in this fiber either. This concludes the proof.
\end{proof}

\begin{example}  Take $G=SL_2\times SL_2$.
Let $\lambda_1=2\varpi_1$, $\lambda_2=4\varpi_1+2\varpi_1'$ and $V=V(\lambda_1)\oplus V(\lambda_2)$.
We consider the polytope $Q$ defined as the convex hull of $0$, $2\lambda_1$ and $\lambda_2$.

In \cite[Example 3.20]{alexeev-brion:SSV}, Alexeev and Brion show that the polytope $Q$ cannot be realized as the momentum polytope of a projective spherical variety over $\mathbb P(V)$. 
Let us recover this result, by applying Theorem~\ref{thm:crit-convexhull}.

Assume there exists a projective spherical $G$-variety $X$ over $\mathbb P(V)$ with momentum polytope $Q$.
Then $\lambda_2$ is the only vertex of $Q$ satisfying Definition~\ref{def:quadruple}-(\ref{def:quadruple:vert}).  Let $\widetilde{\lat}$ be the extended weight lattice of $X$.
The properties (\ref{def:quadruple:conv}) and (\ref{def:quadruple:adm})  in Definition~\ref{def:quadruple} imply that $X$ should have at least two spherical roots, which have to be admissible for the weight monoid $\Gamma(Q) = \QQp(Q\times\{1\}) \cap \widetilde{\lat}$. Recall  that $\Sigma(G)=\{\alpha_1,\alpha_1', 2\alpha_1,2\alpha_1',\alpha_1+\alpha_1',\frac{1}{2}(\alpha_1+\alpha_1')\}$. It follows from Definition~\ref{def_SR_comp_with_monoid}-(\ref{CM1}) and from Definition~\ref{def_SR_comp_with_lattice}-(\ref{CL3}) that $2\alpha_1, \alpha_1+\alpha_1'$ and $\frac{1}{2}(\alpha_1+\alpha_1')$  are not compatible with $\Gamma(Q)$. This means that $\Sigma(X) = \{\alpha_1,\alpha_1'\}$ or $\Sigma(X) = \{\alpha_1,2\alpha_1'\}$. In particular, $\alpha_1$ is compatible with $\wm(Q)$. One of the ray generators of $\wm(Q)^{\vee}$ is a positive multiple of $\rho_1:=\frac{1}{2}\alpha_1^{\vee}|_{\widetilde{\lat}}-\alpha_1'^\vee|_{\widetilde{\lat}}$. Since $\<\rho_1,\alpha_1\> = 1$, it follows from  Definition~\ref{def_SR_comp_with_monoid}-(\ref{CM1}) that $\rho_1$ is a ray generator of $\wm(Q)$. Consequently, there exists $D \in \mathcal{S}(\alpha_1)$ such that $\rho(D) = \alpha_1^{\vee}|_{\widetilde{\lat}} - \rho_1$. Since $\<\rho(D),\alpha_1'\> = 2$, it follows that neither $\{\alpha_1,\alpha_1'\}$ nor $\{\alpha_1,2\alpha_1'\}$ satisfies Definition~\ref{def_AS}.
This contradicts Theorem~\ref{thm:crit-convexhull} 
hence shows that there is no projective $G$-spherical variety over $\mathbb P(V)$ with momentum polytope $Q$.
\end{example}

\begin{corollary}\label{cor:crit-convexhull}
Let $V$ be a finite dimensional $G$-module, $\lat\subseteq\wl$ be a lattice and $Q\subseteq\QQ_{\geq 0}\dw$ be a convex polytope.  
The following assertions are equivalent.
\begin{enumerate}
\item There exists a momentum quadruple $(\widetilde\lat, V,Q,\Sigma)$ for some $\widetilde\lat$ spanned by $\lat$ and some $(\lambda,1)$'s with $\lambda$ being some highest weights of $V^*$.
\item
There exists a  projective spherical $G$-variety over $\mathbb{P}(V)$ with weight lattice $\lat$ and momentum polytope $Q$.

\end{enumerate}
\end{corollary}

\begin{proof}
This follows from Theorem~\ref{thm:crit-convexhull} together with Corollary~\ref{cor:tildelattice}.
\end{proof}

\begin{example}[cf.~{\cite[Chapter 3]{foschi}}] \label{ex:Foschi}
Take $G=SL_3$ along with
$\lambda_1=4\varpi_1+4\varpi_2$, $\lambda_2=5\varpi_1+2\varpi_2$ and $\lambda_3=2\varpi_1+5\varpi_2$.
Let $V=V(\lambda_1)^*\oplus V(\lambda_2)^*\oplus V(\lambda_3)^*$,
$\widetilde\lat$ be the lattice generated by $(\lambda_i,1)$, $i=1,2$ or $3$,
and $Q$ be the convex hull of the weights $\lambda_1,\lambda_2$ and $\lambda_3$.
Since $Q$ is a lattice polytope, $(\widetilde \lat,V,Q,\emptyset)$ is obviously a momentum quadruple.
One next computes that the set $\{\alpha_1,\alpha_2\}$ of simple roots of $G$ forms an admissible set for $\Gamma(Q)$. 
The power set of $\{\alpha_1,\alpha_2\}$ thus yields
four non-isomorphic projective spherical $G$-varieties with given data $(\widetilde\lat,V,Q)$.
\end{example}

\subsection{A characterization as an intersection of half-spaces}
The criterion we establish in this section (Theorem~\ref{thm:compatible-polytope}) is also derived from the classification of affine spherical varieties with a given weight monoid recalled in the previous subsection. It has the advantage that it can be stated without resorting to a monoid.

Throughout this subsection,  $\lat\subseteq \wl$ denotes a lattice and $Q\subseteq \QQ_{\geq 0} \dw$ is a convex polytope such that $Q-\omega$ is a full dimensional polytope in $\lat_{\mathbb Q}$ for some (equivalently, any) $\omega\in Q$.

The results given in this subsection are based on the notion of an {\em admissible set} of spherical roots for the couple $(\lat,Q)$. Before defining it, we introduce some basic notations.
Fix $\omega\in Q$. 
Let $F$ be any facet of $Q$. 
In our setting, $F$ generates an affine subspace of $\wl_\QQ$ of dimension $\rk \lat-1$, which we denote by $H_F$.
Analogously, the corresponding facet $F-\omega$ of $Q-\omega$ generates an affine hyperplane $H_{F-\omega}$ of $\lat_\QQ$ which we can write as
\[
H_{F-\omega}=\left\{\xi\in\lat_\mathbb Q: \left\langle\rho_F,\xi\right\rangle + m_{F,\omega} =0\right\}
\]
where $m_{F,\omega}\in\QQ$ and $\rho_F$ is a primitive element of $\Hom_\ZZ(\lat, \ZZ)$. We choose $\rho_F$ to be inward-pointing, i.e.\ such that it takes on $Q-\omega$ values that are greater than or equal to $-m_{F,\omega}$.

Notice that $\rho_F$ does not depend on $\omega$; it is uniquely determined by $F$ and it uniquely determines $F$ among the facets of $Q$: we call it the \emph{primitive inward-pointing facet normal} of $F$.

We thus have
\[
Q=\omega+\left\{\xi\in\lat_\mathbb Q: \left\langle\rho_F,\xi\right\rangle + m_{F,\omega} \geq 0 \text{ for all facets $F$ of $Q$}\right\}.
\]

In the following two definitions we introduce the projective analogues of Definitions~\ref{def_SR_comp_with_monoid} and~\ref{def_AS}. First, recall the notion of a spherical root compatible with a lattice given in Definition~\ref{def_SR_comp_with_lattice}. 

\begin{definition}\label{def:Q-compatible}
A spherical root $\sigma\in\Sigma(G)$ is \textbf{$\mathbb Q$-compatible with $(\lat,Q)$} if $\sigma$ is compatible with $\lat$, the couple $(\Spp(Q),\sigma)$ satisfies Luna's axiom (S), and $\sigma$ satisfies the following properties:
\begin{enumerate}
\item\label{Q-compatible-b} if $\sigma\notin S$ and a facet $F\subseteq Q$ satisfies $\left\langle\rho_F,\sigma\right\rangle>0$, then there exists $\alpha\in S\smallsetminus \Spp(Q)$ 
such that $\<\alpha^\vee, F\> =0$.
\item \label{Q-compatible-a} if $\sigma=\alpha\in S$ then there exists a facet $F\subseteq Q$ such that
\begin{enumerate}
\item $\left\langle\rho_F,\alpha \right\rangle =1$;
\item if $F'\subseteq Q$ is a facet such that $\left\langle\rho_{F'},\alpha \right\rangle>0$ then $H_{F'}=H_{F}$ or $H_{F'}=s_\alpha (H_{F})$.
\end{enumerate}
\item \label{Q-compatible-d2} if $\sigma = \alpha + \beta$ or $\sigma = \frac{1}{2}(\alpha + \beta)$ for two orthogonal simple roots $\alpha$ and $\beta$, then $\<\alpha^{\vee},q\> = \<\beta^{\vee},q\>$ for all $q \in Q$.
\end{enumerate}

We denote by $\Sigma_\mathbb Q(\lat, Q)$ the set of spherical roots that are $\mathbb Q$-compatible with $(\lat, Q)$.

For a simple root $\alpha\in \Sigma_\QQ(\lat,Q)$,
we denote by $\A(\alpha)$ a set with two elements $D_\alpha^+$ and $D_\alpha^-$, 
and we define a map $\rho\colon \A(\alpha)\to \Hom_\ZZ(\lat,\ZZ)$ 
by setting $\rho(D_\alpha^+)=\rho_F$ and $\rho(D_\alpha^-)=\alpha^\vee|_{\lat}-\rho_F$, where we choose
\footnote{From now on we implicitly fix a choice of such a facet for all $\alpha\in\Sigma_\QQ(\lat,Q)$; see Lemma~\ref{lemma:sum}-(\ref{lemma:sum:both}).}
a facet $F\subseteq Q$ as in (\ref{Q-compatible-a}). 
\end{definition}

\begin{remark} 
\begin{enumerate}[(a)]
\item In case~(\ref{Q-compatible-b}) of the above definition, notice that $\alpha^\vee$ is not constantly $0$ on $Q$. This implies that $F=Q\cap \{\<\alpha^\vee,-\>=0 \}$, and that $\rho_F=\alpha^\vee|_\lat$ up to a positive rational factor.
\item The conditions on the facet $F$ in parts (\ref{Q-compatible-b}) and (\ref{Q-compatible-a}) of Definition~\ref{def:Q-compatible} had already been observed by Woodward in \cite[Theorem 2.5]{woodward}, under the additional assumption that $\rk \lat = \rk \wl$. In fact, it is straightforward to deduce from \loccit\ that if $(X,\lb)$ is a polarized spherical $G$-variety of maximal rank, then every $\sigma \in \Sigma(X)$ satisfies  (\ref{Q-compatible-b}) and (\ref{Q-compatible-a}) in Definition~\ref{def:Q-compatible} with respect to $Q=Q(X,\lb)$ and $\lat = \lat(X)$. We use a similar argument to Woodward's in the proof of Proposition~\ref{prop:realized_then_admissible} and of Theorem~\ref{thm:existence-Kaehler-converse}, without the assumption on the rank of $\lat(X)$. 
\end{enumerate}
\label{rem:multiple_of_coroot}
\end{remark}

In the next lemma, we give some first consequences of $\QQ$-compatibility and $\QQ$-admissibility on
the configuration of the polytope $Q$. Part (\ref{lemma:sum:both}) of the lemma implies that, up to permuting $D_{\alpha}^+$ and $D_{\alpha}^-$, the pair $(\A(\alpha),\rho)$ of Definition~\ref{def:Q-compatible} does not depend on the choice of the facet $F$ of $Q$ as in part (\ref{Q-compatible-a}) of that definition. 

\begin{lemma}\label{lemma:sum}
Let $\alpha\in \Sigma_\QQ(\lat,Q)\cap S$.
\begin{enumerate}
\item\label{lemma:sum:general} For all $q\in Q$ we have
\[
\<\rho(D_\alpha^-), q-\omega\> + \<\alpha^\vee,\omega\> - m_{F,\omega} \geq 0,
\]
where $F$ is the facet of $Q$ such that $\rho(D_\alpha^+)=\rho_F$.
\item\label{lemma:sum:both} Suppose that there exist two distinct facets $F,F'\subseteq Q$ whose primitive inward-pointing facet normals are positive on $\alpha$. Then we have the following two equalities:
\begin{align}
\rho_F + \rho_{F'}&=\alpha^\vee|_\lat; \label{eq:sum1}\\ 
m_{F,\omega} + m_{F',\omega}&=\<\alpha^\vee,\omega\>. \label{eq:sum2}
\end{align}
\end{enumerate} 
\end{lemma}

\begin{proof}
We prove part (\ref{lemma:sum:general}). Since
\[
H_{F} = \{x \in \lat_\QQ+\omega \colon \<\rho_F, x-\omega\> + m_{F,\omega} = 0\}
\]
it is elementary  to show that 
\[
s_\alpha(H_{F}) =\{x \in \lat_\QQ+\omega \colon \<\rho(D_\alpha^-), x-\omega\> + \<\alpha^\vee,\omega\> - m_{F,\omega} = 0\},
\]
using the fact from Definition~\ref{def:Q-compatible} that $\rho(D_\alpha^-)=\alpha^\vee|_\lat - \rho_F$	.
	
Suppose, for the sake of contradiction, that there exists $q\in Q$ such that
\[
\<\rho(D_\alpha^-), q-\omega\> + \<\alpha^\vee,\omega\> - m_{F,\omega} < 0.
\]
Consider the maximum rational number $a$ such that $q-a\alpha\in Q$: we have $a\geq 0$, and the point $q'=q-a\alpha$ is on a facet $E$ of $Q$ such that $\<\rho_E,\alpha\> >0$ (otherwise $q-\QQ_{\geq0}\alpha$ would be entirely contained in $Q$).
Then
\[
\<\rho(D_\alpha^-), q'-\omega\> + \<\alpha^\vee,\omega\> - m_{F,\omega} < 0,
\]
in particular the left hand side is $\neq0$. This excludes the possibility that $H_{E}=s_\alpha(H_{F})$. Then $H_{E}=H_{F}$, i.e.\ $E=F$. This yields $\<\rho_F, q'-\omega\> + m_{F,\omega}= 0$, and
\[
\<\alpha^\vee, q'\> = \<\rho_F, q'-\omega\> + \<\rho(D_\alpha^-), q'-\omega\> + \<\alpha^\vee,\omega\> <0,
\]
which contradicts $Q \inn \QQp\dw$. 

We now prove part~(\ref{lemma:sum:both}). By Definition~\ref{def:Q-compatible} we have $H_{F'}=s_\alpha(H_{F})$, which implies $\rho_{F'}= \pm (\alpha^{\vee}|_\lat-\rho_F)$. Since $\rho_F$ and $\rho_{F'}$ both take value $1$ on $\alpha$, we have $\rho_F + \rho_{F'}=\alpha^\vee|_\lat$. 
	
Since $F$ and $F'$ are distinct subsets of $\QQp\dw$, the affine subspaces $H_{F}$ and $H_{F'}$  of $\wl_\RR$ are not parallel.  It follows that  $H_{F}\cap H_{F'}$ is non-empty and stable under $s_\alpha$, hence there exists $p\in H_{F}\cap H_{F'}$ such that $s_\alpha(p)=p$, in other words $\<\alpha^\vee,p\> =0$. Then $p-\omega \in H_{F-\omega}\cap H_{F'-\omega}$, which yields
\[
m_{F,\omega}+ m_{F',\omega} = \< \rho_F, \omega-p \> + \< \rho_{F'}, \omega-p \> = \<\alpha^\vee,\omega-p\>=\<\alpha^\vee,\omega\>,
\]	
as desired.
\end{proof}

The analogue of Definition~\ref{def_AS} while replacing $\mathcal S$ by $\A$ reads as follows.
\begin{definition}\label{def:Q-admissibleset}
A subset $\Sigma\inn \Sigma_\mathbb Q(\lat, Q)$ is $\mathbb Q$-\textbf{admissible} (for $(\lat,Q)$) if it satisfies the following condition:
\begin{enumerate}[label=(\arabic*'),ref=\arabic*']
\item \label{A1polytope} For every $\alpha \in \Sigma \cap S$, $D \in \A(\alpha)$, and $\sigma \in \Sigma \setminus \lbrace \alpha \rbrace$, the inequality $\langle \rho(D), \sigma \rangle \le 1$ holds, 
and the equality is attained if and only if $\sigma = \beta \in S$ and there is $D' \in \mathcal \A(\beta)$ with $\rho(D') = \rho(D)$. 
\end{enumerate}
\end{definition}

\begin{remark} \label{rem:Qadmiss_subsets}
Notice that the property of being $\QQ$-admissible for a subset of $\Sigma_\QQ(\lat, Q)$  is closed under taking subsets.  Even a stronger property holds evidently: a subset $\Sigma\subseteq\Sigma_\QQ(\lat, Q)$ is $\QQ$-admissible if and only if $\{\sigma, \tau\}$ is $\QQ$-admissible for all $\sigma,\tau\in\Sigma$.  
We also point out that $\Sigma=\{\sigma\}$ is $\QQ$-admissible if and only if $\sigma$ is $\QQ$-compatible with $(\lat,Q)$: 
$\QQ$-admissibility for singletons does not involve checking condition~(1') of Definition~\ref{def:Q-admissibleset}.
\end{remark}

Before defining admissible sets of spherical roots, we need to introduce some additional notation. 
Let $\Sigma \inn \Sigma_\mathbb Q(\lat, Q)$ be $\mathbb Q$-admissible.
Suppose $\alpha\in S\cap\frac12\Sigma$.
Recall from Remark~\ref{rem:multiple_of_coroot} that  whenever $\<\rho_F,\alpha\> >0$ for a facet $F$ of $Q$, then $\rho_F$ equals a positive rational multiple of $\alpha^\vee|_{\lat}$.
In this case, it is useful to rescale $\rho_F$ and $m_{F,\omega}$ accordingly, as follows. 
We denote
\[
\rho^\Sigma_F =
\begin{cases}
\frac12\alpha^\vee|_{\lat} & \text{if $\<\rho_F,\alpha\> >0$ for some $\alpha\in S\cap\frac12\Sigma$,} \\
\rho_F & \text{otherwise,}
\end{cases}
\]
and we set
\[
m^\Sigma_{F,\omega} = -\<\rho^\Sigma_F,F-\omega\>.
\]

Let $\Sigma\inn \Sigma(G)$. In accordance with Definition~\ref{def:orbitface}, we call a face of $Q$ an \textbf{orbit face with respect to $\Sigma$} if it is an orbit face of $Q$ for the cone dual to $-\Sigma$. When $\Sigma$ is clear from the context,  we will sometimes omit ``with respect to $\Sigma$.''

\begin{definition}\label{def:admissible}
A $\QQ$-admissible subset $\Sigma\inn \Sigma_\QQ(\lat,Q)$ is \textbf{admissible} if the following conditions are satisfied:
\begin{enumerate}
\item\label{def:admissible:inXi} for all orbit vertices $\mathsf v,\mathsf w$ of $Q$ w.r.t. $\Sigma$, the difference $\mathsf v-\mathsf w$ is in $\lat$;
\item\label{def:admissible:v} there exists an orbit vertex $\mathsf v\in Q$ w.r.t. $\Sigma$ such that
\begin{enumerate}
\item\label{def:admissible:v:L} $\mathsf v\in \wl$ and
\item\label{def:admissible:v:a_2a} $m^\Sigma_{F,\mathsf v}\in\ZZ$, for all facets $F\subseteq Q$ such that $\<\rho_F,\alpha\> > 0$ with $\alpha\in S\cap (\Sigma\cup \frac12\Sigma)$.
\end{enumerate}
\end{enumerate}
\end{definition}

\begin{remark}\label{rem:orbitvertices}
\begin{enumerate}[(a)]
\item Observe that when condition~(\ref{def:admissible:inXi}) of Definition~\ref{def:admissible} is true, then  condition~(\ref{def:admissible:v})  holds for all orbit vertices of $Q$, whenever it holds for one of them.
\item It follows from Remark~\ref{rem:multiple_of_coroot} that if $\Sigma \inn \Sigma_{\QQ}(\lat,Q)$ is $\QQ$-admissible and $F$ is a facet of $Q$ such that $\<\rho_F,\sigma\> > 0 $ for some $\sigma \in \Sigma \setminus \sr$, then 
\begin{equation} \label{eq:m_simpler}
m_{F,\omega}^{\Sigma} = a\<\alpha^{\vee},\omega\>,
\end{equation}
where $a$ is the positive rational number such that $\rho_F^{\Sigma} = a \alpha^{\vee}|_{\lat}$.
\item It is straightforward to verify that a subset $\Sigma$ of $\Sigma(G)$ is $\QQ$-admissible for $(\lat,Q)$ if and only if there exists a positive integer $n$ such that $\Sigma$ is admissible for $(\lat, nQ)$. 
\item In general, admissibility is not closed under taking subsets of $\Sigma$, because of the integrality conditions involved in the definition (see Example~\ref{ex:bizarre}).
\end{enumerate}
\end{remark}

Here is our second criterion for the geometric realizability of a pair $(\lat,Q)$. Its proof will be given after Proposition~\ref{prop:realized_then_admissible}.

\begin{theorem}\label{thm:compatible-polytope}
For a subset $\Sigma\inn\Sigma(G)$, the following assertions are equivalent.
\begin{enumerate}
\item The set $\Sigma$ is admissible (resp.\ $\QQ$-admissible) for the couple $(\lat,Q)$.
\item There exists a polarized spherical $G$-variety with weight lattice $\lat$, set of spherical roots $\Sigma$, and momentum polytope $Q$ (resp.\ $nQ$ for some positive integer $n$).
\end{enumerate}
\end{theorem}

\begin{remark} 
\begin{enumerate}[(a)]
\item \label{rem:subsets_of_sigma_a} It follows from Theorem~\ref{thm:compatible-polytope} and Remark~\ref{rem:Qadmiss_subsets} that a pair $(\lat,Q)$ can be realized by a $\QQ$-polarized spherical $G$-variety if and only if $\Sigma = \emptyset$ is $\QQ$-admissible for $(\lat,Q)$. For completeness, we recall that a $\QQ$-polarized $G$-variety is a pair $(X,\lb)$ where $X$ is a projective $G$-variety and $\lb$ is an element of $\Pic_G(X)_{\QQ}$ such that $n\lb \in \Pic_G(X)$ is ample for some $n \in \NN$.
\item Given $G$ and a pair $(\lat,Q)$, deciding whether there exists a subset $\Sigma$ of $\Sigma(G)$ which is admissible requires only finitely many verifications. 
In particular, one can, at least in principle, decide algorithmically whether such a pair $(\lat,Q)$ is realized by a polarized spherical $G$-variety. 
Indeed, the set $\Sigma(G)$ is finite, and by Definitions~\ref{def:Q-compatible} and \ref{def:Q-admissibleset} determining the subsets of $\Sigma(G)$ that are $\QQ$-admissible with respect to $(\lat,Q)$ can be done in finitely many steps. 
It is clear from Definition~\ref{def:admissible} that, determining whether a given $\QQ$-admissible subset $\Sigma$ of $\Sigma(G)$ is admissible only requires finitely many further verifications. 
\end{enumerate}  \label{rem:subsets_of_sigma}
\end{remark}

\begin{example} \label{ex:bizarre}
The property of being admissible for a subset of $\Sigma_\QQ(\lat, Q)$ is not closed under taking subsets.
We exhibit here a minimal admissible set of spherical roots having two elements. The polytope $Q$ is taken from \cite[Section~2.1, example c)]{brion:image}.
Let $G=\Sp(6)$, with simple roots $\alpha_i$ and corresponding fundamental dominant weights $\varpi_i$ ($i\in\{1,2,3\}$) numbered as in Bourbaki~\cite{bbki}. 
Define $Q$ as the convex hull of the points $\frac12\varpi_2$, $\varpi_2$, and $\frac13(\varpi_1+\varpi_3)$.
Take $\lat=\Span_\ZZ\{\alpha_1+\alpha_3, \alpha_2\}$.

Then $Q-q$ is contained and of full dimension in $\lat_\QQ$ for any $q\in Q$. One checks that $\Sigma=\{\alpha_1+\alpha_3,\alpha_2\}$ is admissible for $(\lat,Q)$.
Notice that $\mathsf v:=\varpi_2$ is the unique orbit vertex and also the unique lattice point in $Q$.
With respect to any proper subset $\Sigma'\subsetneq\Sigma$, the polytope $Q$ has more than one orbit vertex;
those different from $\mathsf v$ are not integral weights.
Therefore condition~(\ref{def:admissible:inXi}) of Definition~\ref{def:admissible} cannot be fulfilled for such $\Sigma'$ and $\Sigma'$ is not admissible.
\end{example}

We now begin the proof of Theorem~\ref{thm:compatible-polytope}, which we divide into several steps.

\begin{lemma}\label{lemma:sum:11}
Let $\alpha,\beta$ be two distinct simple roots in $\Sigma_\QQ(\lat,Q)$ 
such that $\{\alpha,\beta\}$ is $\QQ$-admissible.
If $\< \rho(D),\alpha\> = \<\rho(D),\beta\> = 1$ for some $D\in\A(\alpha)$
then $\rho(D)$ is an inward-pointing facet normal of $Q$.
\end{lemma}

\begin{proof}
Let $F_{\alpha}$ (resp. $F_{\beta}$) be the facet of $Q$ such that $\rho_{F_{\alpha}} = \rho(D^+_{\alpha})$ (resp. $\rho_{F_{\beta}} = \rho(D^+_{\beta})$) as in Definition~\ref{def:Q-compatible}.
Since $\{\alpha,\beta\}$ is $\QQ$-admissible, $\rho(D)=\rho(D')$ for some $D'\in\A(\beta)$. If $D=D^+_{\alpha}$ or $D' = D^+_{\beta}$ then we are done; so we assume that neither of these two equalities hold. 
By the definitions of $\A(\alpha)$ and $\A(\beta)$, it then follows that $\rho(D) + \rho(D^+_\alpha) = \alpha^{\vee}|_\lat$ and $\rho(D) + \rho(D^+_\beta) = \beta^{\vee}|_\lat$.

Let $F$ be a facet of $Q$ such that $\<\rho_{F},\alpha+\beta\> >0$ (which exists because $\alpha + \beta \in \lat$ and $Q-\omega$ is full dimensional in $\lat_\QQ$). We will prove that $\rho_{F}=\rho(D)$.
Since
\[
\<\rho(D_\alpha^+),\alpha+\beta\> = \<\alpha^\vee|_{\lat} - \rho(D), \alpha+\beta\> = \<\alpha^\vee,\beta\> \leq 0
\]
and similarly $\<\rho(D_\beta^+),\alpha+\beta\> \leq 0$, we know that $F\neq F_{\alpha}$ and $F \neq F_{\beta}$. 
But since $\<\rho_{F},\alpha+\beta\> >0$, we have that $\<\rho_F,\alpha\> > 0$ or $\<\rho_F,\beta\> > 0$. It follows from equation~\eqref{eq:sum1} that $\rho_F + \rho(D^+_{\alpha}) = \alpha^\vee|_{\lat}$ or   $\rho_F + \rho(D^+_{\beta}) = \beta^\vee|_{\lat}$. Either way, $\rho(D) = \rho_F$, as desired.
\end{proof}

Let us now assume that $\Sigma \inn \Sigma(G)$ is admissible for $(\lat,Q)$ and let $\mathsf{v}$ be an orbit vertex of $Q$. 
Recall that $\mathsf{v} \in \wl$ by Definition~\ref{def:admissible}-(\ref{def:admissible:v:L}) and let $\widetilde\lat$ be the following sublattice of $\twl = \wl \times \ZZ$:
\begin{equation} \label{eq:ext-lattice-admissible}
\widetilde{\lat} : = (\lat \times \{0\}) \oplus \ZZ(\mathsf{v},1).
\end{equation}
Observe that by Definition~\ref{def:admissible}-(\ref{def:admissible:inXi}), this lattice does not depend on the choice of $\mathsf{v}$.
We will often implicitly identify $\lat$ with $\lat \times \{0\} \inn \widetilde\lat$ using the map $\lambda \mapsto (\lambda,0)$. 
As before, we set
\begin{equation} \label{eq:defwmQ}
\wm(Q):= \QQp(Q \times \{1\}) \cap \widetilde{\lat} \inn \twl_{\QQ}.
\end{equation}
Moreover, since $Q-\mathsf{v}$ is a full dimensional polytope in $\lat_\QQ$, we have the equalities
\begin{align}
\ZZ\wm(Q)&= \widetilde\lat \label{eq:latwmQ} \\
\Spp(Q)&=\Spp(\widetilde{\lat}); \label{eq:aux-equality}
\end{align}

The next elementary lemma in convex geometry will be used frequently in what follows.
\begin{lemma}\label{lem:aux-ray}
Suppose $\mathsf{v}$ is any point in $\wl \cap Q$ and define $\widetilde{\lat}$ and $\wm(Q)$ by equation \eqref{eq:ext-lattice-admissible} and  equation~\eqref{eq:defwmQ} respectively. The assignment $\widetilde{\rho}\mapsto a\widetilde\rho |_\lat$, where $a\in\QQ_{>0}$ is such that $a\widetilde\rho |_\lat$ is primitive in $\Hom_\ZZ(\lat, \ZZ)$, defines a bijective correspondence between the ray generators of $\Gamma(Q)^\vee$ and the primitive inward-pointing facet normals of $Q$.
The inverse map is given, for every facet $F$ of $Q$, by $\rho_F\mapsto \widetilde\rho$ where $\widetilde\rho$ denotes the ray generator of $\Gamma(Q)^\vee$ defining the facet of $\QQp\wm(Q)$ containing $F\times\{1\}$.
\end{lemma}

\begin{proof}
Elementary, using that dualizing \eqref{eq:ext-lattice-admissible} we have $\Hom_\ZZ(\widetilde{\lat},\ZZ)=\Hom_\ZZ(\lat,\ZZ)\oplus\ZZ$.
\end{proof}

Our next aim is to establish that $\Sigma$ is admissible for the weight monoid $\wm(Q) \inn \tdw$ (in the sense of Definition~\ref{def_SR_comp_with_monoid}.) We break the verifications up in a few lemmas. 

\begin{lemma} \label{lemma:aux-compatibility}
Suppose $\Sigma \inn \Sigma_{\QQ}(\lat,Q)$ is admissible for $(\lat,Q)$. If $\sigma \in \Sigma$, then $\sigma$ is compatible with the lattice $\widetilde{\lat}$. 
\end{lemma}
\begin{proof}
We check the conditions in Definition~\ref{def_SR_comp_with_lattice}. Recall that, by assumption, $\sigma$ is compatible with $\lat$. That $\sigma$ is primitive in $\widetilde\lat$ follows from the fact that it is primitive in $\lat \times \{0\} \inn \widetilde\lat$. 
Condition (\ref{CL2}) of Definition~\ref{def_SR_comp_with_lattice} follows from the equality~\eqref{eq:aux-equality} and the fact that $(\Spp(Q),\sigma)$ satisfies Luna's axiom (S). Condition (\ref{CL3}) follows from the fact that $\sigma$ is compatible with $\lat$ and part (\ref{Q-compatible-d2}) of Definition~\ref{def:Q-compatible}. Finally, condition (\ref{CL4})  in Definition~\ref{def_SR_comp_with_lattice} follows from Definition~\ref{def:Q-compatible}-(\ref{Q-compatible-b}), equation~\eqref{eq:m_simpler} and condition (\ref{def:admissible:v:a_2a}) of Definition~\ref{def:admissible}. 
\end{proof}

\begin{lemma} \label{lemma:aux-compatibility-monoid}
Suppose $\Sigma \inn \Sigma_{\QQ}(\lat,Q)$ is admissible for $(\lat,Q)$. If $\sigma \in \Sigma$, then $\sigma$ is compatible with the monoid $\wm(Q)$. 
\end{lemma}
\begin{proof}
By Lemma~\ref{lemma:aux-compatibility} and equation~\eqref{eq:latwmQ}, what is left is to verify conditions (\ref{CM1}) and (\ref{CM2}) of Definition~\ref{def_SR_comp_with_monoid}. Suppose that $\sigma\in \Sigma \smallsetminus S$, and let $\widetilde\rho$ be a ray generator of $\Gamma(Q)^\vee$ such that $\<\widetilde\rho,\sigma\> >0$. Let $\rho_F$ be the inward-pointing facet normal of $Q$ corresponding to $\widetilde\rho$ under the correspondence in Lemma~\ref{lem:aux-ray}. Then $\<\rho_F,\sigma\> > 0$ and so, by Definition~\ref{def:Q-compatible}-(\ref{Q-compatible-b}), there exists $\alpha \in S \smallsetminus \Spp(\wm(Q))$ such that $\<\alpha^{\vee},F\> = 0$. Therefore $\alpha^{\vee}$ vanishes on the facet of $\QQp\wm(Q)$ that contains $F\times\{1\}$ but not on $\wm(Q)$. Consequently, $\widetilde{\rho}$ is a positive rational multiple of $\alpha^{\vee}|_{\widetilde{\lat}}$, which proves condition (\ref{CM1}) of Definition~\ref{def_SR_comp_with_monoid}. 

Suppose now that $\sigma=\alpha\in S$, let $F$ be a facet of $Q$ as in part (\ref{Q-compatible-a}) of Definition~\ref{def:Q-compatible} and let $\widetilde\rho$ be the ray generator of $\Gamma(Q)^\vee$ corresponding to $\rho_F$ as given by Lemma~\ref{lem:aux-ray}. We claim that $\rho_1:=\widetilde\rho$ and $\rho_2:=\alpha^\vee|_{\widetilde{\lat}}-\widetilde\rho$ satisfy the three properties required by Definition~\ref{def_SR_comp_with_monoid}-(\ref{CM2}).

To prove this claim, we first observe that $\widetilde{\rho} = (s\rho_F, a) \in \Hom_\ZZ(\ZZ\Gamma(Q),\ZZ) = \Hom_{\ZZ}(\lat,\ZZ)\oplus \ZZ$ for some $s\in\QQ_{\ge 0}$ and some $a \in \ZZ$. Since $\<\rho_F, \alpha\> = 1$, we have that $\mathsf{v}-m^{\Sigma}_{F,\mathsf{v}}\alpha \in H_F$ and consequently that 
\[
0 =\<\widetilde\rho,(\mathsf{v}-m^{\Sigma}_{F,\mathsf{v}}\alpha,1)\> =\<(s\rho_F,a),  -m_{F,\mathsf{v}}(\alpha,0) + (\mathsf{v},1)\> = -sm^{\Sigma}_{F,\mathsf{v}} + a,
\]
that is, $s{m^{\Sigma}_{F,\mathsf{v}}}=a$. As $s$ is a positive integer,
since $\alpha\in\lat$ and $\<\widetilde{\rho},(\alpha,0)\> = s\<\rho_F,\alpha\> = s$, and as $\widetilde\rho$ is primitive in $\Hom_\ZZ(\ZZ\Gamma(Q),\ZZ)$, we know that $\gcd(s,a)=1$. Because $m^{\Sigma}_{F,\mathsf{v}} \in \ZZ$ by part (\ref{def:admissible:v:a_2a}) of  Definition~\ref{def:admissible}, it follows that $a= m^{\Sigma}_{F,\mathsf{v}}$ and that 
\begin{equation}\label{eq:rhorestricted}
\widetilde \rho = (\rho_F,m^{\Sigma}_{F,\mathsf{v}}).
\end{equation}
The equalities $\langle \rho_1, \alpha \rangle = \langle \rho_2, \alpha \rangle = 1$ follow, and $\alpha^\vee|_{\widetilde{\lat}} = \rho_1 + \rho_2$ holds by construction. The latter equality also yields 
\begin{equation} \label{eq:rho2restricted}
\rho_2 = (\alpha^\vee|_\lat - \rho_F, \<\alpha^\vee,\mathsf{v}\> - m^{\Sigma}_{F,\mathsf{v}}),
\end{equation}
 so by Lemma~\ref{lemma:sum}, part (\ref{lemma:sum:general}), we have that $\rho_2$ is non-negative on $Q\times\{1\}$. This implies that $\rho_2\in\Gamma(Q)^\vee$. We have proved the first two required properties.

To prove the third required property, let $\overline{\rho}$ be a ray generator of $\wm(Q)^\vee$ for which $\<\overline{\rho},\alpha\> > 0$.  Let  $\rho_{F'}$ be the corresponding primitive inward-pointing facet normal of $Q$. If $F'=F$, then $\overline{\rho} = \widetilde{\rho}=\rho_1$. If $F' \neq F$, then it follows from equation~\eqref{eq:sum1} that $\rho_F' = \alpha^{\vee}|_\lat - \rho_F$, and so $\<\rho_F',\alpha\>=1$. Using the same argument as above, we then find that  $\overline{\rho} = (\rho_{F'},m^{\Sigma}_{F',\mathsf{v}})$.  
Thanks to equation \eqref{eq:sum2}, we then obtain $\widetilde{\rho} = (\alpha^\vee|_\lat - \rho_F, \<\alpha^\vee,\mathsf{v}\> - m^{\Sigma}_{F,\mathsf{v}}) = \rho_2$.  
This proves the third required property and concludes the proof.
\end{proof}

\begin{lemma} \label{lemma:aux-admissibility-monoid}
Suppose $\Sigma \inn \Sigma_{\QQ}(\lat,Q)$ is admissible for $(\lat,Q)$. Then $\Sigma$ is admissible for the monoid $\wm(Q)$. 
\end{lemma}
\begin{proof}
Since we know by Lemma~\ref{lemma:aux-compatibility-monoid} that $\Sigma \inn \Sigma(\wm(Q))$, we only need to check that $\Sigma$ satisfies axiom (\ref{AP}) of Definition~\ref{def_AS}.  
For $\alpha \in \Sigma \cap \sr$, let $\rho_1,\rho_2$ be as in Definition~\ref{def_SR_comp_with_monoid} and $\rho(D_\alpha^+),\rho(D_{\alpha}^-)$ as in Definition~\ref{def:Q-compatible}. In the proof of Lemma~\ref{lemma:aux-compatibility-monoid}
we showed that, after permuting $\rho_1$ and $\rho_2$ if necessary, we have $\rho(D_{\alpha}^+) = \rho_1|_\lat$ and $\rho(D_{\alpha}^-) = \rho_2|_{\lat}$.   To avoid confusion we denote here by $\widetilde{\rho}$ the  map $\mathcal S(\alpha)\to\Hom_\ZZ(\widetilde{\lat},\ZZ)$ with $\widetilde{\rho}(D_\alpha^+)=\rho_1$ and $\widetilde{\rho}(D_\alpha^-)=\rho_2$, and by $\rho$ the map $\A(\alpha)\to\Hom_\ZZ(\lat,\ZZ)$ of Definition~\ref{def:Q-compatible}. 

By the  discussion in the preceding paragraph we can, for every $\alpha \in \Sigma\cap \sr$, identify $\mathcal S(\alpha)$ and $\A(\alpha)$ in such a way that $\widetilde{\rho}|_\lat=\rho$. At this point axiom (\ref{AP}) of Definition~\ref{def_AS} follows from axiom (\ref{A1polytope}) of Definition~\ref{def:Q-admissibleset}, if we prove the following claim: for all $\alpha,\beta\in\Sigma\cap S$, if $\rho(D)=\rho(E)$ for some $D\in\mathcal S(\alpha)$  and $E \in \mathcal{S}(\beta)$, then $\widetilde{\rho}(D)=\widetilde{\rho}(E)$. 

To prove the claim, we may suppose that $\alpha\neq\beta$. It then follows from Lemma~\ref{lemma:sum:11} that $\rho(D)=\rho(E)$ is a primitive inward-pointing facet normal of $Q$. Let $\overline{\rho}$ be the corresponding ray generator of $\wm(Q)^{\vee}$ as in Lemma~\ref{lem:aux-ray}.  In the proof of Lemma~\ref{lemma:aux-compatibility-monoid} we  showed that $\widetilde{\rho}(\mathcal{S}(\alpha)) = \{\overline{\rho}, \alpha^{\vee}|_{\widetilde{\lat}}-\overline{\rho}\}$ and $\widetilde{\rho}(\mathcal{S}(\beta)) = \{\overline{\rho}, \beta^{\vee}|_{\widetilde{\lat}}-\overline{\rho}\}$. It follows from  the equalities $\<\rho(D),\beta\> = 1$ and $\<\rho(E),\alpha\>=1$ that $\widetilde{\rho}(D) = \overline{\rho} = \widetilde{\rho}(E)$. This proves the claim and the lemma.
\end{proof}

We can now establish one implication in Theorem~\ref{thm:compatible-polytope}.
\begin{proposition} \label{prop:admissible_then_realized}
If $\Sigma \inn \Sigma_{\QQ}(\lat,Q)$ is admissible for $(\lat,Q)$ then there exists a polarized spherical $G$-variety $(X,\lb)$ such that $\lat(X) = \lat$, $\Sigma(X)=\Sigma$ and $Q(X,\lb)=Q$. 
\end{proposition}
\begin{proof}
From Lemma~\ref{lemma:aux-admissibility-monoid} and Theorem~\ref{Theorem-AS}, we obtain an affine spherical $\tG$-variety $\widetilde{X}$ with $\wm(\widetilde{X}) = \wm(Q)$ and $\Sigma_{\tG}(\widetilde{X}) = \Sigma$. It follows from equation~\eqref{eq:ext-lattice-admissible} and the exact sequence \eqref{eq:latticecone} that the lattice of the polarized spherical $G$-variety $\operatorname{Proj}(\CC[\widetilde{X}])$ is $\lat$, from equation~\eqref{eq:conepolytope} that its momentum polytope is $Q$ and from Lemma~\ref{lemma:compare-spherical-roots} that its set of spherical roots is $\Sigma$.  
\end{proof}

We now establish the converse of the previous proposition.
\begin{proposition} \label{prop:realized_then_admissible}
If $(X,\lb)$ is a polarized spherical $G$-variety with weight lattice $\lat$, momentum polytope $Q$ and set of spherical roots $\Sigma$, then $\Sigma \inn \Sigma_{\QQ}(\lat,Q)$ and $\Sigma$ is admissible for $(\lat,Q)$. 
\end{proposition}
\begin{proof}
It follows from Theorem~\ref{thm:Brion-Woodward} that the orbit vertices of $Q$ (for the cone $\V(X)=(-\Sigma)^{\vee} \inn N(X)$) are exactly the momentum polytopes of the closed $G$-orbits of $X$. Conditions (\ref{def:admissible:inXi}) and (\ref{def:admissible:v:L}) of Definition~\ref{def:admissible} now follow from Proposition~\ref{prop:PSVFundFacts}-\ref{prop:PSVFundFacts:item:mompolclosedorbit} and Corollary~\ref{cor:tildelattice}. 

Next let $\widetilde{X} = \Spec(R(X,\lb))$ be the affine cone of $(X,\lb)$. Recall that  $\widetilde{X}$ is an affine spherical $\tG$-variety with set of spherical roots $\Sigma$, by Lemma~\ref{lemma:compare-spherical-roots}. Denote its weight lattice by $\widetilde{\lat}$; it is a sublattice of $\twl = \wl \times \ZZ$ and fits into the exact sequence \eqref{eq:latticecone}. In particular, $\widetilde{\lat}\cap\wl = \lat$.  From \eqref{eq:wm_affine_cone} we know that the weight monoid of $\widetilde{X}$ is $\wm(Q) = \QQp(Q \times \{1\}) \cap \widetilde{\lat}$. Theorem~\ref{Theorem-AS} tells us that $\Sigma$ is admissible for $\wm(Q)$. 

We now show that $\Sigma \inn \Sigma_{\QQ}(\lat,Q)$ by verifying the conditions in Definition~\ref{def:Q-compatible}. Let $\sigma\in\Sigma$. Then it is compatible with $\lat$, because it is compatible with the bigger lattice $\widetilde{\lat}$. Moreover $\Spp(Q)=\Spp(\widetilde{\lat})$, hence $(\Spp(Q), \sigma)$ satisfies Luna's axiom (S).

Now, assume $\sigma\notin S$, and let $F\subseteq Q$ be a facet such that $\<\rho_F,\sigma\> >0$. Denote by $\widetilde{\rho}$ the ray generator of $\wm(Q)^{\vee}$ corresponding to $\rho_F$ under the bijection of Lemma~\ref{lem:aux-ray}. Then $\widetilde{\rho}|_\lat=\rho_F$ up to a positive rational factor. Since $\widetilde{\rho}=\alpha^{\vee}|_{\widetilde{\lat}}$ up to a positive rational factor, for some $\alpha \in \sr \smallsetminus \Spp(\widetilde{\lat})$ we get $\rho_F=\alpha^{\vee}|_\lat$ up to a positive rational factor, which proves condition (\ref{Q-compatible-b}) of  Definition~\ref{def:Q-compatible}. 

Assume now $\sigma=\alpha\in S$, and let $\rho_1,\rho_2$ be as in Definition~\ref{def_SR_comp_with_monoid}. We may assume that $\rho_1$ is a ray generator of $\Gamma(Q)^\vee$, let $F$ be the corresponding facet of $Q$. Then $a\rho_1|_{\lat}=\rho_F$ for some rational number $a$ between $0$ and $1$. Because $\rho_F$ is primitive in $\Hom_{\ZZ}(\lat,\ZZ)$ by definition and  $\<\rho_1|_{\lat},\alpha\>=1$, we obtain $\rho_1|_\lat=\rho_F$, which yields $\<\rho_F,\alpha\> = 1$. Suppose now $\<\rho_{F'},\alpha\> >0$ for some facet $F'\subseteq Q$, and let $\rho'$ be the ray generator of $\Gamma(Q)^\vee$ corresponding to $F'$. Notice that $\ker(\rho')\cap(\wl_Q\times\{1\}) = H_{F'}\times\{1\}$.
By hypothesis we have $\rho'=\rho_1$ or $\rho' = \rho_2 = \alpha^{\vee}|_{\widetilde{\lat}} - \rho_1$.  In the former case $H_F=H_{F'}$ and in the latter $H_{F'}=s_\alpha(H_F)$. We have proved condition (\ref{Q-compatible-a}) of Definition~\ref{def:Q-compatible} and therefore $\Sigma\inn \Sigma_\QQ(\lat,Q)$. 

We have also proved that for all $\alpha\in S\cap \Sigma$ we can identify $\mathcal S(\alpha)$ and $\A(\alpha)$ in such a way that $\widetilde{\rho}|_\lat(D)=\rho(D)$ for all $D\in\mathcal S(\alpha)$, where to avoid confusion we denote by $\widetilde{\rho}$ the usual map $\mathcal S(\alpha)\to\Hom_\ZZ(\widetilde{\lat},\ZZ)$ and by $\rho$ the usual map $\A(\alpha)\to\Hom_\ZZ(\lat,\ZZ)$. Condition~(\ref{A1polytope}) of Definition~\ref{def:Q-admissibleset} stems from condition~(\ref{AP}) of Definition~\ref{def_AS}, assuring $\QQ$-admissibility of $\Sigma$. 

To finish the proof that $\Sigma$ is admissible for $(\lat,Q)$, all that remains is to check condition~(\ref{def:admissible:v:a_2a}) of Definition~\ref{def:admissible}. Let  $\mathsf{v}$ be an orbit vertex of $Q$, $\alpha\in S$ and $\<\rho_F,\alpha\>>0$ for a facet $F$ of $Q$. If $2\alpha\in\Sigma$ then $\rho_F^\Sigma=\frac{1}{2}\alpha^\vee|_{\lat}$ and, by Definition~\ref{def_SR_comp_with_lattice}-(\ref{CL4}), $\alpha^{\vee}$ takes even values on $\widetilde{\lat}$.  Equation~\eqref{eq:m_simpler} then gives  $m_{F,\mathsf{v}}^{\Sigma} = \frac{1}{2}\<\alpha^{\vee},\mathsf{v}\>=\frac{1}{2}\<\alpha^{\vee},(\mathsf{v},1)\>\in \ZZ$, since $(\mathsf{v},1) \in \widetilde{\lat}$.

Finally, if $\alpha\in\Sigma$, let $\widetilde{\rho}$ be the ray generator of $\wm(Q)^{\vee}$ corresponding to $\rho_F$. Recall that $\widetilde{\rho} \in \Hom_{\ZZ}(\widetilde{\lat},\ZZ)$ by definition and that $\<\widetilde{\rho},\alpha\>=1$ by condition (\ref{CM2-3}) of Definition~\ref{def_SR_comp_with_monoid}. As above we have $\widetilde{\rho}|_\lat=\rho_F$. Consequently, $m^\Sigma_{F,\mathsf{v}}=m_{F,\mathsf{v}} = \widetilde{\rho}(\mathsf{v},1)\in\ZZ$ and the proof is complete. 
\end{proof}

Finally, we complete the proof of our second criterion. 

\begin{proof}[Proof of Theorem~\ref{thm:compatible-polytope}] \label{thm:compatible-polytope:proof}
The assertion with $\QQ$-admissibility follows from the one with admissibility, since $\Sigma$ is $\QQ$-admissible for $(\lat,Q)$ if and only if it is admissible for $(\lat,nQ)$ for some positive integer $n$.
The theorem follows now from Propositions \ref{prop:admissible_then_realized} and \ref{prop:realized_then_admissible}.
\end{proof}

\section{Some classification results} \label{sec:classifications}

\subsection{A classification of polarized spherical varieties}

A $G$-morphism  between two given polarized $G$-varieties $(X_1,\lb_1)$ and $(X_2,\lb_2)$ is a pair $(f,\varphi)$ where $f:X_1\rightarrow X_2$ is a $G$-equivariant morphism 
and $\varphi: \lb_2\rightarrow f^*\lb_1$ is an isomorphism of $G$-linearized line bundles. 
As is well-known, such $G$-morphisms $(X_1,\lb_1) \to (X_2,\lb_2)$ are in natural bijective correspondence with $\widetilde{G}$-equivariant morphisms $\widetilde{X}_1 \to \widetilde{X}_2$ between the affine cones.

We first deduce a uniqueness statement for polarized spherical varieties from the following uniqueness theorem for affine spherical varieties, due to Losev (see also~\cite{ACF18} for another proof).
\begin{theorem}[{\cite[Theorem 1.2]{losev:knopconj}}] \label{thm:losev}
Two affine spherical $G$-varieties are $G$-equivariantly isomorphic if and only if 
they have the same weight monoid and the same set of spherical roots. 
\end{theorem}

\begin{corollary}\label{cor:PSunique}
Two polarized spherical $G$-varieties $(X_1,\lb_1)$ and $(X_2,\lb_2)$ are $G$-\hspace{0pt}equivariantly isomorphic if and only if
$\lat(X_1)=\lat(X_2)$, $\Sigma(X_1)=\Sigma(X_2)$ and $Q(X_1,\lb_1)=Q(X_2,\lb_2)$.
\end{corollary}

\begin{proof}
We prove the ``if'' statement by applying  Theorem~\ref{thm:losev} to show that the affine cones $\widetilde{X}_1$ and $\widetilde{X}_2$ are $\widetilde G$-equivariantly isomorphic. 
First note that since $X_1$ and $X_2$ have the same weight lattice $\lat$ as well as the same set of spherical roots, they have the same valuation cone. Because these varieties (with the given polarizations) also have the same momentum
polytope $Q$, their closed $G$-orbits correspond to the same vertices of $Q$; see Theorem~\ref{thm:Brion-Woodward}. 
By Corollary~\ref{cor:tildelattice}, it follows that $\widetilde{X}_1$ and $\widetilde{X}_2$ have the same weight lattice $\widetilde\lat$ and, in turn, 
the same weight monoid $\QQ_{\ge 0}(Q \times \{1\})\cap\widetilde\lat$.
Finally, since any polarized spherical $G$-variety has the same set of spherical roots as its affine cone (Lemma~\ref{lemma:compare-spherical-roots}), Theorem~\ref{thm:losev} allows to conclude. 
\end{proof}

\begin{remark}  
In~\cite{losev:uniqueness}, one can find a slightly different proof of the above corollary. See Step 2 of the proof of Theorem 8.3 therein. 
\end{remark}

\begin{corollary}\label{cor:finitelymany}
There exist only finitely many $G$-isomorphism classes of spherical polarized $G$-varieties $X$ with prescribed weight lattice and momentum polytope.
\end{corollary}

\begin{proof}
The set of spherical roots of $G$ being finite, the corollary follows from Corollary~\ref{cor:PSunique}.
\end{proof}

Recall the notion of admissibility (see~Definition~\ref{def:admissible}). 
Let $\lat\subseteq\wl$ be a sublattice, $Q\subseteq\QQ_{\geq 0}\dw$ be a polytope such that $Q-\omega\subseteq \lat_\QQ$ is full dimensional for some $\omega\in Q$ and $\Sigma$ be a subset of $\Sigma(G)$.
We call the triple $(\lat, Q, \Sigma)$ a \textbf{$\QQ$-momentum triple of $G$} (resp. \textbf{momentum triple of $G$}) if $\Sigma\subseteq\Sigma_\QQ(\lat,Q)$ (resp. $\Sigma$ is admissible for the couple $(\lat,Q)$). 
Combining Corollary~\ref{cor:PSunique} and Theorem~\ref{thm:compatible-polytope}, we obtain a classification of polarized spherical $G$-varieties involving these triples.

\begin{theorem}\label{thm:classification-triples}
The map $(X,\lb)\mapsto (\lat(X), Q(X,\lb),\Sigma(X))$
is a bijection between the $G$-isomorphism classes of polarized spherical $G$-varieties and the momentum triples of $G$.
\end{theorem}

\begin{remark}
Pasquier classified the polarized horospherical $G$-varieties by means of some quadruples that he calls moment quadruples; see~\cite[Corollary 2.10]{pasquier:survey}. Recall that a horospherical $G$-variety is a (spherical) $G$-variety such that the stabilizer $H$ of any point of its open $G$-orbit contains a maximal unipotent subgroup of $G$. Moreover, horospherical varieties do not have any spherical roots.
For a polarized horospherical $G$-variety $(X,D)$, Pasquier's moment quadruple  is the quadruple
$(P_X,M_X,Q(X,D),\widetilde Q(X,D))$ where $P_X$ is the parabolic subgroup of $G$ such that $G/H\rightarrow G/P_X$ is a torus fibration; $M_X$ is the sublattice of $\Chi(P_X)$ consisting of characters vanishing on $H$; $Q(X,D)$ is the momentum polytope of $(X,D)$ and 
$\widetilde{Q}(X,D)$ is the polytope $Q(X,D)-q$ where $q$ denotes the $B$-weight of the canonical section of $D$.
Note that $\widetilde{Q}(X,D)$ is contained in $(M_X)_\QQ$. Furthermore,
the datum $(P_X,M_X)$ is equivalent to the datum $\lat(X)$. These observations yield the connection with our momentum quadruples, more precisely with those such that $\Sigma=\emptyset$. 
\end{remark}

\subsection{A classification of all Fano spherical varieties}

Given a normal complex variety $X$, we denote its anti-canonical sheaf by $-K_X$.
If $X$ is projective and $-K_X$ is an ample Cartier divisor then $X$ is called a Fano variety.
In case $-K_X$ is ample but  only $\QQ$-Cartier, $X$ is called $\QQ$-Fano.

In this section, we establish a bijection between all Fano spherical $G$-varieties and peculiar momentum triples.
This result is derived from the classification obtained in Theorem~\ref{thm:classification-triples} together with the description of the anti-canonical sheaf of a projective spherical variety.

We start with the recollection of this description which essentially  follows from some results of Brion's (\cite[Theorem 4.2]{brion97}) and Luna's (\cite[\S 3.6]{luna97}); see~\cite{hof-gag} for details.

We shall need the notation set up in Section~\ref{section:coloredfan} and introduce some further ones. In particular, $\col(X)$ denotes the set of colors of $X$.
Let $P\subseteq G$ be the parabolic subgroup containing $B$ and stabilizing $\Delta(X)$ pointwise and $S^p$ be its corresponding set of simple roots. Observe that $S^p = \Spp(X)$.
By $\rho$ resp.\ $\rho_{S^p}$, we denote the half-sum of the positive roots of $G$, resp.\ of the root system generated by $S^p$. 
For $\alpha\in S$, let $P_\alpha\subseteq G$ be the corresponding minimal parabolic subgroup strictly containing $B$.

\begin{proposition}\label{prop:canonicalsheaf}
Let $X$ be a projective spherical $G$-variety and let $X_1,X_2,\ldots,X_n$ be the $G$-stable prime divisors of $X$. Then there exists $s \in H^0(X,-K_X)^{(B)} \setminus \{0\}$ of $B$-weight $2\rho-2\rho_{S^p}$ with
\begin{equation} \label{eq:div_antican_sheaf}
\operatorname{div}(s)=\sum_{1\leq i\leq n} X_i+\sum_{D\in\Delta(X)} n_D D
\end{equation}
where
\begin{enumerate}
\item $n_D=1$ for $D\in\Delta(X)$ such that $P_\alpha D\neq D$ for some $\alpha \in \sr$ with $\alpha\in\Sigma(X)\cup \frac{1}{2}\Sigma(X)$.

\item $n_D=\left\langle 2\rho-2\rho_{S^p},\alpha^\vee\right\rangle$ for $D\in\Delta(X)$ such that $P_\alpha D\neq D$ where $\alpha \in \sr$ with $\alpha\notin\Sigma(X)\cup \frac{1}{2}\Sigma(X)$.
\end{enumerate}
\end{proposition}

A reflexive $\QQ$-momentum triple is a $\QQ$-momentum triple where the  numbers $m_{F,\msw}^\Sigma$ are prescribed. Here is the precise definition.

\begin{definition} \label{def:reflexivepol}
A $\QQ$-momentum triple $(\lat, Q,\Sigma)$ is called \textbf{$\QQ$-reflexive} if $\msw:=2\rho-2\rho_{\Spp(Q)}\in Q$, and the following properties are satisfied:
\begin{enumerate}[(1)]
\item If $F$ is a facet of $Q$ such that $\left\langle\rho_F,\sigma\right\rangle> 0$ for some $\sigma\in S\cap\Sigma$, then $m_{F, \msw}^\Sigma=1$;
\item If $F$ is a facet of $Q$ such that $\left\langle\rho_F,\sigma\right\rangle \leq 0$ for all $\sigma \in \Sigma$, and $\left\langle\alpha^{\vee},F\right\rangle \neq \{0\}$ for all $\alpha \in S\setminus \Spp(Q)$, then $m_{F, \msw}^\Sigma=1$. 
\end{enumerate}
A momentum triple $(\lat, Q,\Sigma)$ is called \textbf{reflexive} if it is $\QQ$-reflexive and it satisfies the following property:
\begin{enumerate}[(3)]
\item There is an orbit vertex $\msv$ of $Q$ such that $\msv-\msw \in \lat$.  \label{w_integral}
\end{enumerate}
\end{definition}

\begin{remark}
\begin{enumerate}[(a)]
\item The integrality condition \ref{w_integral} in Definition~\ref{def:reflexivepol} does not follow from the integrality conditions on orbit vertices in Definition~\ref{def:admissible}, as the following example shows. Take $G=\Sp(4)$ with $\alpha$ its short simple root and $\beta$ the long one. The weight $\msw$ as in Definition~\ref{def:reflexivepol} is then $\msw=2\rho=4\alpha+3\beta$. Let $Q$ be the convex hull of $0, 4(\alpha+\beta)$ and $8\alpha + 4\beta$ and set $\lat:=\Span_\ZZ\{\alpha+\beta,2\alpha\}$. Then $(\lat,Q,\emptyset)$ is a momentum triple, which is $\QQ$-reflexive, but not reflexive: the three vertices of $Q$, which are all orbit vertices, belong to $\lat$, while $\msw$ does not.
\item There is some redundancy in the definition of a reflexive momentum triple: a $\QQ$-reflexive $\QQ$-momentum triple $(\lat,Q,\Sigma)$ is a reflexive momentum triple if and only if $\msv - \msw \in \lat$ for all orbit vertices $\msv$ of $Q$.
\item If $(\lat,Q,\Sigma)$ is a $\QQ$-reflexive $\QQ$-momentum triple, which is not a reflexive
momentum triple, then there exists $n \in \ZZ_{>0}$, such that $(\lat,nQ,\Sigma)$ is a momentum triple, but no $r \in \ZZ_{>0}$ such that $(\lat,rQ,\Sigma)$ is a reflexive one. 
\end{enumerate}
\end{remark}

\begin{theorem}\label{thm:Fano}
The map $X\mapsto (\lat(X),Q(X,-K_X),\Sigma(X))$ is a bijection between the $\QQ$-Fano spherical $G$-varieties (resp. Fano spherical $G$-varieties) 
and the $\QQ$-reflexive $\QQ$-momentum triples of $G$ (resp. reflexive momentum triples of $G$).
\end{theorem}

\begin{proof}
First, note that if $X$ is a spherical ($\QQ$-)Fano $G$-variety then $(\lat(X),Q(X,-K_X),\Sigma(X))$ is indeed a ($\QQ$-)reflexive ($\QQ$-)momentum triple thanks to Proposition~\ref{prop:canonicalsheaf} and Theorem~\ref{thm:classification-triples}. The injectivity of the given mapping is part of Theorem~\ref{thm:classification-triples}. 
To prove its surjectivity, consider any ($\QQ$-)reflexive ($\QQ$-)momentum triple $(\lat, Q,\Sigma)$. Let $(X,\lb)$ be the corresponding polarized spherical $G$-variety given by Theorem~\ref{thm:classification-triples}. 
In particular, we have $Q(X,\lb)=Q$ (resp.\ $Q(X,\lb)=nQ$ for some $n\in\NN$) and there exists $r \in H^{0}(X,\lb)^{(B)}\setminus\{0\}$ such that $\operatorname{div}(r)$ (resp. $\frac{1}{n} \operatorname{div}(r)$) is equal to $\operatorname{div}(s)$ as in \eqref{eq:div_antican_sheaf}. 
Consequently  $\lb \cong -K_X$ (resp.\ $\lb \cong -nK_X$) as invertible sheaves on $X$, and $-K_X$ (resp.\ $-nK_X$) is an ample line bundle. Finally, $Q(X,-K_X)=Q$ thanks to Proposition~\ref{prop:Pbrion} and Proposition~\ref{prop:canonicalsheaf}.  
\end{proof}

\begin{remark} \label{rmk:Fano}
In~\cite{hof-gag}, generalizing the work of Pasquier's \cite{pasquier08}, Gagliardi and Hofscheier classify the Fano spherical embeddings of a given spherical homogeneous space $G/H$, 
in terms of so-called $G/H$-reflexive polytopes (see \loccit\ for the definition).
The correspondence is given as follows. 
Let $X$ be a such an embedding and $X_1,\ldots, X_n$ be the $G$-stable prime divisors of $X$.
Then $X\mapsto \mathcal Q_X$ where $\mathcal Q_X$ is the convex hull of the $\rho(D)/n_D$ and the $\rho_{X_i}$, for $D\in\Delta(X)$ and $1\leq i\leq n$. In particular, $\mathcal Q_X\subseteq N(G/H)_\QQ$.

The polytopes $\mathcal Q_X$ were introduced by Alexeev and Brion \cite{AB04}; for $X$ a Fano spherical $G$-variety,
it turns out that the momentum polytope of $(X,-K_X)$ is precisely $\mathcal Q_X^*+\msw$  where $\mathcal Q_X^*$ denotes the dual of $\mathcal Q_X$ and $\msw$ is the $B$-weight of the canonical section of $-K_X$ (see Proposition~\ref{prop:Pbrion}).
This yields the connection with the work of Pasquier, Gagliardi and Hofscheier.
\end{remark}

\section{Momentum polytopes of smooth polarized spherical varieties} \label{sec:smoothness}

In this section we will specialize the combinatorial smoothness criterion in \cite{camus} to our setting. It allows one to decide whether an admissible polytope is the momentum polytope of a \emph{smooth} polarized spherical variety. We will use the exposition of the criterion given in \cite[Section 3]{PVS}.

To state Camus' criterion we need a few definitions and notations. We start with spherically closed spherical roots, and the notion of socle. For the motivation behind the next definition, we refer to \cite[Proposition~2.7]{PVS}.

\begin{definition} 
Let $X$ be a spherical $G$-variety. We define the set $\Sigma^{sc}(X)$ of {\em spherically closed spherical roots} of $X$, as the set obtained from $\Sigma(X)$ by replacing $\sigma \in \Sigma(X)$ with $2\sigma$ exactly when it satisfies any one of the following conditions:
\begin{enumerate}
\item\label{prop:doubling_sr:B}  $\sigma = \alpha_1+\ldots+\alpha_n$, where $\{\alpha_1,\ldots,\alpha_n\}\inn S$ has type\footnote{Simple roots are numbered here according to the usual Bourbaki notation \cite{bbki}.} $\sB_n,$ with $n \geq 2$, and $\alpha_i\in \Spp(X)$ for all $i\in\{2,3,\ldots,n\}$,
\item\label{prop:doubling_sr:G} $\sigma = 2\alpha_1+\alpha_2$, where $\{\alpha_1,\alpha_2\}\inn \sr$ has type $\sG_2$,
\item\label{prop:doubling_sr:root} $\sigma$ is not in the root lattice of $G$.
\end{enumerate}
\end{definition}
We observe that  $\Sigma(X)\cap \sr = \Sigma^{sc}(X) \cap \sr$ and introduce the following notation:
\begin{align*}
\A(X,\alpha) &:= \{D \in \col(X) \colon \text{$D$ is moved by $\alpha$}\}\ \text{ for every  $\alpha \in \Sigma(X) \cap \sr$}; \\
\A(X) &:= \cup_{\alpha \in \sr\cap \Sigma(X)}\, \A(X,\alpha).
\end{align*}

\begin{definition}[{\cite{camus}}] \label{def:socle}
Let $X$ be a spherical $G$-variety with a unique closed $G$-orbit $Y$ and write  $\msV_X$ for the set of $G$-stable prime divisors of $X$. 
The \textbf{socle} of $X$ is
\[
\soc(X):=(\sr,\Spp(X), \Sigma^{sc}(X), \A(X), \D_Y, \msV_X, \rho'_X\colon(\D_Y \cup \msV_X) \to (\ZZ\Sigma^{sc}(X))^*)
\]
where $\rho'_X(D):=\rho_X(D)|_{\ZZ\Sigma^{sc}(X)}$ for all $D\in \D_Y \cup \msV_X$.
\end{definition}

Suppose that $(X,\lb)$ is a polarized $G$-variety, and let $Q=Q(X,\lb)$ be its momentum polytope. Recall from Theorem~\ref{thm:Brion-Woodward} and Proposition~\ref{prop:PSVFundFacts}-\ref{prop:PSVFundFacts:item:mompolclosedorbit} that the closed $G$-orbits in $X$ correspond to the orbit vertices $\msv$ of $Q$, and that these correspond to the maximal colored cones $(\C(\msv),\D(\msv))$ in the colored fan of $X$ (see \cite[Lemma 3.2]{knop:LV}).

\begin{definition} \label{def:locsoc}
Let $(X,\lb)$ be a polarized $G$-variety with momentum polytope $Q$, and let $\msv$ be an orbit vertex of $Q$. We define the following:
\begin{enumerate}[1.]
\item $S(\msv):=\{\alpha \in S \colon \text{ if $D \in \col(X)$ and $\alpha$ moves $D$, then $D \in \D(\msv)$}\}.$
\item $\B(\msv):=$ the set of primitive elements in $\Hom_{\ZZ}(\lat(X),\ZZ)$ that lie on extremal rays of $\C(\msv)$ which do not contain any element of $\{\rho_X(D) \colon D \in \D(\msv)\}$. 
\item The \emph{localized socle of $X$ at $\msv$} is
\[
\overline{\soc}(X(\msv)) := (S(\msv), S(\msv) \cap \Spp(Q), \Sigma^{sc}(X)\cap \ZZ S(\msv),\overline{\A(\msv)},\overline{\D(\msv)},\B(\msv) \cup (\D(\msv) \smallsetminus   \overline{\D(\msv)}), \overline{\rho}),
\]
where 
\begin{itemize}
\item $\overline{\A(\msv)} = \bigcup_{\alpha\in S(\msv)\cap \Sigma^{sc}(X)} \A(X,\alpha)$;
\item $\overline{\D(\msv)}=\{D \in \D(\msv) \colon \text{$D$ is moved by some $\alpha$ in $S(\msv)$}\}$; and
\item $\overline{\rho}: (\D(\msv) \cup \B(\msv)) \to [\ZZ(\Sigma^{sc}(X)\cap \ZZ S(\msv))]^*$ is defined by the fact that  $\overline{\rho}(D)$ is the restriction of $\rho_X(D)$ to $\ZZ(\Sigma^{sc}(X)\cap \ZZ S(\msv))$ for every $D \in \D(\msv) \cup \B(\msv)$.
\end{itemize}
\end{enumerate}
\end{definition}

\begin{remark}
Thanks to \cite[Lemma 2.4]{knop:LV}, the set $\B(\msv)$ is (the image in $N(X)$ of) the set of valuations of $G$-stable prime divisors on $X$ that contain the orbit corresponding to $\msv$.
\end{remark}

Since $X$ is smooth if and only if it is smooth along each of its closed orbits, Camus' smoothness criterion, adapted to our setting, reads as follows. 
\begin{proposition}[{\cite{camus}}] \label{prop:smoothness}
Suppose $(X,\lb)$ is a polarized $G$-variety, and let $Q=Q(X,\lb)$ be its momentum polytope. Then $X$ is smooth if and only if for every orbit vertex $\msv$ of $Q$ we have:
\begin{enumerate}[(a)]
\item the $|\D(\msv)\cup\B(\msv)|$-tuple $(\rho_X(D))_{D\in \D(\msv)\cup\B(\msv)}$ is a basis of $\lat(X)^*$; and \label{prop:smoothness:item1}
\item $\overline{\soc}(X(\msv))$ is the socle of a spherical module. \label{prop:smoothness:item2}
\end{enumerate}
\end{proposition}
We recall that the socles of spherical modules are described in \cite[Theorem 3.22]{PVS}, by re-interpreting the combinatorial data in \cite[Section 5]{knop:rmks}.

\begin{remark} \label{rem:toric_smoothness}
\begin{enumerate}[(a)]
\item For each orbit vertex $\msv$, the union of the $G$-orbits in $X$, which contain the orbit corresponding to $\msv$ in their closure, is a $G$-stable open subvariety $X(\msv)$ of $X$ (see, e.g., \cite[Theorem~2.1]{knop:LV}). Condition \ref{prop:smoothness:item1} in Proposition~\ref{prop:smoothness} expresses that $X(\msv)$ is locally factorial (i.e., that every Weil divisor on $X(\msv)$ is Cartier). As is well known (see, e.g. \cite[Section 2.1]{fulton}), a toric variety is smooth if and only if it is  locally factorial. In other words, when $G$ is a torus, condition \ref{prop:smoothness:item2} is superfluous in Proposition~\ref{prop:smoothness} and \ref{prop:smoothness:item1} is just the well-known combinatorial smoothness criterion for toric varieties. More generally, condition  \ref{prop:smoothness:item2} is superfluous in Proposition~\ref{prop:smoothness} under the assumption that $X$ is a \emph{toroidal} spherical variety, see for example \cite[29.2]{timashev}. 
\item We recall that if $X$ is a toric variety, then every vertex of $Q(X,\lb)$ is an orbit vertex, since $\Sigma(G) = \emptyset$ when $G$ is abelian. 
\end{enumerate}
\end{remark}

\begin{remark}\label{rem:combinatorialversions}
The conditions \ref{prop:smoothness:item1} and \ref{prop:smoothness:item2} of Proposition~\ref{prop:smoothness} can be viewed as purely combinatorial conditions on the triple $(\lat(X),Q(X,\lb),\Sigma(X))$, because the data needed to apply the proposition can be combinatorially computed from the triple using standard constructions in the theory of spherical varieties that are due to Luna. 
To make this more precise, let $(\lat,Q,\Sigma)$ be a $\QQ$-momentum triple. The data needed to apply the smoothness criterion are obtained as follows:
\begin{enumerate}[(1)]
\item From the sets $\A(\alpha)$ and the maps $\rho\colon \A(\alpha) \to \lat^*$ of Definition~\ref{def:Q-compatible},  one defines a set $\A = \A(\lat,Q,\Sigma)$ and a map $\rho\colon\A \to \lat^*$ as follows:
\begin{equation} \label{eq:AfromAalpha}
	\A := \left(\coprod_{\alpha\in\Sigma\cap S} \A(\alpha)\right)/\sim,
\end{equation}
	where $\sim$ is the equivalence relation on the above union defined by $D\sim E$ if and only if there exist $\alpha,\beta\in \Sigma\cap S$ with $\alpha\neq\beta$ such that $D\in \A(\alpha)$, $E\in \A(\beta)$, and $\rho(D)=\rho(E)$. Then the given maps $\A(\alpha)\to \lat^*$ glue to a map $\rho\colon \A\to \lat^*$. 
\item The quintuple $(\Spp(Q), \Sigma, \A,\lat,\rho)$ is  a \emph{homogeneous spherical datum} as in \cite[\S 2.2]{luna:typeA} (see also the proof of Lemma~\ref{lemma:homogeneous}). Combinatorially encoding the properties \eqref{eq:color_functional} and \eqref{eq:color_moved_by_two}, \cite[\S 2.3]{luna:typeA} then defines a combinatorial set of colors $\col = \col(\lat,Q,\Sigma)$ associated to $(\Spp(Q), \Sigma, \A,\lat,\rho)$ and extends $\rho$ to a map $\col \to \lat^*$, which we still denote by $\rho$. Furthermore, for every $\beta \in \sr$, one also defines the elements of $\col$  that are \emph{moved} by $\beta$ (see, e.g., \cite[Definition 2.17]{PVS}).
\item For any orbit face $\msv$ of $Q$ (for the cone $\V = (\QQp(-\Sigma))^{\vee}$), the colored cone $(\C(\msv),\D(\msv))$ is then combinatorially defined as in Lemmas \ref{lemma:nD} and \ref{lemma:fan}.
\item Definition~\ref{def:locsoc} then yields the set $\B(\msv)$ and the localized socle $\overline{\soc}(\msv)$. Furthermore, one defines the map
\(
\rho_{\msv}\colon \D(\msv) \cup \B(\msv) \to \lat^*
\)
by $\rho_{\msv}(D)=\rho(D)$ if $D \in \D(\msv)$ and $\rho_{\msv}(n)=n$ for $n\in \B(\msv)$.    
\end{enumerate}
Standard arguments then imply that if $(\lat,mQ,\Sigma) = (\lat(X),Q(X,\lb),\Sigma(X))$ for some polarized spherical $G$-variety $(X,\lb)$ and some $m\in \NN$, then 
\begin{itemize}
\item[-] $\Spp(Q) = \Spp(X)$ by \eqref{eq:aux-equality}, Lemma~\ref{lemma:compare-spherical-roots} and, e.g., \cite[Proposition 5.3(c)]{ACF};
\item[-] the sets $\A(\alpha)$ and $\A(X,\alpha)$ are naturally identified, as are $\rho|_{\A(\alpha)}$ and $\rho_X|_{\A(X,\alpha)}$ (cf.\ Proposition~\ref{proposition:face} and Corollary~\ref{cor:cancellation});
\item[-]$\A(\lat,Q,\Sigma)$ and $\col(\lat,Q,\Sigma)$ are naturally identified with $\A(X)$ and $\col(X)$, respectively, by \cite[Proposition 3.2]{luna:typeA} and $\rho:\col(\lat,Q,\Sigma) \to \lat^*$ is naturally identified with $\rho_X\colon \col(X) \to \lat(X)^*$ by \cite[\S 1.4]{luna:typeA};
\item[-] for every orbit vertex $\msv$ of $Q$, the colored cone $(\C(\msv),\D(\msv))$ is naturally identified with the colored cone in $\F(X)$ associated with the closed orbit in $X$ corresponding to $\msv$ (thanks to Theorem~\ref{thm:Brion-Woodward} and Corollary~\ref{cor:cancellation});
\item[-] $\overline{\soc}(\msv)$ and $\rho_{\msv}$ are naturally identified with $\overline{\soc}(X(\msv))$ and $\rho_X|_{\D(\msv)\cup\B(\msv)}$, respectively.
\end{itemize}
\end{remark}

Given a $\QQ$-momentum triple $(\lat,Q,\Sigma)$ and an orbit vertex $\msv$ of $Q$, Remark~\ref{rem:combinatorialversions}
defines a  map $\rho_{\msv}\colon \mathcal{D}(\msv) \cup \mathcal{B}(\msv) \to \lat^*$ and a localized socle $\overline{\soc}(\mathsf{v})$.

\begin{definition} \label{def:smooth_Q_momentum_triple}
We will call a momentum triple or a $\QQ$-momentum triple $(\lat,Q,\Sigma)$ \textbf{smooth} if for every orbit vertex $\mathsf{v}$ of $Q$, the socle $\overline{\soc}(\mathsf{v})$ and the pair $( \mathcal{D}(\mathsf{v}) \cup \mathcal{B}(\mathsf{v}), \rho_{\mathsf{v}})$ satisfy the following conditions:
\begin{enumerate}[(a)]
\item the $|\D(\msv)\cup\B(\msv)|$-tuple $(\rho_\msv(D))_{D\in \D(\msv)\cup\B(\msv)}$ is a basis of $\lat^*$; and \label{def:smooth_Q_momentum_triple:item1}
\item $\overline{\soc}(\msv)$ is the socle of a spherical module. \label{def:smooth_Q_momentum_triple:item2}
\end{enumerate} 
\end{definition}

Combining Theorem~\ref{thm:classification-triples} with Proposition~\ref{prop:smoothness} we obtain the following classification of smooth polarized spherical $G$-varieties.
\begin{corollary} \label{cor:classif_smooth_polarized}
The map $(X,\lb)\mapsto (\lat(X), Q(X,\lb),\Sigma(X))$
is a bijection between the $G$-isomorphism classes of smooth polarized spherical $G$-varieties and the smooth momentum triples of $G$. 
\end{corollary}

\begin{remark} 
\begin{enumerate}[(a)] 
\item \label{rem:SvinSkv:inclusion} For future use, we remark that if $(\lat,Q,\Sigma)$ is a $\QQ$-momentum triple and $\msv$ is an orbit vertex of $Q$ (for the cone $(-\QQp\Sigma)^{\vee}$), then
\begin{equation}
S(\msv) \inn \{\alpha \in \sr \colon \<\alpha^{\vee},\msv\> = 0\}.\label{eq:SvinSKv}
\end{equation}
This follows from the description of the colored cone $(\C(\msv),\D(\msv))$ in Lemmas \ref{lemma:nD} and \ref{lemma:fan} (or, geometrically, in Theorem~\ref{thm:Brion-Woodward} and Corollary~\ref{cor:cancellation}). To see how, let $\alpha \in S(\msv)$. We may assume that $\alpha \notin \Spp(Q)$. This implies that there is at least one color in $\col(\lat,Q,\Sigma)$ moved by $\alpha$ (see Remark \ref{rem:combinatorialversions}). If $\alpha \notin \Sigma$ and $D$ is the color moved by $\alpha$, then $\rho(D) = \alpha^{\vee}|_\lat$ or $\rho(D)=\frac{1}{2}\alpha^{\vee}|_\lat$ and $\<\rho(D),\msv\>=0$, by Lemma~\ref{lemma:fan}, since $D \in \D(\msv)$. On the other hand, if $\alpha \in \Sigma$, let $D_{\alpha}^+, D_{\alpha}^-$ be the two colors moved by $\alpha$. Since $\alpha \in S(\msv)$, we have that $D_{\alpha}^+, D_{\alpha}^- \in \D(\msv)$ and one deduces, again with Lemma~\ref{lemma:fan}, that $\<\alpha^{\vee},\msv\>=0$. This proves the inclusion~\eqref{eq:SvinSKv}.
\item \label{rem:SvinSkv:local_toric_smoothness} It follows from part \ref{rem:SvinSkv:inclusion} of this remark that if $\msv$ is an orbit vertex of $Q$ which lies in the relative interior of the dominant Weyl chamber, then condition \ref{def:smooth_Q_momentum_triple:item2} of Definition~\ref{def:smooth_Q_momentum_triple} is always satisfied, while condition \ref{def:smooth_Q_momentum_triple:item1} holds if and only if
\begin{equation} \label{eq:local_toric_smoothness}
\{\rho_F\colon  F \text{ is a facet of $Q$ containing $\msv$}\} \text{ is a basis of }\lat^*,
\end{equation}
where $\rho_F \in \lat^*$ is the primitive inward-pointing facet normal to $F$. Let us prove this claim. It follows from \eqref{eq:SvinSKv} that 
\(S(\msv) = \emptyset. 
\) This implies that \[\overline{\soc}(\msv)=(\emptyset,\emptyset,\emptyset,\emptyset,\emptyset,\B(\msv)\cup\D(\msv),0).\] which is the socle of the module $\CC^k$ under the standard action of $(\CC^{\times})^k$, where $k=|\B(\msv) \cup \D(\msv)|$. 
By the definition of $\D(\msv)$, it also follows from \(S(\msv)=\emptyset\) that $\D(\msv)\inn \A$ and that 
\begin{equation}|\D(\msv) \cap \A(\alpha)| \leq 1 \label{eq:DcapA1}
\end{equation}
 for every $\alpha \in \sr \cap \Sigma$. By the definition of $\B(\msv)$, the claim will follow once we show that $\rho(D)$ is a primitive inward-pointing facet normal of $Q$ for every $D \in \D(\msv)$.  Let $D \in \D(\msv)$. Then $D \in \A(\alpha)$ for some $\alpha \in \sr\cap \Sigma$, and $\rho(D)$ is primitive in $\lat^*$ because $\<\rho(D),\alpha\>=1$ and $\alpha \in \lat$. If $\rho(D)$ were not an inward-pointing facet normal of $Q$, then $\rho(D)$ would be a nontrivial positive rational linear combination of primitive inward-pointing normals to facets of $Q$ that contain $\msv$. Consequently, at least one of these normals, say $\rho_1$, would take a positive value on $\alpha$. This would then imply that $\rho_1 = \rho(D')$ for some $D' \in \A(\alpha)$, by Definition~\ref{def:Q-compatible}, and that $D' \in \D(\msv)$, which contradicts \eqref{eq:DcapA1}. This proves the claim.  
\end{enumerate} \label{rem:SvinSKv}
\end{remark}

\begin{example}[{\cite[Chapter 3]{foschi}}] \label{ex:Foschi_smooth} 
We revisit Example~\ref{ex:Foschi}:  take $G = \SL(3)$, $\lat = \ZZ\{\alpha_1,\alpha_2\} = \wl_R$ and $Q = \Conv(\lambda_1,\lambda_2,\lambda_3)$ where $\lambda_1=4\varpi_1+4\varpi_2$, $\lambda_2=5\varpi_1+2\varpi_2$ and $\lambda_3=2\varpi_1+5\varpi_2$. It follows from Example~\ref{ex:Foschi} --- or it can be checked directly from the definition --- that $(\lat,Q,\Sigma)$ is a momentum triple for every subset $\Sigma \inn \{\alpha_1,\alpha_2\}$. 
In fact, the four corresponding projective spherical varieties are smooth as follows from their geometric description in \cite[Chapter 3]{foschi}.  We illustrate Corollary~\ref{cor:classif_smooth_polarized}, and more specifically  Remark~\ref{rem:SvinSKv}\ref{rem:SvinSkv:local_toric_smoothness}, by combinatorially verifying their smoothness. Since all the vertices of $Q$ lie in the interior of the dominant Weyl chamber, the same is true for all its orbit vertices (recall that it depends on $\Sigma$ which of the vertices of $Q$ are orbit vertices). The smoothness of the four momentum triples $(\lat,Q,\Sigma)$, with $\Sigma \inn  \{\alpha_1,\alpha_2\}$, follows from Remark~\ref{rem:SvinSKv}\ref{rem:SvinSkv:local_toric_smoothness} once we verify the claim that \eqref{eq:local_toric_smoothness} holds at every vertex of $Q$.
Let $\{\varepsilon_1,\varepsilon_2\}$ be the basis of $\lat^*$ that is dual to the basis $\{\alpha_1,\alpha_2\}$ of $\lat$ and number the facets of $Q$ as follows:
\[F_1 = \Conv(\lambda_1,\lambda_2), F_2 = \Conv(\lambda_2,\lambda_3), F_3 = \Conv(\lambda_3,\lambda_1)\]
One then computes that the corresponding primitive inward-pointing facet normals of $Q$ are 
\[\rho_1 = -\varepsilon_1,\quad  \rho_2 = \varepsilon_1+\varepsilon_2= (\alpha_1^{\vee}+\alpha_2^{\vee})|_\lat, \quad \rho_3 = -\varepsilon_2.\] Since any subset of $\{\rho_1,\rho_2,\rho_3\}$ with two elements forms a basis of $\lat^*$, the claim and the  smoothness follow. 
\end{example}

\begin{example} \label{ex:bizarre_smooth} 
We illustrate Corollary~\ref{cor:classif_smooth_polarized} and Remark~\ref{rem:combinatorialversions} by showing that the $\Sp(6)$-variety from Example~\ref{ex:bizarre} is not smooth.  Recall that the momentum triple of the variety was $(\lat,Q,\Sigma)$ where $\Sigma = \{\alpha_1+\alpha_3, \alpha_2\}, \lat = \ZZ\Sigma$ and $Q = \Conv(\frac{1}{2}\varpi_2, \varpi_2, \frac{1}{3}(\varpi_1+\varpi_3))$. Let $(\varepsilon_1,\varepsilon_2)$ be the basis  of $\lat^*$ that is dual to the basis $(\alpha_1+\alpha_3, \alpha_2)$ of $\lat$ and number the vertices and facets of $Q$ as follows:
\begin{align*}
&\msv_1 = \frac{1}{2}\varpi_2,\quad \msv_2 = \varpi_2, \quad \msv_3 = \frac{1}{3}(\varpi_1+\varpi_3),\\
&F_1 = \Conv(\msv_2,\msv_3), \quad F_2 = \Conv(\msv_1,\msv_3), \quad F_3 = \Conv(\msv_1,\msv_2).
\end{align*}
One computes that the corresponding primitive inward-pointing facet normals are
\[\rho_1 = -3\varepsilon_1 + \varepsilon_2, \quad \rho_2 = \varepsilon_1, \quad \rho_3 = 2\varepsilon_1 - \varepsilon_2.
\]
The constructions recalled in Remark~\ref{rem:combinatorialversions}  then yield
\begin{itemize}
\item $\A(\lat,Q,\Sigma) = \A(\alpha_2) = \{D_2^+,D_2^-\}$ with $\rho(D_2^+) = \rho_1$ and $\rho(D_2^-) = \rho_2$;
\item $\col(\lat,Q,\Sigma) = \{D_2^+, D_2^-, D\}$ with $\rho(D) = \rho_3 = \alpha_1^\vee|_\lat = \alpha_3^\vee|_\lat$;
\item $\alpha_2$ moves only $D_2^+$ and $D_2^-$, while $\alpha_1$ and $\alpha_3$ both only move $D$.
\end{itemize}
As stated in Example~\ref{ex:bizarre}, the only orbit vertex of $Q$ is $\msv_2$. One checks that 
\[
 \C(\msv_2) = \QQp\{\rho_1,\rho_3\},
\D(\msv_2) =\{D,D^+_2\}, \text{ and }
\B(\msv_2) = \emptyset. 
\]
Because $\{\rho(D),\rho(D^+_2)\} = \{\rho_1,\rho_3\}$ is a basis of $\lat^*$, condition \ref{def:smooth_Q_momentum_triple:item1} of Definition~\ref{def:smooth_Q_momentum_triple} is satisfied. 
On the other hand, 
\begin{align*}
&\overline{\soc}(\msv_2) = (\{\alpha_1,\alpha_3\}, \emptyset,\{\alpha_1+\alpha_3\}, \emptyset, \{D\}, \{D_2^+\}, \overline{\rho}\colon \{D,D_2^+\} \to \ZZ\{\alpha_1+\alpha_3\})\\
&\text{with }\overline{\rho}(D) = \rho_3|_{\ZZ\{\alpha_1+\alpha_3\}} \text{ and }\overline{\rho}(D_2^+) = \rho_1|_{\ZZ\{\alpha_1+\alpha_3\}}.
\end{align*}
Using \cite[Theorem 3.22]{PVS}, we see that $\overline{\soc}(\msv_2)$ is not the socle of a spherical module. Indeed, the only socle in Table 2 of \cite{PVS} which has the same first four entries as $\overline{\soc}(\msv_2)$ (up to isomorphism of socles) is socle $\# 6$. On the other hand, because $\<\overline{\rho}(D_2^+),\alpha_1+\alpha_3\> = -3 \neq -1$, the  socle  $\overline{\soc}(\msv_2)$ is not isomorphic to socle $\# 6$ of \cite[Table 2]{PVS}. Condition \ref{def:smooth_Q_momentum_triple:item2} of Definition~\ref{def:smooth_Q_momentum_triple} is not satisfied and we have shown that the momentum triple $(\lat,Q,\Sigma)$ is not smooth,  that is, the corresponding variety is singular along the orbit corresponding to $\msv_2$. 
\end{example}

\section{K\texorpdfstring{\"a}{ae}hlerizability of multiplicity free Hamiltonian manifolds} \label{sec:Kaehler}

Throughout this section, we are concerned with compact and connected multiplicity free Hamiltonian manifolds. 
We obtain a necessary and sufficient condition for these manifolds to 
admit a K\"ahler structure.

\subsection{Basics and recollections}
Given a connected compact Lie group $K$ with a maximal torus $T_\RR$, we denote their respective complexifications by $G$ and $T$. Let $\ft_{\RR}$ be the Lie algebra of $T_{\RR}$ and  $\ft_\RR^*=\Hom_{\RR}(\ft_{\RR},\RR)$, which we identify with the space of $T_{\RR}$-invariant vectors in $\fk^*=\Hom_{\RR}(\fk,\RR)$, where $\fk$ is the Lie algebra of $K$. We identify the weight lattices $\wl=\Hom(T,\CC^{\times})$ and $\Hom(T_{\RR},\Uni(1))$ by the restriction map $\Hom(T,\CC^{\times}) \to \Hom(T_{\RR},\Uni(1)) \colon \lambda \mapsto \lambda|_{T_{\RR}}$, and we view $\wl$ as a sublattice of $\ft_\RR^*$ by mapping $\gamma \in \Hom(T_{\RR},\Uni(1))$ to $(2\pi i)^{-1}\gamma_* \in \ft^*_{\RR}$, where $\gamma_*\colon \ft_{\RR} \to i\RR$ is the differential of $\gamma$. Fix a Borel subgroup $B\subseteq G$ containing $T$ and let $\mathfrak t^*_+\inn \ft_\RR^*$ be the corresponding positive Weyl chamber. The identifications above then imply that $\wl_{\RR} = \ft^*_{\RR}$, $\dw = \ft^*_+ \cap \wl$,  and $\RR_{\ge 0}\dw = \ft^*_+$. 

Let $(M,\omega,\Phi)$ be a Hamiltonian $K$-manifold that is, a symplectic manifold $(M,\omega)$ equipped with a smooth action of $K$ and a momentum map  $\Phi:M\rightarrow \mathfrak  k^*$. Let $\Pc(M)=\Pc(M,\omega,\Phi)$ be the \emph{Kirwan set of $M$} that is, 
\[
\Pc(M):=\Pc(M,\omega,\Phi):=\Phi(M)\cap \mathfrak t_+^*. 
\]
By a theorem of F. Kirwan's in  \cite{kirwan}, $\Pc(M)$ is a convex polytope whenever $M$ is compact and connected; in this case, we call $\Pc(M)$ the \textbf{Kirwan polytope} of $(M,\omega,\Phi)$.

As usual, we will write $[\wom]$ for the class in the second $K$-equivariant cohomology group $H^2_K(M,\RR)$  of the equivariantly closed form
\[\widetilde{\omega} := \omega + \Phi.\]
Moreover, $\widetilde{\omega}$ is called an \emph{equivariant K\"ahler form} if $\omega$ is K\"ahler, in which case the Kirwan polytope $\Pc(M,\omega,\Phi)$ depends only on the cohomology class of $\wom$ in $H^2_K(M,\RR)$ (see e.g.~\cite{bgh14} for a detailed proof of this fact). 
Henceforth, we thus denote the polytope $\Pc(M,\omega,\Phi)$ simply by $\Pc(M,[\widetilde{\omega}])$ whenever $\omega$ is K\"ahler.

From now on, a \textbf{multiplicity free $K$-manifold} is a compact connected Hamiltonian $K$-manifold $(M,\omega,\Phi)$  such that the map $M/K \to \Pc(M)\colon K\cdot m \mapsto \Phi(K\cdot m)\cap \ft_+^*$ is a homeomorphism. 

\emph{In the remainder of this subsection, $M=(M,\omega,\Phi)$ is a multiplicity free $K$-manifold.}

Let $L_0$ be the principal isotropy group of $M$, that is, the stabilizer in $K$ of a point  $x\in \Phi^{-1}(\eta)$, where $\eta$ is a generic element of $\Pc(M)$, and let $L$ be the stabilizer of $\eta$ in $K$. Then $A_\RR:=L/L_0$ is abelian. Following Knop (\cite{knop:hamilton}), we define \emph{the weight lattice of $M$} as the following group
\begin{equation} \label{eq:lattice_M}
\lat(M):=\mathrm{Hom}(A_\RR, \CC^\times),
\end{equation}
regarded as a sublattice of the weight lattice $\wl$.  
Since  $\lat(M)$ contains the same information as $L_0$ (see pages 570-571 in~\cite{knop:hamilton}), Theorem 10.2 in \loccit, specialized to the compact case,
can be stated as follows.
\begin{theorem}[Knop]\label{thm:knop-uniqueness}
A multiplicity free $K$-manifold is uniquely determined by its Kirwan polytope together with its weight lattice.
\end{theorem}

In case   $(M,\omega)$  has a $K$-invariant complex structure $J$ compatible with $\omega$, i.e. $(M,\omega,J)$ is K\"ahler, the $K$-action on $M$ extends to a unique action of $G$ by holomorphic automorphisms. Moreover, we have the following theorem of Huckleberry and Wurzbacher specialized to the compact case (see also \cite[Corollary 5.3]{woodward} for a different proof).

\begin{theorem}\cite[Equivalence Theorem in \S6]{huckwurz}\label{thm:huck-wurz}
If $(M,\omega,J)$ is K\"ahler then the complex $G$-manifold  $(M,J)$ is $G$-equivariantly biholomorphic to a projective spherical $G$-variety. 
\end{theorem} 

We denote the projective spherical $G$-variety  in the above theorem by $X(J)$.
It is well-known to experts that then
\begin{equation}
\lat(M) = \lat(X(J)), \label{eq:lattice_equality}
\end{equation}
where $\lat(M)$ is defined in \eqref{eq:lattice_M} and $\lat(X(J))$ in Section~\ref{subsec:polarized_varieties} (for a proof see, e.g., \cite[Proposition 8.6]{losev:knopconj}.)

Recall from Section~\ref{subsec:polarized_varieties} that there is a standard way to assign an equivariant K\"ahler class $[\wom_{\lb}]$ on (the real manifold) $X$  to an ample $G$-line bundle $\lb$ on (the variety) $X$. 

\begin{lemma}[{\cite[Lemmas 8.4-8.5]{losev:knopconj}}]\label{lemma:Losev-cohomology}
Let $X$ be a smooth projective spherical $G$-variety.  
\begin{enumerate}
\item The unique homomorhism $\varphi \colon \Pic_G (X)_\RR \to H^2_K(X,\RR)$ which sends a very ample line bundle $\lb \in \Pic_G(X)$ to $[\wom_{\lb}]$ is an isomorphism. \label{lemma:Losev-cohomology-isom}
\item The dependence of the Kirwan polytope $\Pc(X,[\wom])$ on $[\wom] \in H^2_K(X,\RR)^+$ is continuous. 
\label{lemma:Losev-cohomology-depcont}
\end{enumerate}
\end{lemma}

We also recall the basic fact that, for $X$ as in the previous lemma, the subset $H^2_K(X,\RR)^+ \subseteq H^2_K(X,\RR)$ consisting of the classes of equivariant K\"ahler forms is open. 

The following lemma is standard. We give a proof here for completeness. 
\begin{lemma}\label{lemma:amplesequence}
Suppose the projective spherical $G$-variety $X$ is equipped with an equivariant K\"ahler form $\widetilde\omega=\omega+\Phi$ and let $\varphi$ be as in Lemma~\ref{lemma:Losev-cohomology}. Then there exists a sequence of equivariant K\"ahler forms $\widetilde\omega_n=\omega_n+\Phi_n$ (for $n\in\NN$) whose classes converge to $[\widetilde\omega]$ in $H^2_K(X,\RR)$, and such that for all $n\in\NN$ a  positive integral multiple of $\varphi\inv([\widetilde\omega_n])$ is an ample linearized line bundle on $X$.
\end{lemma}

\begin{proof}
Recall that the isomorphism $\varphi$ fits into the commutative diagram
\begin{equation} \label{eq:H2}
\begin{CD}
 \Pic_G (X)_\RR @>>{f}>\Pic(X)_\RR\\
@VV{\varphi}V @VV{\Psi}V \\
H^2_K(X,\RR)@>{g}>>
H^2(X,\RR)
\end{CD}
\end{equation}
where  $\Psi$ sends a line bundle to its first Chern class, and the natural maps $f$ and $g$ are surjective.
We can take a sequence of elements $[\widetilde\omega_n]$ (for $n\in\NN$) converging to the class $[\widetilde\omega]$ in $H^2_K(X,\RR)$, such that $\varphi\inv([\widetilde\omega_n])\in \Pic_G(X)\otimes_\ZZ\QQ$. Since $H^2_K(X,\RR)^+ $ is an open subset of $H^2_K(X,\RR)$, we may assume that $\wom_n$ is an equivariant K\"ahler form $\om_n+\Phi_n$ for all $n\in\NN$. Now, some positive integral multiple $k_n[\widetilde\omega_n]$ of $[\widetilde\omega_n]$ corresponds to a $G$-linearized line bundle $\mathcal L_n$ on $X$. Moreover $\mathcal L_n$ is a positive line bundle, since $\omega_n$ is positive and $c_1(\mathcal L_n)=[\omega_n]$. By the Kodaira embedding theorem, we have that $\mathcal L_n$ is ample.
\end{proof}

The description of the momentum polytope of a  polarized spherical variety as an intersection of half-spaces (Proposition~\ref{prop:Pbrion}) can be extended to K\"ahler multiplicity free  $K$-manifolds. 
Following~\cite{losev:knopconj}, let us review this description.
Assume $(M,\omega,\Phi)$ admits a compatible complex structure $J$ and denote by $X=X(J)$ the projective spherical $G$-variety to which $M$ is equivariantly biholomorphic. Then one can view $\omega + \Phi$ as an equivariant K\"ahler form on $X$.  

Recall  from Section~\ref{subsec:mom_pol_inters_halfspaces} the definition of $\rho(D)\in N(X) = \Hom_\ZZ(\lat(X),\QQ)$ for $D\in\divB(X)$, and embed $\lat(X)$ into $\wl \times \ZZ^{\divB(X)}$ by $\lambda \mapsto (\lambda,\sum_{D \in \divB(X)} \<\rho(D),\lambda\> D)$. Extending $\rho(D)$ to $\lat(X)_\RR$ by linearity, we obtain an element of $N(X)_\RR$, still denoted by $\rho(D)$ below. Given $\lb \in \Pic_G(X)$ and a $B$-semi-invariant rational section $s$ of $\lb$, let $\Chi(s)$ be the $B$-weight of $s$. Thanks to the isomorphism $\varphi$ of Lemma~\ref{lemma:Losev-cohomology}, the assignment $\Pic_G(X) \to (\wl \times \ZZ^{\divB(X)})/\lat(X) \colon \lb\mapsto (\chi(s),\sum_{D\in \divB(X)}v_D(s)D)$ induces an $\RR$-linear map:
\begin{equation} \label{eq:chiH2KX}
\overline{\chi}:H^2_K(X,\RR)\to(\wl_\RR\times \RR^{\divB(X)})/\lat(X)_\RR.
\end{equation}
Using part (\ref{lemma:Losev-cohomology-depcont}) of Lemma~\ref{lemma:Losev-cohomology} we get the extension of Proposition~\ref{prop:Pbrion}.

\begin{proposition}[Losev]\label{prop:Kaehler-polytope}
Let $\wom$ be an arbitrary $K$-equivariant K\"ahler form on a projective spherical $G$-variety $X$ and let $(\chi_0,\sum_{\mathrm{div}^B(X)}\chi_D D)\colon H^2_K(X,\RR) \to (\wl_\RR\times \RR^{\divB(X)})$ be a lifting of $\overline{\chi}$ in \eqref{eq:chiH2KX}.
Then the Kirwan polytope $\Pc(X,[\wom])$
satisfies the following equality
\[
\Pc(X,[\wom])=\chi_0([\widetilde\omega]) + \{ \xi\in \lat(X)_\RR \;|\; \<\rho(D),\xi\> + \chi_D([\widetilde\omega]) \geq 0 \text{ for all $D\in \divB(X)$}\}.
\]
\end{proposition}

Next we recall, and slightly generalize, a result from \cite{foschi}; we
deduce it from the intermediate generalization in \cite[Lemma 30.24]{timashev}.

\begin{lemma}[\cite{foschi}]\label{lemma:foschi}
Let $X$ and $\wom$ be as in Proposition~\ref{prop:Kaehler-polytope}. For $\alpha\in S$, set $\varepsilon_\alpha=\frac12$ if $2\alpha\in \Sigma(X)$, and $\varepsilon_\alpha=1$ otherwise. Then we have
\[
\varepsilon_\alpha\langle\alpha^\vee,\chi_0([\wom])\rangle = \sum_{D} \chi_D([\wom])
\]
where the sum is taken over all colors $D$ of $X$ moved by $\alpha$.
\end{lemma}
\begin{proof}

First, we may assume that $[\wom]$ is in $\Pic_G(X)$, indeed then the general case follows by linearity. Let $X_0\subseteq X$ be the open $G$ orbit, and identify $X_0$ with a homogeneous space $G/H$. Let $\pi\colon G\to G/H$ the quotient map, and consider the pull-back $\mathcal L$ of $[\wom]$ on $G$. We may assume that $G$ is factorial so that this pull-back is trivial as a line bundle (possibly it has a different linearization than the standard one on $\mathcal O_G$).

Recall that the definitions of $\chi_0$ and of $\chi_D$ for all colors $D$ of $X$ depend on the choice of a lift of $\overline\chi$. Again by linearity, we may assume that the image of $[\wom]$ via this lift is $(\chi(s),\sum_{D\in \divB(X)}v_D(s)D)$ for some rational section $s$. Then $\chi_0([\wom])$ is the $B$-weight of the pull-back $\widetilde s$ of $s|_{X_0}$ to $G$, and $\chi_D([\wom])$ is the order of vanishing of $\widetilde s$ along the pull-back $\widetilde D$ of $D\cap X_0$ to $G$.

Fix an isomorphism of invertible sheaves $\varphi\colon \mathcal L\to \mathcal O_G$, and consider the rational function $f=\varphi(\widetilde s)$. It is a $B$-eigenvector under the action of left translation of $G$, and its $B$-eigenvalue differs from that of $\widetilde s$ by a character of $G$, so the two pair equally on any simple coroot. On the other hand $\widetilde s$ and $f$ have the same order of vanishing on all divisors of $G$.

For any color $D$ of $X$, fix a global equation $f_D\in\CC[G]$ of the divisor $\widetilde D$. It is a $B$-eigenvector. Since $\widetilde s$ is pulled-back from $X_0$, we can write
\[
f = a\prod_{D\in\Delta(X)} f_D^{\textrm{ord}_D(f)}
\]
where $a\in\CC[G]$ is invertible and a $B$-eigenvector. Notice that the $B$-eigenvalue of $a$ vanishes on all simple coroots of $G$. So, to finish the proof, it is enough to apply the following fact from \cite[Lemma 30.24]{timashev}: the $B$-eigenvalue of $f_D$ paired with $\alpha^\vee$ is equal to $\varepsilon_\alpha\inv$ if $\alpha$ moves $D$, and $0$ otherwise.
\end{proof}

\begin{corollary}\label{cor:cancellation}
Let $X$, $\wom$ be as in Proposition~\ref{prop:Kaehler-polytope} and set $\Pc = \Pc(X,[\wom])$. Let $\alpha\in S$ be a simple root and $p\in \Pc$, and define $\varepsilon_\alpha$ as in Lemma~\ref{lemma:foschi}. Then
\[
\sum_{D} (\<\rho(D),p-\chi_0([\wom])\> + \chi_D([\wom])) = \varepsilon_\alpha\<\alpha^\vee,p\>,
\]
where the sum is taken over all colors $D$ moved by $\alpha$.
\end{corollary}

\begin{proof}
We observe that $\varepsilon_\alpha \alpha^\vee|_{\lat(X)}$ is equal to
the sum of the $\rho(D)$'s taken over all colors $D$ moved by $\alpha$.
The corollary is now a direct consequence of Lemma~\ref{lemma:foschi}.
\end{proof}

We end this subsection with a generalization of Theorem~\ref{thm:Brion-Woodward}, after recalling an application of
Sjamaar's Local Convexity Theorem to compact K\"ahler $G$-manifolds and their compact $G$-orbits (hence K\"ahler $K$-orbits); see Theorem 6.5 and Equation (6.9) in~\cite{sjamaar}.

\begin{lemma}\label{lemma:slice}
Let $X$ be a smooth projective spherical $G$-variety and let $Z$ be a closed $G$-orbit of $X$. Given $[\wom]\in H^2_K(X,\RR)^+$, set $\Pc = \Pc(X,[\wom])$.
Let $x$ be a point on $Z$ such that its stabilizer $G_x$ contains $B$ and let $\msv$ be the unique point in $\ft^*_+$ contained in $K\cdot \Phi(x)$. Let $L$ be the Levi subgroup of the stabilizer $G_x$ such that $L\supseteq T$, and $W=T_xX/T_xZ$. Then $K_x=K_\msv$, the action of $K_x$ on $W$ extends to an action of $L$, and the convex cone generated by $\Pc(X,[\widetilde \omega])-\msv$ is the convex cone generated in $\wl_\RR$ by the weight monoid of $W$ as an $L$-variety.
\end{lemma}

Since every facet of the polytope $\Pc$ (as in Lemma~\ref{lemma:slice}) admits a normal vector in $N(X)=\Hom_{\ZZ}(\lat(X),\QQ)$, the normal fan of $\Pc$ is defined in the rational vector space $N(X)$ 
and its elements are finitely generated convex cones.
 
Our generalization of Theorem~\ref{thm:Brion-Woodward} is the following Proposition.

\begin{proposition}\label{thm:Brion-Woodward-gen}
Let $X$ be a smooth projective spherical $G$-variety 
and let $[\widetilde\omega]\in H^2_K(X,\RR)^+$. 
Set $\Pc=\Pc(X,[\widetilde\omega])$. 
The colored fan $\F(X)$ of $X$ is the set of pairs $(\C(F),\D(F))$ where $F$ varies over the orbit faces of $\Pc$ and for each such face $F$
\begin{equation}
\D(F) := \{D \in \col(X) \;|\;  \<\rho(D), p-\chi_0([\widetilde\omega])\> + \chi_D([\widetilde\omega]) =0 \text{ for all $p \in F$}\}. 
\end{equation}
\end{proposition}

\begin{proof}
To avoid confusion, let us call for the moment $(\C'(F), \D'(F))$ instead of $(\C(F),\D(F))$ the couple defined in the theorem, for $F$ an orbit face of $\Pc$. If $F$ is an orbit face of the momentum polytope of an ample $\QQ$-line bundle on $X$, we denote as usual by $(\C(F), \D(F))$ the colored cone of $X$ corresponding to $F$.

Let $\widetilde\omega_n=\omega_n+\Phi_n$ be a sequence of equivariant K\"ahler forms as in Lemma~\ref{lemma:amplesequence}, and recall the polytopes $\Pc_n=\Pc(X,[\widetilde\omega_n])$ which converge to $\Pc$. We notice that the theorem holds for $\widetilde\omega_n$ and $\Pc_n$, because in this case it is just Theorem~\ref{thm:Brion-Woodward}.

To show the theorem for $\widetilde\omega$, we denote by $\F$ the set of the couples $(\C'(F),\D'(F))$ with $F$ as above. We observe that $\F$ satisfies properties (CC1), (CF1), (CF2) of colored fans, by construction, and also property (CC2), by the definition of orbit faces. It follows that it is enough to prove that $\F$ and $\F(X)$ have the same elements that are maximal with respect to taking faces.

Let $Z$ be a closed $G$-orbit of $X$. For every $n$, let $\msv_n$ be the unique point in $\Phi_n(Z)\cap \ft^*_+$, and let $\msv$ be the unique point in $\Phi(Z)\cap \ft^*_+$. We know that $\msv_n$ is an orbit vertex of $\Pc_n$. By Lemma~\ref{lemma:slice}, the convex cones generated by $\Pc_n-\msv_n$ and generated by $\Pc-\msv$ are equal. We deduce that $\C(\msv_n)=\C'(\msv)$ and that $\msv$ is an orbit vertex of $\Pc$ (of course neither $\C(\msv_n)$ nor $\D(\msv_n)$ depend on $n$).

This argument, applied to all closed $G$-orbits of $X$, shows that the first component of any maximal colored cone of $\F(X)$ appears as the first component of some maximal element in $\F$. Since $X$ is complete, the maximal cones in $\F(X)$ cover the entire valuation cone of $X$, which implies that all maximal elements of $\F$ must appear in this way. In other words $(\C(\msv_n),\D(\msv_n))\mapsto (\C'(\msv),\D'(\msv))$, for $Z$ varying in the set of closed $G$-orbits, is a bijection between the maximal colored cones of $\F(X)$ and the maximal elements of $\F$, preserving the first component. It remains to show that it preserves the second component too.

Fix a closed $G$-orbit $Z$ in $X$ and let $\msv$ and $\msv_n$ be as above. Let $D\in\Delta(X)$, and suppose $D$ is not moved by any simple root in $\Sigma(X)$. Let $\alpha$ be a simple root moving $D$. Then the equality $\<\rho(D), \msv-\chi_0([\widetilde\omega])\> + \chi_D([\widetilde\omega]) =0$ is equivalent to $\<\alpha^\vee, \msv\> =0$, by Corollary~\ref{cor:cancellation}. Thanks to the equality $K_x=K_\msv$ in Lemma~\ref{lemma:slice}, those $\alpha$ satisfying this equality are the same as the simple roots associated with the parabolic subgroup $P$ that is the stabilizer of the point in $Z$ fixed by $B$. The same holds for the equality $\<\alpha^\vee, \msv_n\> =0$ and it follows that $D\in\D'(\msv)$ if and only if $D\in\D(\msv_n)$.

Finally, let us consider a simple root $\alpha\in S\cap\Sigma(X)$, and the two colors $D^+,D^-$ moved by $\alpha$. For a point $p$ in $\Pc$, the two equalities $\<\rho_{D^\pm}, p-\chi_0([\widetilde\omega])\> + \chi_{D^\pm}([\widetilde\omega]) =0$ are simultaneously satisfied if and only if $\<\alpha^\vee,p\>=0$, again by Corollary~\ref{cor:cancellation}. So both $D^+$ and $D^-$ are in $\D(\msv_n)$ if and only if they are both in $\D'(\msv)$.

On the other hand, if a color in $D\in\{D^+,D^-\}$ satisfies $D\in\D(\msv_n)$ then $\<\rho_{D}, \msv_n-\chi_0([\widetilde\omega_n])\> + \chi_{D}([\widetilde\omega_n]) =0$, and by continuity we deduce $\<\rho_{D}, \msv-\chi_0([\widetilde\omega])\> + \chi_{D}([\widetilde\omega]) =0$, hence $D\in\D'(\msv)$.

We are left with the case where $D\notin\D(\msv_n)$, and we must show that $D\notin\D'(\msv)$. Up to switching $D^+$ and $D^-$, we may assume $D=D^+$. Suppose for sake of contradiction that $D^+\in\D'(\msv)$. If $D^-\in \D(\msv_n)$ then $D^-\in\D'(\msv)$ as we have seen, so both $D^+$ and $D^-$ must be in $\D(\msv_n)$: contradiction. 
Hence we may assume that $D^-\notin \D(\msv_n)$. As a consequence, the cone $C=\RR_{\geq0}(\Pc_n-\msv_n)$ is defined in $\lat(X)_\RR$ by inequalities of the form $\<\rho_E,-\> \geq 0$ for $B$-stable prime divisors $E$ of $X$ such that $\<\rho_E,\alpha\> \leq 0$. Therefore $-\alpha\in C$. Since $D^+\in\D'(\msv)$, for all $p\in \Pc-\msv$ we have $\<\rho_{D^+},p\> \geq 0$. This holds then for all $p\in C=\RR_{\geq0} (\Pc-\msv)$, in particular for $-\alpha$. This contradicts $\<\rho_{D^+},\alpha\> = 1$, and the proof is complete.
\end{proof}

\begin{remark}
Proposition~\ref{thm:Brion-Woodward-gen} shows that the features of the Kirwan polytope $\Pc(X,[\wom])$ that are relevant in the computation of the colored fan $\F(X)$ do not change when we vary $[\wom]$. In particular, the number of orbit vertices is preserved, as well as the poset structure of the orbit faces and their inclusions.   We underline, however, that this is not true for other faces. For example, \cite[Example 3.1]{foschi} exhibits a projective spherical $\SL(3)$-variety $Y$ and ample $G$-linearized bundles $\lb_1, \lb_2$ and $\lb_3$ on $Y$ such that $Q(Y,\lb_1)$ is a regular hexagon, $Q(Y,\lb_2)$ is a pentagon and $Q(Y,\lb_3)$ is a trapezoid. Only near their unique orbit vertex do these three polytopes have the same shape.
\end{remark}

\begin{lemma}\label{lemma:Sp}
Let $X$ and $\Pc$ be as in Proposition~\ref{thm:Brion-Woodward-gen}. Then $S^\perp(\Pc) = S^\perp(X)$.
\end{lemma}
\begin{proof}
In the proof of Proposition~\ref{thm:Brion-Woodward-gen} we have shown that there exist an ample line bundle $\lb$ on $X$ and vertices $\msv\in \Pc$, $\msv'\in Q(X,\lb)$ such that $S^\perp(\msv)=S^\perp(\msv')$. This implies $S^\perp(\Pc)=S^\perp(Q(X,\lb))$, because both polytopes are full-dimensional in the affine space spanned by $\msv+\Xi(X)$ resp.\ $\msv'+\Xi(X)$. The assertion now follows from Lemma~\ref{lemma:compare-spherical-roots}.
\end{proof}

\subsection{Existence of K\texorpdfstring{\"a}{ae}hler structures} \label{sec:kaehler_existence}
In this section, we state our criterion, in Theorem~\ref{thm:criterion-Kaehler}, for a given multiplicity free compact Hamiltonian manifold to be 
K\"ahlerizable. Together with Losev's uniqueness result \cite[Theorem 8.3]{losev:knopconj}, it yields a combinatorial classification of compatible complex structures on a given multiplicity free manifold, cf.\ Corollary~\ref{cor:classif_of_complex}.

We first introduce the necessary combinatorial notions. 
Throughout this section, $\lat$ denotes a sublattice of $\wl$ and $\Pc$ a convex polytope in $\ft^*_+$
satisfying the following properties:
\smallbreak
\noindent\enspace (Q1)
$\Pc-\mathsf{w}$ is a full dimensional polytope in  $\lat_{\RR}$ for some (equivalently, any) $\mathsf{w}\in \Pc$ and
\smallbreak  \noindent
\enspace (Q2) the inward-pointing facet normals of $\Pc$ may be chosen to be  rational, i.e. they can be chosen in $\Hom_\ZZ(\lat, \QQ)$. 

\begin{remark} \label{rem:KPsatisfiesQ1Q2}
The Kirwan polytope of any compact connected Hamiltonian manifold 
satisfies (Q1) and (Q2) by \cite[Theorem 6.5]{sjamaar}.
\end{remark}

In the following, when $F$ is a facet of $\Pc$, we denote by $\rho_F$ the unique primitive element in $\Hom_{\ZZ}(\lat,\ZZ)$ which is an inward-pointing facet normal to $F$. 

Fixing $\mathsf{w} \in \Pc$, we thus have,
\begin{equation}
\Pc=\mathsf{w}+\left\{\xi\in\lat_{\RR}: \left\langle\rho_F,\xi\right\rangle + m_{F,\mathsf{w}} \geq0 \mbox{ for all facets $F$ of $\Pc$}\right\}, \label{eq:Pc_and_mFw}
\end{equation}
where, for every facet $F$ of $\Pc$, $m_{F,\mathsf{w}}$ is a real (not necessarily rational) number. For each such $F$, we also let
\begin{equation}
H_F^{\RR}:= \mathsf{w} + \{\xi \in \lat_{\RR} \colon \<\rho_F,\xi\>+m_{F,\mathsf{w}}=0\}
\end{equation}
be the affine subspace of $\wl_{\RR}$ it generates.

Our K\"ahlerizability criterion involves the notion of a \emph{smooth $\RR$-momentum triple}, which extends the notion of smooth momentum triple of Definition~\ref{def:smooth_Q_momentum_triple}.
We first extend the notions of $\QQ$-compatibility and $\QQ$-admissibility from Definitions \ref{def:Q-compatible} and \ref{def:Q-admissibleset}.
As these notions do not involve the vertices of the polytope under consideration, 
we can  extend their definitions verbatim to any real polytope as above in order to obtain the notions of  $\RR$-compatibility and $\RR$-admissibility.

\begin{definition}\label{def:R-compatible}
A spherical root $\sigma\in\Sigma(G)$ is \textbf{$\RR$-compatible with $(\lat,\Pc)$} if $\sigma$ is compatible with $\lat$, the couple $(\Spp(\Pc),\sigma)$ satisfies Luna's axiom (S), and $\sigma$ satisfies the following properties:
\begin{enumerate}
\item\label{R-compatible-b} if $\sigma\notin S$ and a facet $F\subseteq \Pc$ satisfies $\left\langle\rho_F,\sigma\right\rangle>0$, then there exists $\alpha\in S\smallsetminus \Spp(\Pc)$ 
such that $\<\alpha^\vee, F\> =0$.
\item \label{R-compatible-a} if $\sigma=\alpha\in S$ then there exists a facet $F\subseteq \Pc$ such that
\begin{enumerate}
\item $\left\langle\rho_F,\alpha \right\rangle =1$;
\item if $F'\subseteq \Pc$ is a facet such that $\left\langle\rho_{F'},\alpha \right\rangle>0$ then $H_{F'}^\RR=H_{F}^{\RR}$ or $H_{F'}^{\RR}=s_\alpha (H_{F}^{\RR})$.
\end{enumerate}
\item \label{R-compatible-d2} if $\sigma = \alpha + \beta$ or $\sigma = \frac{1}{2}(\alpha + \beta)$ for two orthogonal simple roots $\alpha$ and $\beta$, then $\<\alpha^{\vee},q\> = \<\beta^{\vee},q\>$ for all $q \in \Pc$.
\end{enumerate}
By $\Sigma_\RR(\lat, \Pc)$, we denote the set of spherical roots that are $\RR$-compatible with $(\lat, \Pc)$.

For a simple root $\alpha\in \Sigma_\RR(\lat,\Pc)$, we denote by $\A(\alpha)$ an ``abstract'' set with two elements $D_\alpha^+$ and $D_\alpha^-$, 
and we define a map $\rho\colon \A(\alpha)\to \Hom_\ZZ(\lat,\ZZ)$ 
by setting $\rho(D_\alpha^+)=\rho_F$ and $\rho(D_\alpha^-)=\alpha^\vee|_{\lat}-\rho_F$, where we choose\footnote{From now on we implicitly fix a choice of such a face for all $\alpha\in\Sigma_\RR(\lat,\Pc)$.}
a facet $F\subseteq \Pc$ as above.
\end{definition}

\begin{definition}\label{def:R-admissibleset}
A subset $\Sigma\subseteq \Sigma_\RR(\lat, \Pc)$ is $\RR$-\textbf{admissible} (for $(\lat,\Pc)$) if it satisfies condition (\ref{A1polytope})  of Definition~\ref{def:Q-admissibleset}. In that case, $(\lat,\Pc,\Sigma)$ is called an \textbf{$\RR$-momentum triple.} 
\end{definition}

Suppose that $\Pc$ is a rational polytope in $\RR_{\ge 0} \dw$, that is, suppose that all its vertices belong to $\QQp \dw$. If $\sigma$ is an element of $\Sigma(G)$ and  $\Sigma$ is a subset of $\Sigma(G)$, then it follows immediately from the definitions that
\begin{itemize}
\item $\sigma$ is $\RR$-compatible with $(\Pc,\lat)$ if and only if $\sigma$ is $\QQ$-compatible with $(\Pc \cap \QQp \dw, \lat)$;
\item $\Sigma$ is $\RR$-admissible for $(\Pc,\lat)$ if and only if $\Sigma$ is $\QQ$-admissible with $(\Pc \cap \QQp \dw, \lat)$.
\end{itemize}

Since the rationality of vertices plays no role in Definition~\ref{def:smooth_Q_momentum_triple}, we can also extend it to our current setting. As in Section~\ref{sec:smoothness}, it follows from Remark~\ref{rem:combinatorialversions} that, given an $\RR$-momentum triple $(\lat,\Pc,\Sigma)$ and an orbit vertex $\mathsf{v}$ of $\Pc$, one can combinatorially define an associated socle $\overline{\soc}(\mathsf{v})$ and map 
$\rho_{\mathsf{v}}\colon \mathcal{D}(\mathsf{v}) \cup \mathcal{B}(\mathsf{v}) \to \lat^*$. 

\begin{definition} \label{def:smooth_R_momentum_triple}
We will call an $\RR$-momentum triple $(\lat,\Pc,\Sigma)$ \textbf{smooth} if for every orbit vertex $\mathsf{v}$ of $\Pc$, the socle $\overline{\soc}(\mathsf{v})$ and the pair  $( \mathcal{D}(\mathsf{v}) \cup \mathcal{B}(\mathsf{v}), \rho_{\mathsf{v}})$ satisfy conditions \ref{def:smooth_Q_momentum_triple:item1} and \ref{def:smooth_Q_momentum_triple:item2} of Definition~\ref{def:smooth_Q_momentum_triple}. 
\end{definition}

We can now give our criterion for a multiplicity free manifold to be K\"ahlerizable. The proof relies on Theorems \ref{thm:existence-Kaehler-converse} and \ref{thm:existence-Kaehler}, which will be proved in Sections \ref{sec:kaehler-necessity} and \ref{sec:kaehler_sufficient}, respectively.  

\begin{theorem}\label{thm:criterion-Kaehler}
Suppose $M$ is a multiplicity free $K$-manifold with Kirwan polytope $\Pc$ and weight lattice $\lat$. Then $M$ has a compatible complex structure if and only if there exists a subset $\Sigma$ of $\Sigma(K^{\CC})$ such that $(\lat,\Pc,\Sigma)$ is a smooth $\RR$-momentum triple. 
\end{theorem}

\begin{proof}
We first show the ``if'' statement. Let $(\lat,\Pc,\Sigma)$ be a smooth $\RR$-momentum triple.
By Theorem~\ref{thm:existence-Kaehler} below, there exists a smooth projective spherical $G$-variety $X$ with weight lattice $\lat$ as well as an equivariant K\"ahler form $\wom_X$ on the real manifold $X$ such that $\Pc(X,[\wom_X])=\Pc$. 
By Knop's Theorem~\ref{thm:knop-uniqueness}, the equality between the two polytopes and the two lattices (by \eqref{eq:lattice_equality}) implies that the multiplicity free $K$-manifolds $(X,\wom_X)$ and $(M,\wom)$ are isomorphic. It follows that $(M,\wom)$ has a compatible complex structure.

We turn to the ``only if'' assertion. 
Suppose that $(M,\omega,J)$ is K\"ahler and let $X=X(J)$ be the underlying spherical $K^{\CC}$-variety (see Theorem~\ref{thm:huck-wurz}).
Using the equivariant biholomorphism $(M,J) \to X$, we view $\wom$ as an equivariant K\"ahler form on $X$. Then $(\lat(X),\Pc(X,[\wom]),\Sigma(X))$ is a smooth $\RR$-momentum triple by Theorem~\ref{thm:existence-Kaehler-converse}. Since $\lat(X)=\lat(M)$ and $\Pc(X,[\wom]) = \Pc(M,[\wom])$, the assertion follows. 
\end{proof}

When $(M,\omega,J)$ is K\"ahler,  let $\Sigma(X(J))$ denote the set of
spherical roots of the projective spherical $G$-variety $X(J)$ given in Theorem~\ref{thm:huck-wurz}.

\begin{corollary} \label{cor:classif_of_complex}
Suppose $(M,\omega,\Phi)$ is a multiplicity free $K$-manifold with Kirwan polytope $\Pc$ and weight lattice $\lat$. The map $J \mapsto \Sigma(X(J))$ defines a one-to-one correspondence between the set of isomorphism classes of $K$-invariant complex structures $J$ on $M$ compatible with $\omega$ and the set $\{\Sigma \inn \Sigma(G) \colon \text{$(\lat,\Pc,\Sigma)$ is a smooth $\RR$-momentum triple}\}.$
\end{corollary}

\begin{proof}
The well-definedness and the surjectivity of the map follow from Theorem~\ref{thm:criterion-Kaehler}. The injectivity was proved by Losev in \cite[Theorem 8.3]{losev:knopconj}
\end{proof}

\begin{example}[Woodward] \label{ex:woodward}
Let $G=\GL(2)$. Denote the highest weight of $\bigwedge^i \CC^2$ by $\varpi_i$ for $i \in \{1,2\}$. 
Let $\lat=\wl$ and $\Pc= \Conv(0,\varpi_1,\varpi_1-\varpi_2,4\varpi_1-\varpi_2)$. The polytope $\Pc$ can be realized as the Kirwan polytope of a symplectic cut $X$ of a coadjoint orbit of $U(3)$ equipped with the natural Hamiltonian action of  $U(2)$ defined by the embedding of $U(2)$ in $U(3)$ given by $A\mapsto \mathrm{diag}(A,1)$. 
It is shown in~\cite{woodward-notkaehler} that $X$ equipped with the momentum map given by the projection of $X$ onto $\mathfrak u(2)^*$ is a 
non-K\"ahlerizable multiplicity free $U(2)$-manifold.

Note that $\lat(X)=\wl$. We can recover that $X$ is not K\"ahlerizable by showing that there is no smooth $\RR$-momentum triple $(\lat, \Pc,\Sigma)$ and applying Theorem~\ref{thm:criterion-Kaehler}.
For this, we first note that $\Sigma(G) = \{\alpha, 2\alpha\}$ with $\alpha$ being the simple root of $G$.
One checks that 
the empty set is the only subset of $\Sigma(G)$ that is $\RR$-admissible for $(\lat,\Pc)$ by simply checking the axioms of Definition~\ref{def:R-compatible}. Finally, we prove that the $\RR$-momentum triple $(\lat,\Pc,\emptyset)$ is not smooth: $0$ is an orbit vertex of $\Pc$ and  $\D(0) \cup \B(0)$ contains three elements, hence does not satisfy condition \ref{def:smooth_Q_momentum_triple:item1} of Definition~\ref{def:smooth_Q_momentum_triple} as required by Definition~\ref{def:smooth_R_momentum_triple}.
\end{example}

\subsection{K\texorpdfstring{\"a}{ae}hlerizability results of Delzant and Woodward}
\label{subsec:Kaehler_Delzant_Woodward}

In this section, we explain how earlier K\"ahlerizability results from \cite{delzant} and \cite{woodward} can be deduced from Theorem~\ref{thm:criterion-Kaehler}, and we slightly generalize some of these results from \cite{woodward}. 

We recall that under the mild additional assumption that the action of $K$ is effective (i.e., that  $\lat(M) = \wl$), manifolds $M$ as in the following theorem are known as symplectic toric manifolds. 

\begin{theorem}[{\cite[Appendice]{delzant}}] \label{thm:Delzant}
If $K$ is abelian, then every multiplicity free $K$-manifold $M$ admits a compatible complex structure, and any two such compatible complex structures are $K$-equivariantly biholomorphic.
\end{theorem}

\begin{proof} The uniqueness of the compatible complex structure follows immediately from  Corollary~\ref{cor:classif_of_complex} since $\Sigma(K^{\CC}) = \emptyset$. We show how to deduce the  existence of such a structure from Theorem \ref{thm:criterion-Kaehler}.  Let $\Pc$ be the Kirwan polytope and $\lat$ the weight lattice of $M$. We have to check that $(\lat,\Pc,\emptyset)$ is a smooth $\RR$-momentum triple. That it is an $\RR$-momentum triple is trivial. Delzant proved in \cite[Section 2]{delzant}, that the pair $(\lat,\Pc)$ satisfies the following property at every vertex $\msv$ of $Q$:
\begin{equation} \label{eq:DZ_condition}
\text{the collection $\{\rho_F \colon F \text{ is a facet of $\Pc$ containing $\mathsf{v}$}\}$ is a basis of $\Hom_{\ZZ}(\lat,\ZZ)$}.
\end{equation}
These are exactly the conditions \eqref{eq:local_toric_smoothness} and it follows by Remark~\ref{rem:SvinSKv}\ref{rem:SvinSkv:local_toric_smoothness} that $(\lat,\Pc,\emptyset)$ is a smooth $\RR$-momentum triple.
\end{proof}

We now turn to Woodward's results; we first recall the
definition of `reflective' polytopes. 

\begin{definition}[{\cite[Definition 1.1]{woodward-transversal}}]\label{def:reflective_polytope}
A convex polytope $\Pc$ in $\ft^*_+$ is called \textbf{reflective} if the following conditions are fulfilled:
\begin{enumerate}[(a)]
  \item $\Pc$ is of maximal dimension, i.e. $\dim \Pc = \rk K$;
  \item for all $a\in \Pc$, the set of hyperplanes generated by the facets of $\Pc$ containing $a$ is stable under the stabilizer of $a$ in the Weyl group $W$;
  \item any facet of $\Pc$ meets the relative interior of $\ft^*_+$. \label{def:reflective_polytope:c}
 \end{enumerate}
\end{definition}

For $\alpha \in \sr$, set $H_{\alpha} := \{\nu \in \ft^*\colon \<\alpha^{\vee},\nu\>=0\}$. The following lemma, due to Woodward, follows from Definition~\ref{def:reflective_polytope} with elementary arguments, and gives some combinatorial properties of reflective polytopes. 
We recall that a polytope is called {\bf simple} if the facets containing any given face have linearly independent facet normals.

\begin{lemma}\label{lem:woodward_refl}
Let $\Pc$ be a reflective polytope in $\ft^*_+$ and let $\alpha \in \sr$
be such that $\Pc \cap H_{\alpha} \neq \emptyset$.
\begin{enumerate}[(1)]
\item If $F$ is a facet of $\Pc$ meeting $H_{\alpha}$ then either $F$ is orthogonal to $H_{\alpha}$; or there exist a facet $\overline{F}\neq F$ of $\Pc$ and inward pointing facet normals $\rho$ and $\overline{\rho}$ to $F$ and $\overline{F}$, respectively, such that
\begin{align}
& \<\rho, \alpha\> = \<\overline{\rho}, \alpha\> = 1 \label{lem:woodward_refl:item1} \\
& \rho + \overline{\rho} = \alpha^{\vee} \text{ in } \Hom_{\ZZ}(\wl, \RR); \text{ and } \label{lem:woodward_refl:item2} \\ 
& F \cap \overline{F} = H_{\alpha} \cap \Pc. \label{lem:woodward_refl:item3}
\end{align} \label{lem:woodward_refl:part1}
\item There exist distinct facets $F$ and $\overline{F}$ of $\Pc$ with inward pointing facet normals $\rho$ and $\overline{\rho}$, respectively, such that
\eqref{lem:woodward_refl:item1}, \eqref{lem:woodward_refl:item2} and \eqref{lem:woodward_refl:item3} hold.  \label{lem:woodward_refl:part2}
\item If $\Pc$ is simple then it has exactly two facets $F,\overline F$ as in part~\ref{lem:woodward_refl:part2}, and they are the only facets of $\Pc$ that contain $\Pc\cap H_\alpha$.  \label{lem:woodward_refl:part3}
\end{enumerate}
\end{lemma}
\begin{proof}
Parts \ref{lem:woodward_refl:part1} and \ref{lem:woodward_refl:part2} are proved in \cite[Corollary 5.2]{woodward-transversal}. Part~\ref{lem:woodward_refl:part3} is in \cite[Proposition 4.7]{woodward-notkaehler}, let us provide a proof for convenience. We assume $\Pc$ is simple and let $F,\overline F$ be as in part~\ref{lem:woodward_refl:part2}, with facet normals $\rho,\overline \rho$.
Suppose, for the sake of contradiction, that there exists a facet $F_1$ of $\Pc$ containing $\Pc\cap H_\alpha$ and with $F_1 \notin \{F,\overline{F}\}$. Choose an inward pointing facet normal $\rho_1$ to $F_1$.
Since $\Pc$ is simple, $F\cap \overline{F}$ has codimension $2$ in $\Pc$.
We deduce that $V:=\ker(\rho)\cap\ker(\overline\rho)$ is the underlying vector subspace of the affine subspace spanned by $\Pc\cap H_\alpha$.  Therefore $\rho_1$ vanishes on $V$. But this implies that $\rho_1$ is a linear combination of $\rho$ and $\overline\rho$, which is in contradiction with the simplicity of $\Pc$.
\end{proof}

Lemma~\ref{lem:spherical_roots_reflective} below is a combinatorial version of \cite[Corollary 3.5]{woodward}. In its proof, we will make use of the following convex-geometric fact. Recall the notion of orbit face of a polytope from Definition~\ref{def:orbitface}.
 
\begin{lemma} \label{lem:orbitface_contains_orbitvertex}
If $\V$ in Definition~\ref{def:orbitface} is a full-dimensional polyhedral convex cone, any orbit face of $Q$ contains an orbit vertex.
\end{lemma}
\begin{proof}
For any $n\in N$, denote by $\C_n$ the element of $\F(Q)$ such that $n$ is in the relative interior of $\C_n$, and denote by $F_n$ the corresponding face of $Q$. Denote by $K_n\subseteq \F(Q)$ the subset of elements that contain $n$, and by $U_n$ the union of their relative interiors. This union is a neighborhood of $n$, since $N\setminus U_n$ is the union of all cones in $\F(Q)\setminus K_n$. For all $m\in U_n$ we have $\C_{m}\supseteq\C_n$, therefore $F_{m}\subseteq F_{n}$. Finally, denote by $\Omega$ the set of all $n\in N$ such that $F_{n}$ is a vertex of $Q$. It is dense in $N$.

Let $F$ be an orbit face of $Q$, and let $n\in\V$ be in the relative interior of $\C(F)$. Under our assumptions, $\V$ is the closure of its interior, therefore $U_n\cap\V\cap \Omega$ is non-empty. For any $m$ in this intersection, the face $F_m$ is an orbit vertex contained in $F_{n}=F$.
\end{proof} 
 
\begin{lemma} \label{lem:spherical_roots_reflective}
Let $\Pc$ be a simple reflective polytope in $\ft^*_+$ and $\lat$ a sublattice of $\wl$. Assume that $\Pc$ meets every face of codimension $1$ of $\ft^*_+$. If $\Sigma$ is a subset of $\Sigma(K^{\CC})$ such that $(\lat,\Pc,\Sigma)$ is a smooth $\RR$-momentum triple, then $\Sigma = \sr$. 
\end{lemma}
\begin{proof}
Observe that the assumption that $(\lat,\Pc,\Sigma)$ is an $\RR$-momentum triple includes that  $\rk \lat = \dim \Pc = \rk \wl$.
Recall that when $F$ is a face of $\Pc$, we put
\[
\C(F) = \text{dual cone in $\Hom_{\ZZ}(\lat,\QQ)$ to $\RR_{\ge 0}(\Pc-p) \inn \ft^* = \lat_{\RR}$},
\]
where $p$ is a point in the relative interior of $F$. 

We first show that $\Sigma \inn \sr$. Suppose, for the sake of contradiction, that $\sigma \in \Sigma \setminus \sr$. Then the set $\sr \cup \{\sigma\}$ is linearly dependent and lies in an open half-space, and so must contain two elements that form an acute angle. Since one of the two must be $\sigma$, this means there exists $\beta \in \sr$ such that $\<\beta^{\vee}, \sigma\> > 0$. 
Set $F = \Pc\cap H_{\beta}$, which is nonempty by assumption. Then $\beta^{\vee} \in \C(F)$ since $\Pc \inn \ft^*_+$. This means there is a ray generator $\rho$ of $\C(F)$ for which $\<\rho,\sigma\> > 0$. Because $\sigma \notin \sr$, it follows from condition (\ref{R-compatible-b}) of Definition~\ref{def:R-compatible} that the facet of $\Pc$ defined by $\rho$ is included in $H_{\alpha}$ for some $\alpha \in \sr$. This contradicts the fact that $\Pc$ is reflective (and   Definition~\ref{def:reflective_polytope}-\ref{def:reflective_polytope:c} in particular).

We now show the reverse inclusion, namely $\sr \inn \Sigma$. Let $\alpha \in \sr$ and put $F = \Pc \cap H_{\alpha}$. 
Suppose that $\alpha \notin \Sigma$. We claim that then $F$ is an orbit face of $\Pc$. Indeed, it follows from parts \ref{lem:woodward_refl:part1}, \ref{lem:woodward_refl:part2} and \ref{lem:woodward_refl:part3} of Lemma~\ref{lem:woodward_refl} that there are only two facets of $\Pc$ containing $F$. Also, we can choose corresponding inward pointing facet normals $\rho, \overline{\rho}$ which satisfy $\rho + \overline{\rho}= \alpha^{\vee}$.

This implies that $\alpha^{\vee}$ lies in the relative interior of $\C(F)$ since the extremal rays of $\C(F)$ are exactly the half lines spanned by the facet normals $\rho, \overline{\rho}$. On the other hand, $\<\alpha^{\vee}, \beta\> \leq 0$ for every $\beta \in \sr \setminus \{\alpha\}$. Since $\Sigma \inn \sr$, this proves the claim. 

To conclude the proof, we now show that $F$ being an orbit face contradicts the smoothness of the $\RR$-momentum triple $(\lat,\Pc,\Sigma)$ and in particular local factoriality (i.e.\ condition \ref{def:smooth_Q_momentum_triple:item1} of Definition~\ref{def:smooth_Q_momentum_triple}). By Lemma~\ref{lem:orbitface_contains_orbitvertex}, there exists an orbit vertex of $\Pc$ contained in $F$ (note that because $\Sigma \inn \sr$, it is a linearly independent subset of $\lat_{\QQ}$ and so $\V = (\QQp(-\Sigma))^{\vee} \inn  \Hom_{\ZZ}(\lat,\QQ)$ is full-dimensional\footnote{As recalled in Section~\ref{section:coloredfan}, it is a general fact, due to Brion \cite{brion90}, that the set of spherical roots of a spherical variety is linearly independent.}). Then $\C(F)$ is a face of $\C(\msv)$. Recall the notations $\rho_{\msv}$, $\D(\msv)$ and $\B(\msv)$ from Section~\ref{sec:smoothness}. Since $\Pc$ has maximal dimension, $\Spp(\Pc)= \emptyset$. Using that $\alpha^{\vee} \in  \C(\msv)$ and that $\alpha \notin \Sigma$, it follows from the combinatorial description of the colors associated to $(\lat,\Pc,\Sigma)$ recalled in Remark~\ref{rem:combinatorialversions} that there is a $D \in \D(\msv)$ such that $\rho(D) = \alpha^{\vee}|_{\lat}$. On the other hand, since the functionals $\rho, \overline{\rho}$ also lie on extremal rays of $\C(\msv)$, there exist $D_1$ and $D_2$ in $\D(\msv) \cup \B(\msv)$ such that $\rho_{\msv}(D), \rho_{\msv}(D_1)$ and $\rho_{\msv}(D_2)$ are linearly dependent. This is the promised contradiction and finishes the proof. 
\end{proof}

\begin{remark} \label{rem:complex_structure_refl}
It follows from Corollary~\ref{cor:classif_of_complex} and Lemma~\ref{lem:spherical_roots_reflective} that a multiplicity free $K$-manifold with a Kirwan polytope $\Pc$ that is reflective and simple and meets every wall of $\ft^*_+$ can have at most one compatible complex structure (up to equivariant biholomorphism).
\end{remark}

The following necessary condition for the K\"ahlerizability of certain multiplicity free manifolds is a generalization of Example~\ref{ex:woodward}. 

\begin{corollary}[{\cite[Theorem 6.2]{woodward}}] \label{cor:WoodwardThm62}
Let $M$ be a multiplicity free $K$-manifold with a simple reflective Kirwan polytope $\Pc$. Assume that $\Pc$  meets every wall of $\ft^*_+$. If $M$ admits a compatible complex structure, then for every facet $F$ of $\Pc$ (with inward pointing normal $\rho$) and every $\alpha \in \sr$ we have that 
\begin{equation}
\<\rho,\alpha\> >0 \text{ if and only if }F \supseteq \Pc \cap H_{\alpha}. 
\end{equation}
\end{corollary}
\begin{proof}
Observe that $\lat(M)$ has maximal rank since $\Pc$ has maximal dimension. Let $F$ be a facet of $\Pc$, with inward pointing normal $\rho$, and $\alpha \in \sr$. The fact that $\<\rho,\alpha\> >0$ when $\Pc \cap H_{\alpha} \inn F$ follows from parts \ref{lem:woodward_refl:part1} and \ref{lem:woodward_refl:part3} of Lemma~\ref{lem:woodward_refl}. 

Now suppose that $M$ admits a compatible complex structure. It follows from Theorem~\ref{thm:criterion-Kaehler} that there exists a subset $\Sigma \inn \Sigma(K^{\CC})$ such that $(\lat(M),\Pc,\Sigma)$ is a smooth momentum triple, and from Lemma~\ref{lem:spherical_roots_reflective} that $\alpha \in \Sigma$.  Assume that $\<\rho,\alpha\> > 0$. Then it follows from Definition~\ref{def:R-compatible}-(\ref{R-compatible-a}) that $F$ has to be one of the two facets of $\Pc$ that contain $\Pc \cap H_{\alpha}$ (see Lemma~\ref{lem:woodward_refl}-\ref{lem:woodward_refl:part3}.) This completes the proof. 
\end{proof}

\begin{remark}
In Theorem 6.2 of \cite{woodward}, $M$ is characterized geometrically as a ``transversal'' multiplicity free $K$-manifold with a discrete principal isotropy group, rather than combinatorially as having a simple reflective Kirwan polytope. Nevertheless, as stated in \cite[Theorem 4.6]{woodward} it follows from \cite{woodward-transversal} that these two geometric conditions imply that the Kirwan polytope of $M$ is reflective and simple.
\end{remark}

Our next aim is to  give a converse to Corollary~\ref{cor:WoodwardThm62} in Corollary~\ref{cor:converseWoodward}.
We first give a combinatorial result that we will need.

\begin{proposition} \label{prop:smoothreflective}
Let $\Pc$ be a simple reflective polytope in $\ft^*_+$ and $\lat$ a sublattice of $\wl$ such that $\dim \Pc = \rk \lat = \rk \wl$, that (Q2) holds and that $\Pc$ meets every face of codimension $1$ of $\ft^*_+$. Assume furthermore that the following hold:
\begin{enumerate}[(a)]
\item $\sr \inn \lat$. \label{prop:smoothreflective:rootsinlat}
\item For every $\alpha \in \sr$, the inward pointing facet normals $\rho,\overline{\rho}$ as in Lemma~\ref{lem:woodward_refl}-\ref{lem:woodward_refl:part2} belong to the lattice $\Hom_{\ZZ}(\lat,\ZZ)$. \label{prop:smoothreflective:normalinlat}
\item $\Pc$ satisfies condition \eqref{eq:DZ_condition} at every vertex $\msv$ that lies in the relative interior of $\ft^*_+$. \label{prop:smoothreflective:itemDelzant} 
\end{enumerate} 
 Then the following are equivalent:
\begin{enumerate}[(1)]
\item $(\lat,\Pc,\sr)$ is a smooth $\RR$-momentum triple. \label{prop:smoothreflective:part1}
\item For every facet $F$ of $\Pc$ and every $\alpha \in \sr$ we have that  \label{prop:smoothreflective:part2}
\begin{equation}
\<\rho_F,\alpha\> >0 \Leftrightarrow F \supseteq \Pc \cap H_{\alpha}. \label{eq:facetrefl}
\end{equation}
\end{enumerate}
\end{proposition}
\begin{proof}
We first show that \ref{prop:smoothreflective:part1} implies \ref{prop:smoothreflective:part2}. The implication ``$\Leftarrow$'' in \eqref{eq:facetrefl} follows from Lemma~\ref{lem:woodward_refl}-\ref{lem:woodward_refl:part3}. To show the implication ``$\Rightarrow$'' in \eqref{eq:facetrefl}, observe that since $\alpha$ is $\RR$-compatible with $(\lat,\Pc)$, there are at most two facets of $\Pc$ for which the inward pointing facet normal takes a positive value on $\alpha$, see Definition~\ref{def:R-compatible}-(\ref{R-compatible-a}). By Lemma~\ref{lem:woodward_refl}-\ref{lem:woodward_refl:part3} there are exactly two such facets, and they are those that contain $\Pc \cap H_{\alpha}$.

We now prove that \ref{prop:smoothreflective:part2} implies \ref{prop:smoothreflective:part1}. First note that (Q1) holds because the dimension of $\Pc$ and the rank of $\lat$ are maximal. We begin by verifying that every $\alpha \in \sr$ is $\RR$-compatible with $(\lat,\Pc)$. Note that from $\Spp(\Pc)=\Spp(\lat)=\emptyset$ and the fact that $\alpha$ is a primitive element of $\lat$ (by assumptions \ref{prop:smoothreflective:rootsinlat}  and \ref{prop:smoothreflective:normalinlat}), it follows that $\alpha$ is compatible with $\lat$ and that $(\Spp(\Pc),\alpha)$ satisfies Luna's axiom (S). To see that condition (\ref{R-compatible-a}) of Definition~\ref{def:R-compatible} holds, one observes that by our assumption \ref{prop:smoothreflective:normalinlat} and by \eqref{eq:facetrefl}, $\Pc$ has exactly two facets $F$ and $\overline{F}$ for which an inward pointing facet normal takes a positive value on $\alpha$, and that for these two facets we have
\begin{equation*}
\<\rho_F,\alpha\> = \<\rho_{\overline{F}},\alpha\>=1 \quad \text{and} \quad
H^{\RR}_{\overline{F}} = s_{\alpha}(H^{\RR}_F).
\end{equation*}
This also implies that 
\begin{equation} \label{eq:rhoAalpha}
\rho(\A(\alpha))=\{\rho_F,\rho_{\overline{F}}\}.
\end{equation}

To check that $\sr$ is $\RR$-admissible for $(\lat,\Pc)$, let $\beta \in \sr \setminus\{\alpha\}$ and $D \in \A(\alpha)$ with $\<\rho(D),\beta\> > 0$. Then it follows from \eqref{eq:rhoAalpha} that $\rho(D)=\rho_E$ for  some facet $E$ of $\Pc$, and from ``$\Rightarrow$'' in \eqref{eq:facetrefl} that $E \supseteq \Pc \cap H_{\beta}$. Then it follows from assumption \ref{prop:smoothreflective:normalinlat} that $\<\rho_E,\beta\>=1$ and from \eqref{eq:rhoAalpha} (with $\beta$ substituted for $\alpha$) that there exists $D' \in \A(\beta)$ such that $\rho(D') = \rho_E=\rho(D)$. We have shown that $(\lat,\Pc,\sr)$ is an $\RR$-momentum triple. 

What remains is to show that $(\lat,\Pc,\sr)$ is smooth. We first claim that if $\mathsf{w}$ is a vertex of $\Pc$ lying on a wall $H_{\alpha}$ of $\ft^*_+$ (where $\alpha \in \sr$), then $\mathsf{w}$ is not an orbit vertex. Indeed, it follows from Lemma~\ref{lem:woodward_refl}-\ref{lem:woodward_refl:part1} that if $E$ is a facet of $\Pc$ containing $\mathsf{w}$, then $\<\rho_E,\alpha\> \geq 0$. Since there is a facet $F$ of $\Pc$ containing $\mathsf{w}$ such that $\<\rho_F,\alpha\>>0$, the relative interior of 
\[\C(\mathsf{w}) = \QQp\{\rho_E\colon \text{$E$ is a facet of $\Pc$ containing $\mathsf{w}$}\}\]
does not meet the cone $\{\nu \in \Hom_{\ZZ}(\lat,\QQ) \colon \<\nu,\alpha\>\leq 0\}$ and therefore also not the valuation cone
$\{\nu \in \Hom_{\ZZ}(\lat,\QQ) \colon \<\nu,\beta\>\leq 0 \text{ for all } \beta \in \sr\}$. This proves the claim. 

The claim implies that every orbit vertex $\msv$ of $\Pc$  lies in the relative interior of $\ft^*_+$. The smoothness of the 
$\RR$-momentum triple $(\lat,\Pc,\sr)$ now follows from Remark~\ref{rem:SvinSKv}\ref{rem:SvinSkv:local_toric_smoothness},
since the condition \eqref{eq:local_toric_smoothness} is exactly our assumption \ref{prop:smoothreflective:itemDelzant}.
\end{proof}

Thanks to Proposition~\ref{prop:smoothreflective}, we can strengthen the necessary conditions for K\"ahlerizability in Corollary~\ref{cor:WoodwardThm62} to necessary and sufficient conditions, assuming that the weight lattice of the multiplicity free Hamiltonian manifold satisfies  \ref{prop:smoothreflective:rootsinlat}, \ref{prop:smoothreflective:normalinlat}
and \ref {prop:smoothreflective:itemDelzant} of Proposition~\ref{prop:smoothreflective}.

\begin{corollary} \label{cor:converseWoodward}
Let $M$ be a multiplicity free $K$-manifold with a simple reflective Kirwan polytope $\Pc$. 
Assume that $\Pc$  meets every wall of $\ft^*_+$. Assume furthermore that the following hold:
\begin{enumerate}[(a)]
\item $\sr \inn \lat(M)$. \label{cor:converseWoodward:rootsinlat}
\item For every $\alpha \in \sr$, the inward pointing facet normals $\rho,\overline{\rho}$ as in Lemma~\ref{lem:woodward_refl}-\ref{lem:woodward_refl:part2} belong to the lattice $\Hom_{\ZZ}(\lat(M),\ZZ)$. \label{cor:converseWoodward:normalinlat}
\item $\Pc$ is Delzant at every vertex $\msv$ that lies in the relative interior of $\ft^*_+$.\label{cor:converseWoodward:itemDelzant} 
\end{enumerate}
Then the following are equivalent:
\begin{enumerate}[(1)]
\item $M$ has a compatible complex structure. \label{cor:converseWoodward:part1}
\item For every facet $F$ of $\Pc$ and every $\alpha \in \sr$ we have that  \label{cor:converseWoodward:part2}
\begin{equation}
\<\rho_F,\alpha\> >0 \Leftrightarrow F \supseteq \Pc \cap H_{\alpha}. 
\end{equation}
\item $(\lat(M),\Pc,\sr)$ is a smooth $\RR$-momentum triple. \label{cor:converseWoodward:part3}
\end{enumerate} 
\end{corollary}
\begin{proof}
The equivalence of \ref{cor:converseWoodward:part2} and \ref{cor:converseWoodward:part3} is exactly Proposition~\ref{prop:smoothreflective}. The implication ``\ref{cor:converseWoodward:part1} $\Rightarrow$ \ref{cor:converseWoodward:part3}'' follows from Theorem~\ref{thm:criterion-Kaehler} and Lemma~\ref{lem:spherical_roots_reflective}, while the reverse implication follows from Theorem~\ref{thm:criterion-Kaehler}.
\end{proof}

\begin{remark}
Condition \ref{cor:converseWoodward:itemDelzant} in Corollary~\ref{cor:converseWoodward} is in fact superflous. Indeed,  \cite[Theorem 11.2]{knop:hamilton} implies that it is a consequence of the fact that $\Pc$ is the Kirwan polytope of the multiplicity free Hamiltonian manifold $M$. 
\end{remark}

\subsection{Existence of K\texorpdfstring{\"a}{ae}hler structures: necessity} \label{sec:kaehler-necessity} 

In this section we will prove the following theorem.

\begin{theorem} \label{thm:existence-Kaehler-converse}
Let $X$ be a smooth projective spherical $G$-variety and let $[\wom]\in H^2_K(X,\RR)$. 
Then $(\lat(X),\Pc(X,[\wom]),\Sigma(X))$ is a smooth $\RR$-momentum triple. 
\end{theorem}

Before proving Theorem~\ref{thm:existence-Kaehler-converse}, we prepare with some auxiliary results. 
Let $X$ and $\wom$ be as in the theorem, and $\Pc=\Pc(X,[\wom])$.  
For $D\in \divB(X)$, we set
\[
F_D=\Pc\cap \left(\chi_0([\widetilde\omega])+\{ \xi\in \lat(X)_\RR \;|\; \<\rho(D),\xi\> + \chi_D([\widetilde\omega]) =0\}\right).
\]

\begin{proposition}\label{proposition:face}
Let $X$, $\wom$ be as in Theorem~\ref{thm:existence-Kaehler-converse} and $\Pc = \Pc(X,[\wom])$, and $\sigma\in\Sigma(X)$. Then 
\begin{enumerate}
\item\label{lemma:face:exist} There exists a facet $F\subseteq \Pc$ and a color $D\in\Delta(X)$ such that $F=F_D$ and $\<\rho(D),\sigma\> >0$.
\item\label{lemma:face:atmost} Suppose $\sigma\in S\cap \Sigma(X)$, and that there exists a facet $E\subseteq \Pc$ such that $\<\rho_E,\sigma\> > 0$. Then $\<\rho_E,\sigma\> = 1$ and there exists a color $D$ of $X$ moved by $\sigma$ and such that $E=F_D$.
\end{enumerate} 
\end{proposition}
\begin{proof}
Let $\mathcal D$ be the subset of $\divB(X)$ such that $F_D$ is a facet of $\Pc$, and consider the ray $r=q + \RR_{\geq0}(-\sigma)$ for some point $q\in \Pc$. Since $q\in \Pc$ but $r$ is not contained in $\Pc$, there exists some $D\in \mathcal D$ such that $\<\rho(D),\sigma) > 0$. This implies that $D$ is not $G$-stable, because $\sigma$ is a spherical root, thus $D$ is a color: statement~(\ref{lemma:face:exist}) of the proposition follows by setting $F=F_D$.

Let now $\sigma$ and $E$ be as in~(\ref{lemma:face:atmost}). There exists $D\in\mathcal D$ such that $E=F_D$. It follows that $\rho_E$ and $\rho(D)$ are positive on the same half space of $\lat(X)_\RR$, in other words $\rho(D)$ is a positive rational multiple of $\rho_E$. This yields $\<\rho(D),\sigma\> >0$, which implies that $D$ is a color and that it is moved by $\sigma$. In particular $\<\rho(D),\sigma\>=1$, thus $\rho(D)$ is primitive in the dual lattice of $\lat$, whence $\rho(D)=\rho_E$.
\end{proof}

We are ready to give the proof of Theorem~\ref{thm:existence-Kaehler-converse}.

\begin{proof}[Proof of Theorem~\ref{thm:existence-Kaehler-converse}] \label{proof_of_thm:existence-Kaehler-converse}
Let $X$ and $\wom$ be as in the theorem, and let $\sigma\in\Sigma(X)$.
Then  $\sigma$ is compatible with $\lat(X)$ and since $\Spp(\Pc) = \Spp(X)$ by Lemma~\ref{lemma:Sp}, the couple $(\Spp(\Pc),\sigma)$ satisfies Luna's axiom (S). 
Conditions~(\ref{R-compatible-b}) and (\ref{R-compatible-a}) of Definition~\ref{def:R-compatible} hold by Proposition~\ref{proposition:face} and Corollary~\ref{cor:cancellation}, whereas condition~(\ref{R-compatible-d2}) of Definition~\ref{def:R-compatible} holds by Lemma~\ref{lemma:foschi}. The set $\Sigma(X)$ is  $\RR$-admissible, because condition~(\ref{A1polytope})  of Definition~\ref{def:R-admissibleset} is a well-known property of spherical varieties discovered by Luna \cite{luna:typeA}.

It remains to prove that 
the $\RR$-momentum triple $(\lat(X),\Pc,\Sigma(X))$ is smooth. For this, we observe that the (combinatorially defined) data involved in Definition~\ref{def:smooth_Q_momentum_triple} of a smooth $\RR$-momentum triple coincide with that of $X$ as in Definition~\ref{def:locsoc}, thanks to Proposition~\ref{thm:Brion-Woodward-gen} and the arguments in Remark~\ref{rem:combinatorialversions}. We conclude that the triple is smooth, because $X$ is smooth, and the conditions in Definition~\ref{def:smooth_Q_momentum_triple} are exactly the conditions in Camus' smoothness criterion.
\end{proof}

\subsection{Existence of K\texorpdfstring{\"a}{ae}hler structures: sufficiency}  \label{sec:kaehler_sufficient}

In this section we prove the converse to Theorem~\ref{thm:existence-Kaehler-converse}.

\begin{theorem}\label{thm:existence-Kaehler}
Let $\lat\subseteq\wl$ be a sublattice, $\Sigma\subseteq\Sigma(G)$ and $\Pc\subseteq\RR_{\geq 0}\dw$ be a convex polytope satisfying properties $(Q1)$ and $(Q2)$.
Let $(\lat,\Pc,\Sigma)$ be a smooth $\RR$-momentum triple.
Then there exists a smooth projective spherical $G$-variety $X$ with weight lattice $\lat$, set of spherical roots $\Sigma$, and an equivariant K\"ahler form $\wom$ on the compact real $K$-manifold $X$ such that the Kirwan polytope $\Pc(X,[\wom])$ is the polytope $\Pc$.
\end{theorem}

After some preparations, we give the proof of Theorem~\ref{thm:existence-Kaehler} after Lemma~\ref{lemma:delta}.  We first construct a spherical variety for any prescribed $\RR$-momentum triple by a different method as the one pursued in Section~\ref{sec:characterizations_mp}. The different steps of this construction are traceable in the following lemmas.

\begin{lemma}\label{lemma:face-has-divisor}
Suppose $(\lat,\Pc,\Sigma)$ is an $\RR$-momentum triple. 
Let $F$ be a facet of $\Pc$. Then at least one of the following statements holds.
\begin{enumerate}
\item\label{lemma:face-has-divisor:A} There exist $\alpha\in S\cap \Sigma$
and $D\in\A(\alpha)$ such that $\rho_F=\rho(D)$.
\item\label{lemma:face-has-divisor:b} 
There exists $\alpha\in S\smallsetminus (\Sigma\cup \Spp(\Pc))$ such that $\<\alpha^\vee,F\>=0$.
\item\label{lemma:face-has-divisor:V} The inequality $\<\rho_F,\sigma\> \leq0$ holds for all $\sigma\in\Sigma$.
\end{enumerate}
\end{lemma}

\begin{proof} 
Suppose that statement~(\ref{lemma:face-has-divisor:V}) fails. Then there exists $\sigma\in\Sigma$ such that $\<\rho_F,\sigma\> >0$.

If $\sigma\in S$, then by $\RR$-compatibility of $\sigma$ there exists a facet $E\subseteq \Pc$ such that $\<\rho_E,\sigma\>=1$ and $\rho_E=\rho(D_\sigma^+)$. Moreover $H^{\RR}_{F}=H_{E}$ or $H^{\RR}_{E}=s_\sigma(H^{\RR}_{F})$, since $\<\rho_F,\sigma\>\geq0$. By definition, the set $\A(\sigma)$ contains two elements $D^+_\sigma$ and $D^-_\sigma$ with $\rho(D^+_\sigma)=\rho_E$ and $\rho(D^-_\sigma)=\sigma^\vee|_\lat - \rho_E$. If now $F=E$ then $\rho_F=\rho(D_\sigma^+)$. If $H_{E}=s_\sigma(H_{F})$ then, reasoning as in the proof of Lemma~\ref{lemma:sum}, we obtain $\rho_F=\rho(D_\sigma^-)$, and in both cases statement~(\ref{lemma:face-has-divisor:A}) holds.

If $\sigma\notin S$, then by $\RR$-compatibility of $\sigma$ there exists $\alpha\in S\smallsetminus \Spp(\Pc)$ such that $\<\alpha^\vee,F\> =0$, which is statement~(\ref{lemma:face-has-divisor:b}).
\end{proof}

\begin{remark}
In general, the three statements of the above lemma are not mutually exclusive. For example, assume that statement~(\ref{lemma:face-has-divisor:A}) holds for some facet $F\subseteq \Pc$. 
Then condition~(\ref{lemma:face-has-divisor:V}) fails, but condition~(\ref{lemma:face-has-divisor:b}) may hold when $\rho(E)=\frac12\alpha^\vee|_\lat$ for both $E\in\A(\alpha)$.
\end{remark}

\begin{lemma}\label{lemma:homogeneous}
Suppose $(\lat,\Pc,\Sigma)$ is an $\RR$-momentum triple.  
Then there exists a unique spherical homogeneous space $G/H$ with weight lattice $\lat$, set of spherical roots $\Sigma$ and such that $\Spp(G/H)=\Spp(\Pc)$ and for all $\alpha\in S\cap \Sigma$ the set $\A(G/H,\alpha)$  can be identified with $\A(\alpha)$ of Definition~\ref{def:R-compatible} in such a way that $\rho_{G/H}|_{\A(G/H,\alpha)}$ coincides with $\rho$ of Definition~\ref{def:R-compatible}.
\end{lemma}

\begin{proof}
We use the classification of spherical homogeneous spaces \cite{luna:typeA, losev:uniqueness, bp, cf} via homogeneous spherical data, which were introduced by Luna in \cite[\S 2.2]{luna:typeA} (see \cite[30.21]{timashev} for a compact definition).  
From the sets $\A(\alpha)$, where $\alpha \in \Sigma \cap \sr$, and the maps $\rho \colon \A(\alpha) \to \Hom_{\ZZ}(\lat,\ZZ)$ in Definition~\ref{def:R-compatible}, equation~\eqref{eq:AfromAalpha} defines a set $\A$ and a map $\rho\colon \A \to \Hom_{\ZZ}(\lat,\ZZ)$. 
It is somewhat lengthy but straightforward to verify that the axioms of a spherical homogeneous datum hold for $(\Spp(\Pc),\Sigma,\A,\lat,\rho)$.	
\end{proof}

Let us fix a point $\mathsf{w}\in \Pc$. Then the numbers $m_{F,\mathsf{w}}$ are defined by \eqref{eq:Pc_and_mFw} for all facets $F$ of $\Pc$,  as is the support function $l_{\Pc-\mathsf{w}}$ of $\Pc-\mathsf{w}$:
\[l_{\Pc-\mathsf{w}}\colon N \to \RR, x\mapsto \max\{\<x,-q\> : q\in \Pc-\mathsf{w}\},\]
where $N=\Hom_{\ZZ}(\lat,Q)$.

\begin{lemma}\label{lemma:nD}
Suppose $\Sigma\subseteq \Sigma_\RR(\lat,\Pc)$ is $\RR$-admissible, and let $G/H$ be as in Lemma~\ref{lemma:homogeneous}. For all colors $D$ of $G/H$ we define a real number $n_D$ as follows.

\begin{enumerate}
\item 
For $\alpha\in S\cap \Sigma$, we choose a facet $F_\alpha^+$ of $\Pc$ such that $\rho_{F_\alpha^+}=\rho(D^+_\alpha)$. We set $n_{D_\alpha^+}=m_{F_\alpha^+,\mathsf{w}}$. 
For the other element $D_\alpha^-\in \A(\alpha)$ we set $n_{D_\alpha^-}=\<\alpha^\vee,\mathsf{w}\> - n_{D_\alpha^+}$.

\item 
For $D\in \Delta(G/H)\smallsetminus\A(G/H)$, let $\alpha$ be a simple root that moves $D$.
If $2\alpha\notin\Sigma$ we set $n_D=\<\alpha^\vee,\mathsf{w}\>$, otherwise we set $n_D=\frac12\<\alpha^\vee,\mathsf{w}\>$.
\end{enumerate}
Then, for all $q\in \Pc$ and all $D\in\Delta(G/H)$, we have the following inequality
\begin{equation}\label{eq:contained}
\<\rho(D), q-\mathsf{w}\> +n_{D} \geq 0.
\end{equation}

\end{lemma}

\begin{proof}
Let $D\in\Delta(G/H)$ and $\alpha\in S$ that moves $D$. 
Suppose $\alpha\in \Sigma$, then $D=D_\alpha^\pm$. But $\<\rho(D_\alpha^+), q-\mathsf{w}\> +n_{D_\alpha^+} \geq 0$ 
because $n_{D_\alpha^+}=m_{F,\mathsf{w}}$ by definition, and $\<\rho(D_\alpha^-), q-\mathsf{w}\> +n_{D_\alpha^-} \geq 0$ thanks to the argument given for Lemma~\ref{lemma:sum}-(\ref{lemma:sum:general}).
	
Let now $D\in \Delta(G/H)\smallsetminus\A(G/H)$, then $\rho(D)=\varepsilon\alpha^\vee|_\lat$ for some
$\varepsilon\in\{1,\frac12\}$. Then we have
\[
\<\rho(D), q-\mathsf{w}\> +n_{D} = \<\varepsilon \alpha^\vee,q-\mathsf{w}\> +\<\varepsilon\alpha^\vee,\mathsf{w}\> \geq 0.
\]
\end{proof}

\begin{lemma}\label{lemma:fan}
Under the assumptions of Lemma~\ref{lemma:nD}, let $\F_\Sigma(\Pc)$ denote the set of couples $(\C(F),\D(F))$ where $F$ varies over the set of orbit faces of $\Pc$ and for each such face $F$ we set $D\in \D(F)$ if and only if $\<\rho(D),F-\mathsf{w}\> + n_D = 0$. Then $\F_\Sigma(\Pc)$ is a colored fan, and the corresponding embedding of $G/H$ is complete.
\end{lemma}

\begin{proof}
For all $(\C(F),\D(F))\in\F_\Sigma(\Pc)$, property~(CC2) is clear. For property~(SCC), we observe that $\C(F)$ is strictly convex because $\Pc$ has dimension equal to the rank of $\lat$. 
We must also show that $0\notin\rho(\D(F))$, so assume $\rho(D)=0$ for some $D\in \Delta(G/H)$. Then $D\notin\A(G/H)$, and $\alpha^\vee|_\lat=0$, where $\alpha\in S\smallsetminus \Spp(\Pc)$ moves $D$, i.e.\ $\<\alpha^\vee,\Pc\> = \<\alpha^\vee,\mathsf{w}\>$. 
But $\alpha\notin \Spp(\Pc)$, so $\<\alpha^\vee,\mathsf{w}\> >0$, 
therefore there does not exist any facet $F'$ of $\Pc$ such that $\<\alpha^\vee,F'\>=0$. It follows that $\<\rho(D), \Pc-\mathsf{w}\> + n_D \subseteq \RR_{>0}$ by the definition of $n_D$, hence $D\notin\D(F)$.
	
We check property~(CC1). 
Recall that $\C(F)$ is generated by $\rho_{F'}$, for all facets $F'$ of $\Pc$ containing $F$.
By Lemma~\ref{lemma:face-has-divisor}, any such $\rho_{F'}$ is in $\V(G/H)$, or it is equal to $\rho(D)$ for some $D\in \A(G/H)$, or (assuming the two previous conditions fail) a positive multiple of $\alpha^\vee|_{\lat_\RR}$ for some $\alpha\in S\smallsetminus (\Sigma\cup \Spp(\Pc))$ 
such that $\<\alpha^\vee,F'\>=0$. 
In the third case, the color $D\in \Delta(G/H)\smallsetminus \A(G/H)$ moved by $\alpha$ satisfies
\begin{equation}\label{eq:DinD}
\<\rho(D),F'-\mathsf{w}\>+n_D=0
\end{equation}
by the definition of $n_D$, so $D\in \D(F)$. In the second case, suppose $D=D_\alpha^+$ for $\alpha\in \Sigma\cap S$, or 
there exist two distinct facets of $\Pc$ whose normals against $\alpha$ is positive.
Then $n_D=l_{\Pc-\mathsf{w}}(\rho(D))$ and the equality~(\ref{eq:DinD}) holds again, hence $D\in\D(F)$. If $D=D_\alpha^-$ and 
there exists a unique facet of $\Pc$ whose normal against $\alpha$ is positive,
 we have $\rho(D)=\rho(D_\alpha^-)=\rho(D_\alpha^+)$,
because anyway $\rho(D) = \rho_{F'}$ is positive on $\alpha$. 
In this case it is enough to set $D=D_\alpha^+$ instead of $D_\alpha^-$ to obtain once again $n_D=l_{\Pc-\mathsf{w}}(\rho(D))$ and~(\ref{eq:DinD}), yielding $D\in \D(F)$. 
Property~(CC1) for $(\C(F),\D(F))$ follows. Properties~(CF1) and~(CF2) are obvious, hence $\F_\Sigma(\Pc)$ is a colored fan.
	
Thanks to the definition of $\F_\Sigma(\Pc)$, the union of all cones in the fan contains $\V(G/H)$, hence the embedding of $G/H$ corresponding to $\F_\Sigma(\Pc)$ is complete.
\end{proof}

As announced, the above lemmas enable us to associate to any $\RR$-momentum triple $(\lat,\Pc,\Sigma)$ a spherical $G$-variety, denoted by $X(\lat,\Pc,\Sigma)$ below: it is the embedding determined by the colored fan in Lemma~\ref{lemma:fan} of the homogeneous space $G/H$ as in Lemma~\ref{lemma:homogeneous}.

\begin{proposition} \label{prop:XLPS_smooth}
Suppose $(\lat,\Pc,\Sigma)$ is a smooth $\RR$-momentum triple.
Then the variety $X(\lat,\Pc,\Sigma)$ is smooth and projective.
\end{proposition}

\begin{proof}
We first prove that $X(\Pc,\lat,\Sigma)$ is smooth. The colored fan of $X(\Pc,\lat,\Sigma)$ is given by Lemma~\ref{lemma:fan} and the conditions in Definition~\ref{def:smooth_R_momentum_triple} of a smooth $\RR$-momentum triple are exactly the conditions in Camus' smoothness criterion \cite[\S 6.3]{camus} (see also \cite[Theorem 3.16]{PVS}) for the spherical embedding $X(\Pc,\lat,\Sigma)$ to be smooth (along each of its closed orbits).  

The support of the colored fan of $X$ consisting of some of the maximal cones of the normal fan of $\Pc$, it is straightforward to  exhibit, using \cite[Theorem 3.3]{brion:picard}, a Cartier divisor (i.e.\ a Weil divisor, since $X$ is smooth) which is ample. Consequently, $X$ is projective.
\end{proof}

Let $n_D$ be the real numbers given in Lemma~\ref{lemma:nD}. 
For all $G$-stable prime divisors $Y$ of $X(\lat,\Pc,\Sigma)$, let $n_Y = m_{F,\mathsf{w}}$, where $F$ is the facet of $\Pc$ such that $\QQ_{\geq0}\rho_F$ is the cone of the colored fan of $X(\lat,\Pc,\Sigma)$ corresponding to $Y$.
Set
\begin{equation}\label{eq:delta}
\delta =\delta(\lat,\Pc,\Sigma)= \sum_{D\in\Delta(X)} n_D D + \sum_Y n_Y Y,
\end{equation}
where $Y$ varies in the set of $G$-stable prime divisors of $X(\lat,\Pc,\Sigma)$.

Let $[\delta]$ be an element of $\Pic_G(X(\lat,\Pc,\Sigma))_\RR$ mapped to the class of $\delta$ in $\Pic(X(\lat,\Pc,\Sigma))_\RR$, and recall Lemma~\ref{lemma:Losev-cohomology}.

\begin{lemma}\label{lemma:delta}
Regarded as a $K$-invariant real $2$-class on $X(\lat,\Pc,\Sigma)$, $[\delta]$ is K\"ahler.
\end{lemma}

\begin{proof}
Let $l_{\Pc-\mathsf{w}}$ be the support function of the polytope $\Pc-\mathsf{w}$.
Then $l_{\Pc-\mathsf{w}}$ is linear on each colored cone of $X(\lat,\Pc,\Sigma)$ and obviously satisfies the ampleness conditions of \cite[Theorem 3.3]{brion:picard}.
These conditions being open conditions for piecewise linear functions on $N_\RR$ which are linear on each colored cone of $X(\lat,\Pc,\Sigma)$, there exist finitely many rational strictly convex piecewise linear functions $l_i$, linear on each colored cone of $X(\lat,\Pc,\Sigma)$, such that $l_{\Pc-\mathsf{w}}$ belongs to their convex hull.
By \cite[Theorem 3.3]{brion:picard}, every function $l_i$ defines an ample $\QQ$-Cartier divisor on $X(\lat,\Pc,\Sigma)$ hence a K\"ahler metric on the real manifold $X(\lat,\Pc,\Sigma)$. By the convexity of the K\"ahler cone, it follows that $[\delta]$ itself is a K\"ahler class on the real manifold $X(\lat,\Pc,\Sigma)$.
\end{proof}

\begin{proof}[Proof of Theorem~\ref{thm:existence-Kaehler}] \label{proof_thm_existence-Kaehler}
Take $(\lat,\Pc,\Sigma)$ and let $X=X(\lat,\Pc,\Sigma)$ and $\delta=\delta(\lat,\Pc,\Sigma)$ be as above.
Thanks to the previous lemmas, we are left with showing that $\Pc(X,[\delta])=\Pc$, for some choice of the class $[\delta]\in \Pic_G(X)_\RR$.

For all $D\in\divB(X)$, choose an element $[D]\in \Pic_G(X)_\RR$ mapped to the class of $D$ in $\Pic(X)_\RR$, in such a way that~(\ref{eq:delta}) holds in $\Pic(X(\lat,\Pc,\Sigma))_\RR$.

Recall that the definition of $\chi_0$ and $\chi_D$ involved choosing a lift of the map $\overline{\chi}$. By linearity, as in the proof of Lemma~\ref{lemma:foschi}, we may assume that this lift is given by a rational section of the line bundle. Moreover, one can make the relationship between section and line bundle explicit, by fixing a non-zero rational function $f\in\CC(X)$, and defining
\[
\begin{array}{lcl}
\chi_0([E]) &=& \chi(s(f,E)), \\
\chi_D([E]) &=& v_D(s(f,E)) \quad \forall E\in \divB(X)
\end{array}
\]
where $s(f,E)$ is the rational section of $[E]$ corresponding to $f$ given by the choice of $E$ as a representative of $[E]$. Let us choose the constant function $f=1$, then $s(f,E)$ is the canonical section, yielding $\chi_D([E])=1$ if $D=E$ and $\chi_D([E])=0$ otherwise.

Putting this together with Lemma~\ref{lemma:foschi}, we obtain
\[
\varepsilon_\alpha\<\alpha^\vee, \chi_0([\delta])\>=\sum_{D} \<\alpha^\vee,\chi_D([\delta])\> = \sum_D n_D
\]
where the sum is over all colors $D$ moved by $\alpha$. The definition of $n_D$ then gives
\begin{equation}\label{eq:chi}
\<\alpha^\vee, \chi_0([\delta])\> =\<\alpha^\vee,\mathsf w\>
\end{equation}
for all $\alpha\in S$. Recall now the exact sequence
\[
0\to (\fk/[\fk,\fk])^* \to\Pic_G (X)_\RR \to \Pic(X)_\RR\to 0
\]
from~\cite[Lemma~8.8]{losev:knopconj}. Together with~(\ref{eq:chi}), it implies that we can choose $[\delta]\in\Pic_G (X)_\RR$ in such a way that $\chi_0([\delta])=\mathsf w$.

Now, by Proposition~\ref{prop:Kaehler-polytope}, the polytope $P(X,[\delta])-\chi_0([\delta])$ is defined by the inequalities
\begin{equation}\label{eq:complete}
\<\rho(Z),\xi\> + n_Z \geq 0
\end{equation}
for $\xi\in \lat_\RR$, with $Z$ varying in the set of $B$-stable prime divisors of $X$. These inequalities hold for all $\xi$ of the form $\xi=q-\mathsf w$ with $q\in \Pc$, 
thanks to formula~(\ref{eq:contained}) of Lemma~\ref{lemma:nD} if $Z$ is a color, and thanks to the definition of $n_Z$ if $Z$ is a $G$-stable prime divisor. It follows that $\Pc(X,[\delta])\supseteq \Pc$.
On the other hand, the set of inequalities of the form~(\ref{eq:complete}) contains a set of inequalities that define $\Pc$, thanks to Lemma~\ref{lemma:face-has-divisor}, the definition of $\F_\Sigma(Q)$ and the numbers $n_Z$. 
The inclusion $\Pc(X,[\delta])\subseteq \Pc$ follows, hence $\Pc(X,[\delta])= \Pc$.
\end{proof}

\subsection*{Acknowledgements} The authors thank the referee for several useful suggestions, which improved the paper.
S.~C.-F. is supported by the SFB/TRR 191 ``Symplectic Structures in Geometry, Algebra and Dynamics.''  B.~V.~S.\ received support from the (US) N.S.F. through grant DMS 1407394 and from the City University of New York PSC-CUNY Research Award Program. He thanks Friedrich Knop and the Department of Mathematics at the FAU for hosting him in 2016-17 and Medgar Evers College for his 2016-17 Fellowship Award.

\end{document}